\documentclass[11pt]{article}
\usepackage[margin=1in]{geometry}
\usepackage{amsmath}
\usepackage{amsfonts}
\usepackage{xcolor}
\usepackage{latexsym}
\usepackage{enumitem}
\usepackage{tablefootnote}

\usepackage{todonotes}
\usepackage{tablefootnote}
\usepackage{epsfig}
\usepackage{multirow}
\usepackage{float}
\usepackage{array}
\usepackage{mathtools}
\usepackage{hyperref}
\usepackage{comment}
\usepackage{amssymb,amsthm}
\usepackage{bm}
\usepackage{algorithm,algorithmic}
\usepackage[labelformat=simple]{subcaption}

\mathtoolsset{showonlyrefs=false}

\newcommand*{\QEDA}{\hfill\hbox{\vrule width1.0ex height1.0ex}}

\newcommand{\prox}{\operatorname{prox}}
\newcommand{\xt}{\tilde x}

\newcommand{\lora}{\mathrm{L}}
\newcommand{\flora}{\mathrm{F}}
\renewcommand{\L}{\mathcal{L}}
\newcommand{\bardtL}{\hat{R}_{\tilde\Lambda}}
\newcommand{\Lp}{\L_\rho}
\newcommand{\tnu}{\tilde{\nu}}

\newcommand{\dtL}{R_{\tilde\Lambda}}
\newcommand{\tL}{\tilde{\L}}

\newcommand{\Lpt}{\tilde{\Lp}}
\newcommand{\CConv}[2]{\underset{#1}{\overline{\mathrm{Conv}}}(#2)}
\newcommand{\td}{\tilde{d}}
\newcommand{\tO}{\tilde{\mathcal{O}}}
\usepackage{makecell,booktabs}
\usepackage[version=4]{mhchem}
\usepackage[strict]{changepage}

\usepackage{mdframed}
\usepackage[most]{tcolorbox}

\newtcolorbox{eqbox}{
  colback=gray!5!white,
  colframe=gray!80!black,
  boxrule=0.5pt,
  arc=4pt,
  left=6pt,
  right=6pt,
  top=4pt,
  bottom=4pt,
  enhanced
}

\usepackage{amsmath}
\usepackage{amsfonts}
\usepackage{latexsym}
\usepackage{amssymb}

\newtheorem{thm}{Theorem}[section]
\newtheorem{theorem}[thm]{Theorem}

\newtheorem{lemma}[thm]{Lemma}

\newtheorem{proposition}[thm]{Proposition}

\newtheorem{corollary}[thm]{Corollary}

\theoremstyle{definition}

\newtheorem{example}[thm]{Example}

\newtheorem{remark}[thm]{Remark}

\newcommand{\beq}{\begin{equation}}
\newcommand{\eeq}{\end{equation}}
\newcommand{\beqa}{\begin{eqnarray}}
\newcommand{\eeqa}{\end{eqnarray}}
\newcommand{\beqas}{\begin{eqnarray*}}
\newcommand{\eeqas}{\end{eqnarray*}}
\newcommand{\bi}{\begin{itemize}}
\newcommand{\ei}{\end{itemize}}

\newcommand{\nn}{\nonumber}

\newcommand{\R}{\mathbb{R}}

\newcommand{\lam}{{\lambda}}

\newcommand{\A}{{\cal A}}

\newcommand{\inner}[2]{\langle #1,#2\rangle}

\newcommand{\argmin}{\mathrm{argmin}\,}
\newcommand{\argmax}{\mathrm{argmax}\,}

\newcommand{\dom}{\mathrm{dom}\,}

\newcommand{\Argmin}{\mathrm{Argmin}\,}

\newcommand{\tx}{\tilde x}
\newcommand{\ty}{\tilde y}

\newtheorem{assumption}{Assumption}
\begin{document}
	\title{Improved Analysis of Restarted Accelerated Gradient and Augmented Lagrangian Methods via Inexact Proximal Point Frameworks}
	\date{February 19, 2026}
	\author{Matthew X. Burns\thanks{
        Department of Electrical and Computer Engineering, University of Rochester, Rochester, NY 14627 (email: {\tt mburns13@ur.rochester.edu}).}
  \and
		Jiaming Liang \thanks{Goergen Institute for Data Science and Artificial Intelligence (GIDS-AI) and Department of Computer Science, University of Rochester, Rochester, NY 14620 (email: {\tt jiaming.liang@rochester.edu}). This work was partially supported by AFOSR grant FA9550-25-1-0182.
		}
		 }
	\maketitle

	\begin{abstract}
    This paper studies a class of double-loop (inner-outer) algorithms for convex composite optimization. For unconstrained problems, we develop a restarted accelerated composite gradient method that attains the optimal first-order complexity in both the convex and strongly convex settings. For linearly constrained problems, we introduce inexact augmented Lagrangian methods, including a basic method and an outer-accelerated variant, and establish near-optimal first-order complexity for both methods. The established complexity bounds follow from a unified analysis based on new inexact proximal point frameworks that accommodate relative and absolute inexactness, acceleration, and strongly convex objectives. Numerical experiments on LASSO and linearly constrained quadratic programs demonstrate the practical efficiency of the proposed methods. \\
    
		\\
		
	\end{abstract}
	
	\section{Introduction}\label{sec:intro}

    In this paper, we consider two optimization problems: the convex smooth composite optimization (CSCO) problem
    \begin{equation}\label{eq:ProbIntro}
        \phi_*:=\min_{x\in\R^n}\{\phi(x):=f(x)+h(x)\},
    \end{equation}
    and the linearly constrained CSCO (LC-CSCO) problem
    \begin{equation}\label{eq:ProbIntro_LC}
        \hat\phi_*:=\min_{x\in\R^n}\{\phi(x):=f(x)+h(x):Ax=b\},
    \end{equation}
    where $A\in\R^{m\times n}$ and $b\in \R^m$ define $m\leq n$ linear equality constraints.
    In both problems, we assume that i) $f,h:\R^n\to\R\cup\{+\infty\}$ are closed proper convex functions such that $\dom h \subset \dom f$, ii) $f$ is $L_f$-smooth on $\R^n$, and iii) $h$ has a computable proximal mapping.
    Moreover, we also assume for LC-CSCO that $\dom h$ is bounded with diameter $D$ and Slater's condition is satisfied.

    For any $\varepsilon>0$, we say a point $x\in\R^n$ is an $\varepsilon$-solution to~\eqref{eq:ProbIntro} if $\phi(x)-\phi_*\leq \varepsilon$. 
    To define an optimality criterion for~\eqref{eq:ProbIntro_LC}, we consider the unconstrained primal-dual reformulation,
    \[
    \max_{\lambda\in\R^m}\min_{x\in\R^n}\{\L(x,\lambda):=\phi(x)+\inner{\lambda}{Ax-b}\},
    \]
    where $\L$ is the Lagrangian function and $\lambda\in\R^m$ is the Lagrange multiplier for the constraint $Ax=b$. Slater's condition implies the strong duality equivalence
    \begin{equation}
    \hat\phi_* = \max_{\lambda\in\R^m}\bigl\{d(\lambda):=\min_{x\in\R^n} \L(x,\lambda)\bigr\},\label{eq:strong_duality}
    \end{equation}
    where $d(\lambda)$ is the Lagrangian dual of~\eqref{eq:ProbIntro_LC}. 
    For fixed $\varepsilon > 0$, we call the pair $(x,\lambda)$ an $\varepsilon$-primal-dual solution to~\eqref{eq:ProbIntro_LC} if it is an $\varepsilon$-stationary point of $\L$, that is for some $v\in\R^n$
    \begin{equation}
        v\in \partial h(x) +\nabla f(x)+A^\top\lambda,\quad
        \|v\|\leq\varepsilon,\quad\|Ax-b\|\leq \varepsilon.
    \label{def:approximate_kkt}
    \end{equation} 
    Additionally, we call a point $x\in\dom h$ an $\varepsilon$-primal solution if
    \begin{equation}\label{def:primal_lc_sol}
        |\phi(x)-\hat\phi_*|\leq \varepsilon,\quad \|Ax-b\|\leq \varepsilon,
    \end{equation}
    We can show (see Lemma~\ref{lem:pd_gap} in Appendix \ref{appdx:technical}) that~\eqref{def:approximate_kkt} implies an $\mathcal{O}(\varepsilon)$ primal solution, hence we focus on~\eqref{def:approximate_kkt} in this work for generality.

    \noindent
    \textbf{Literature Review.}
    Nesterov's accelerated composite gradient (ACG) method is standard for solving CSCO problems. First proposed for purely smooth problems ($h=0$)~\cite{ag_nesterov83}, accelerated methods have since been extended to the composite setting~\cite{auslender2006interior,beck2009fast,monteiro2016adaptive,nesterov2013gradient}, where they achieve optimal complexity $\mathcal{O}(\varepsilon^{-1/2})$ for obtaining an $\varepsilon$-solution to~\eqref{eq:ProbIntro}. However, some undesirable phenomena are observed in practice, namely oscillations in the objective value. ``Restarted'' ACG methods are a widely used strategy to improve ACG performance and suppress oscillations. A restarted ACG method periodically resets the acceleration scheme according to some predefined rule~\cite{alamo2019gradient,alamo2022restart,odonoghueAdaptiveRestartAccelerated2015,su2016differential}. The seminal work~\cite{odonoghueAdaptiveRestartAccelerated2015} proposed ``gradient''  and ``function value'' restart heuristics which restart when the gradient forms an acute angle with the update direction (``gradient'') or when the function value increases (``function value''). While these restart criteria are empirically performant, they were initially heuristic strategies without theoretical support. Recent work has shown that gradient restart achieves optimal rates in the strongly convex setting~\cite{baoAcceleratedGradientMethods2025}, however the authors do not consider the more general class of merely convex objectives. A ``speed restart'' strategy was further proposed by~\cite{su2016differential}. Motivated by continuous-time ODE analysis, discrete-time speed restart resets acceleration when $\|x_{k}-x_{k-1}\|< \|x_{k-1}-x_{k-2}\|$. While a convergence analysis was presented for the continuous-time limit, the final bound contains some constants which are simply shown to exist, lacking exact characterization. Parameter-free restarting schemes for strongly convex optimization were proposed by~\cite{sujanani2025efficient} based on an estimation procedure for the (unknown) strong convexity modulus.

    A classical method for solving LC-CSCO problems is the augmented Lagrangian method (ALM), also known as the method of multipliers, which has the iteration
    \begin{align}
        x_{k+1}&= \underset{u\in\R^n}\argmin \Lp(u, \lambda_k),\label{eq:exact_alm_primal}\\
        \lam_{k+1}&= \lam_k + \rho(Ax_{k+1}-b),\label{eq:exact_alm_dual}
    \end{align}
    where
    \begin{equation}\label{def:augmented_lagrangian}
        \Lp(x,\lambda) = \phi(x) + \inner{\lambda}{Ax-b} +\frac{\rho}{2}\|Ax-b\|^2
    \end{equation}
    is the augmented Lagrangian with penalty coefficient $\rho>0$.
    First analyzed by Hestenes~\cite{hestenes1969multiplier} and Powell~\cite{powell1969method}, variants of ALM have become standard methods for linearly-constrained optimization. 
    The classical exact ALM is typically intractable to implement, motivating the development of the inexact ALM (I-ALM),
    \begin{align}
        x_{k+1}&\approx \underset{u\in\R^n}\argmin\Lp(u, \lambda_k)\label{eq:AL-x},\\
        \lam_{k+1}&= \lam_k + \rho(Ax_{k+1}-b)\label{eq:AL-y},
    \end{align}
    which permits some inexactness in the primal minimization step. While analysis of the I-ALM dates back to Rockafellar~\cite{rockafellarAugmentedLagrangiansApplications1976}, non-asymptotic guarantees for more general problem classes have emerged only recently.
The authors of~\cite{lanIterationcomplexityFirstorderAugmented2016a} provided non-ergodic complexity bounds for an I-ALM when $h(x)=\delta_Q(x)$ is the indicator function of a compact convex set $Q$. In a pattern replicated in later works,~\cite{lanIterationcomplexityFirstorderAugmented2016a} used ACG as a first-order inner solver for~\eqref{eq:AL-x} and separated their ``inner'' and ``outer'' complexity analyses. The baseline ACG-based I-ALM was shown to have an $\mathcal{O}(\varepsilon^{-7/4})$ complexity. To improve the complexity,~\cite{lanIterationcomplexityFirstorderAugmented2016a} added a strongly convex perturbation to~\eqref{def:augmented_lagrangian}, $\gamma_p\|x-x_0\|^2/2$ for some $x_0\in \dom h$ with $\gamma_p\propto \varepsilon/D$ to maintain an $\varepsilon$-primal-dual solution to the original problem. The perturbation seemingly improved I-ALM iteration complexity to $\tO(\varepsilon^{-1})$, however, as shown by~\cite{luIterationComplexityFirstOrderAugmented2023}, this perturbation adds a hidden $\varepsilon$ dependence.
Accordingly, further studies have rectified and extended non-asymptotic guarantees for the I-ALM. 
 Authors have proven $\tO(\varepsilon^{-1})$ complexities by considering ergodic iterates~\cite{patrascuAdaptiveInexactFast2017,xuIterationComplexityInexact2021} and geometrically increasing penalty terms~\cite{luIterationComplexityFirstOrderAugmented2023,xuIterationComplexityInexact2021}. Several works have also extended the problem class to include nonlinear inequality constraints~\cite{luIterationComplexityFirstOrderAugmented2023,xuIterationComplexityInexact2021}, general simple nonsmooth $h$~\cite{liuNonergodicConvergenceRate2019,luIterationComplexityFirstOrderAugmented2023,xuIterationComplexityInexact2021}, and projection-free inner subroutines~\cite{liuNonergodicConvergenceRate2019}. We refer interested readers to the recent survey~\cite{dengAugmentedLagrangianMethods2025a} for a more comprehensive treatment of the ALM in mathematical programming.  
 \textbf{Remark.} Numerous works that use inexact first-order subroutines also assume that $\dom h$ is bounded~\cite{lanIterationcomplexityFirstorderAugmented2016a,liuNonergodicConvergenceRate2019,luIterationComplexityFirstOrderAugmented2023,patrascuAdaptiveInexactFast2017,xuIterationComplexityInexact2021}. Interestingly, works which do not assume boundedness also do not rely on inexact first-order subroutines, instead using explicit minimization (either as a theoretical oracle, a linear program, or as a linearized approximation~\cite{ouyangAcceleratedLinearizedAlternating2015,sabachFasterLagrangianBasedMethods2022,xuAcceleratedFirstOrderPrimalDual2017}).

While the Restarted ACG and I-ALM algorithms target distinct problem classes, we can view them in a unified perspective by considering inexact proximal point (IPP) methods~\cite{rockafellarAugmentedLagrangiansApplications1976,rockafellar1976monotone,solodovHybridApproximateExtragradient1999,solodovInexactHybridGeneralized2000a}, which solve a generic optimization problem $\Phi_*:=\min_{x \in \R^n} \Phi(x)$ by the iteration
\begin{equation*}
    x_{k+1}\approx\underset{x\in \R^n}\argmin\left\{\Phi(x)+\frac{1}{2\lambda}\|x-x_k\|^2\right\}.
\end{equation*}
The approximation ``$\approx$'' can be characterized in a number of ways, either bounded by some absolute tolerance $\delta_k\geq 0$~\cite{rockafellar1976monotone,solodovHybridApproximateExtragradient1999} or by some relative term $\|x_{k+1}-x_k\|$~\cite{monteiro2010complexity,MonteiroSvaiterAcceleration,solodovHybridApproximateExtragradient1999}. IPP frameworks have long been used to analyze algorithms for optimization over problems with convex structures. Rockafellar's absolute error~\cite{rockafellarAugmentedLagrangiansApplications1976} framework has repeatedly been used for I-ALM analysis~\cite{liuNonergodicConvergenceRate2019,xuIterationComplexityInexact2021}. Solodov and Svaiter's HPE framework provided the first iteration-complexity bound for ADMM~\cite{monteiroIterationComplexityBlockDecompositionAlgorithms2013} and its accelerated extension~\cite{MonteiroSvaiterAcceleration} has been instrumental in the development of high-order methods~\cite{bubeckNearoptimalMethodHighly2019,gasnikovOptimalTensorMethods2019,jiangOptimalHighOrderTensor2021}.

\noindent
\textbf{Contributions.}
For CSCO problems, we propose a novel Restarted ACG method (Algorithm~\ref{alg:restart}) that achieves the same optimal iteration-complexity as that of ACG in both convex and strongly convex settings. To our knowledge, this is a novel result in the restarted ACG literature. For LC-CSCO problems, we first prove that a classical I-ALM algorithm (Algorithm~\ref{alg:al}) achieves near-optimal, non-ergodic $\tO(\varepsilon^{-1})$ complexity with a fixed penalty parameter $\rho>0$. To our knowledge, this is a novel finding in the ALM literature, where previous methods either require geometrically increasing $\rho$ or ergodic convergence to achieve near-optimal complexity. Building on our analysis of I-ALM, we propose a dual-accelerated inexact ``fast'' ALM (I-FALM, Algorithm~\ref{alg:dsc_aalm}), which achieves $\tO(\varepsilon^{-1})$ non-ergodic complexity with improved dependence on the domain diameter $D$. Table~\ref{tab:lit_review_alm} places the proposed I-ALM and I-FALM methods in the context of the broader I-ALM literature, where the stated complexity is to find an $\varepsilon$-primal solution in the sense of~\eqref{def:primal_lc_sol}. Numerical experiments suggest that Algorithm~\ref{alg:restart} is competitive with existing restart schemes, as well as showing that Algorithm~\ref{alg:dsc_aalm} can significantly outperform existing non-ergodic I-ALM variants. 

Furthermore, we introduce two analytical frameworks, namely lower oracle approximation (LOrA) and its accelerated variant, fast LOrA (FLOrA), which provide a unified and principled foundation for analyzing algorithms for solving either CSCO or LC-CSCO. LOrA is an IPP framework built upon ``lower estimation" functions that arise naturally in the analysis of convex optimization methods. It encompasses, as special cases, proximal gradient methods, proximal bundle methods, and I-ALM. By incorporating Nesterov's acceleration scheme into LOrA, we develop FLOrA, which further captures ACG, Restarted ACG, and I-FALM as instances.

\noindent
\textbf{Organization.}
Section~\ref{sec:primal} provides an overview of ACG and introduces the Restarted ACG algorithm along with its optimal iteration-complexity. Section~\ref{sec:dual} provides the setup and complexity bounds for the I-ALM and I-FALM algorithms in Subsections~\ref{ssec:baseline_alm} and~\ref{ssec:acc_alm}, respectively. Section~\ref{sec:framework} introduces the LOrA and FLOrA frameworks along with their theoretical guarantees. Building on the two frameworks, Section~\ref{sec:proof} proves the main complexity results presented in Sections~\ref{sec:primal} and~\ref{sec:dual}. Preliminary computational results are reported in Section~\ref{sec:numerical}. Section~\ref{sec:conclusion} provides concluding remarks and potential future directions. Appendix~\ref{appdx:numerical_details} provides additional details for numerical experiments. Technical lemmas can be found in Appendix~\ref{appdx:technical}. LOrA and FLOrA analyses are presented in Appendix~\ref{appdx:framework}. Appendices~\ref{appdx:primal_deferred} and~\ref{appdx:deferred_alm} contain deferred proofs relevant to Sections~\ref{sec:primal} and~\ref{sec:dual}, respectively. 
\begin{table}[H]
    \centering
    \begin{tabular}{l|c|c|c|c|c|c|c}
    \toprule
        Paper & Alg. &  Complexity & $\rho$ & $\phi$ & Constraints & Subroutine & Conv. Pt. \\
        \midrule
        \multirow{1}{*}{\cite{lanIterationcomplexityFirstorderAugmented2016a}} & I-ALM & $\mathcal{O}(\varepsilon^{-7/4})$ & Static &$f+\delta_Q$ & Linear & ACG & Non-Erg. \\ 
        \cite{patrascuAdaptiveInexactFast2017} & IFAL & $\mathcal{O}(\varepsilon^{-1})$  & Static & $f+\delta_Q$ &  Linear & ACG & Erg. \\ 
        \cite{liuNonergodicConvergenceRate2019} & I-ALM & $\mathcal{O}(\varepsilon^{-2})$ & Static & $f+h$ & Linear & ACG/CG & Non-Erg. \\
        \multirow{2}{*}{\cite{xuIterationComplexityInexact2021}} &I-ALM & $\mathcal{O}(\varepsilon^{-1})$ & Geo. & $f+h$ & Nonlinear & ACG & Non-Erg. \\ 
         &I-ALM & $\mathcal{O}(\varepsilon^{-1})$ &  St./Geo. & $f+h$ & Nonlinear & ACG & Erg.\\
        
        \cite{luIterationComplexityFirstOrderAugmented2023}& aI-ALM & $\tO(\varepsilon^{-1})$ & Geo. & $f+h$ & Nonlinear & ACG & Non-Erg. \\
        \cite{liAcceleratedAlternatingDirection2019} & LPALM & $\mathcal{O}(\varepsilon^{-1})$ & Static & $f+h$ & Linear & Prox & Non-Erg.\\
        \textbf{TW} & Alg. \ref{alg:al} & $\tO(\varepsilon^{-1})$ & Static & $f+h$ & Linear & ACG & Non-Erg. \\   
        \textbf{TW} & Alg. \ref{alg:dsc_aalm} & $\tO(\varepsilon^{-1})$ & Static & $f+h$ & Linear & ACG & Non-Erg. \\ 
        \bottomrule
    \end{tabular}
    \caption{Non-exhaustive summary of related works on I-ALM. \textbf{TW} indicates ``This Work''. For simplicity, we use the common term ``ACG'' to refer to either ACG (Algorithm~\ref{alg:ACG}) or related variants such as FISTA~\cite{beck2009fast}.  ``CG'' refers to the Conditional Gradient (or Frank-Wolfe) algorithm~\cite{braun_conditional_2025,frank1956algorithm}. ``Prox'' refers to a single, closed-form proximal mapping for a linearized augmented Lagrangian model~\cite{liAcceleratedAlternatingDirection2019,ouyangAcceleratedLinearizedAlternating2015}. ``Static'' $\rho$ selection refers to choosing a constant penalty $\rho$ across all iterations, while ``Geo(metric)'' refers to a geometrically increasing $\rho$, i.e., $\rho_k=\rho_0\cdot \beta^k$ for some $\beta > 1$. ``Conv. Pt.'' refers to the point of convergence, where ``Erg(odic)'' refers to convergence in an averaged point (e.g., $\hat x_k=k^{-1}\sum_{i=1}^kx_i$) while ``Non-Erg(odic)'' directly shows convergence in some single iterate $x_k$ (e.g., the best or the last). $\delta_Q$ is taken to be the indicator function of some simple, closed convex set $Q$, $h$ is a simple, possibly nonsmooth closed convex function, and $f$ is a smooth closed convex function. Algorithm acronyms are: ``IFAL'' is ``Iterative  Fast Augmented Lagrangian'', ``aI-ALM'' is ``adaptive I-ALM'', and ``LPALM'' is ``Linearized Proximal ALM''. All iteration-complexity results are to obtain an $\varepsilon$-primal solution to~\eqref{eq:ProbIntro_LC} in the sense of~\eqref{def:primal_lc_sol}.}
    \label{tab:lit_review_alm}
\end{table}

\subsection{Basic Definitions and Notation} \label{subsec:DefNot}    
    

    The set of real numbers is denoted by $\R$, non-negative reals by $\R_+$, and positive reals by $\R_{++}$. Let $\R^n$ be the $n$-dimensional Euclidean space equipped with the standard inner product $\inner{\cdot}{\cdot}$ and norm $\|\cdot\|$. Let $\R^{n\times n}$ be the space of real-valued $n\times n$ matrices equipped with the spectral norm 
    \[
        \|A\|=\sup_{x\in\R^n}\{\|Ax\|:\|x\|\le 1\}.
    \]

    For a convex set $Q\subseteq\R^n$, we define the \textit{diameter} $D$ as $D=\sup_{x,y\in Q}\{\|x-y\|\}$.
    If $D<\infty$, then $Q$ is bounded. We define the \textit{relative interior} of $Q$, $\mathrm{relint}(Q)$ as
    \[\operatorname{relint}(Q)=\{x\in Q: B(x,r)\cap \operatorname{affine}(Q) \subseteq Q\text{ for some }r>0\},\]
    where $\operatorname{affine}(Q)$ is the affine hull of $Q$. We say that~\eqref{eq:ProbIntro_LC} satisfies \textit{Slater's condition} if there exists a feasible point in $\operatorname{relint}(\dom h)$, i.e., 
    \begin{equation}
        \operatorname{relint}(\dom h)\cap \{x\in\R^n:Ax=b\}\neq \emptyset.\label{def:slater}
    \end{equation}

	For a proper function $f$, the \emph{subdifferential} of $ f $ at $x \in \dom f$ is denoted by
	\[
	    \partial f (x):=\left\{ s \in\R^n: f(y)\geq f(x)+\left\langle s,y-x\right\rangle, \forall y\in\R^n\right\}.
	\]
    For a given \textit{subgradient}
$f'(x) \in \partial f(x)$, we denote the \textit{linearization of $f$ at $x$} by $\ell_f(\cdot;x)$, which is defined as
\[
\ell_f(\cdot;x):=f(x)+\inner{ f'(x)}{\cdot-x}.
\]
For a function $f:\R^n\to(-\infty,+\infty]$, we denote its effective domain by $\dom f = \{x:f(x)<+\infty\}$. We say that $f$ is \textit{$\mu$-strongly convex} for some $\mu > 0$ if for every $x, y \in \dom f$ and $\lam \in [0,1]$,
    \[
    f(\lam x+(1-\lam) y)\le \lam f(x)+(1-\lam)f(y) - \frac{\lam(1-\lam) \mu}{2}\|x-y\|^2.
    \]
    Equivalently, $f$ is $\mu$-strongly convex if for every $x,y\in\dom f$ and all $f'(x)\in\partial f(x)\neq \emptyset$,
    \begin{equation*}
        f(y)-f(x)-\inner{f'(x)}{y-x}\geq \frac{\mu}{2}\|x-y\|^2.
    \end{equation*}
    With $\mu=0$ we recover the standard definitions of convexity. We denote the set of proper closed $\mu$-strongly convex functions over set $Q$ as $\CConv{\mu}{Q}$, with $\CConv{}{Q}$ used if $\mu=0$.

    We say that a differentiable function $f$ is \textit{$L_f$-smooth} if $\nabla f$ is $L_f$-Lipschitz continuous on $\R^n$. Equivalently, $f$ is $L_f$-smooth if there exists an $L_f>0$ such that for every $x, y \in \R^n$
    \begin{equation}\label{ineq:Lsmooth}
       f(y)-f(x)-\inner{\nabla f(x)}{y-x}\leq \frac{L_f}{2}\|x-y\|^2.
    \end{equation}
    We define the \emph{proximal mapping} (or ``prox mapping'') of a closed convex function $h$ as 
    \[
        \prox_{h}(x)=\underset{y\in\R^n}\argmin\left\{h(y)+\frac{1}{2}\|x-y\|^2\right\}.
    \]
    We say $h$ is \emph{simple} if it has a computable prox mapping. We define the \emph{iteration-complexity} of an algorithm as the number of prox mappings it needs to solve a problem to a specified tolerance.
    
    Given a positive scalar $\lambda$ and a composite function $\phi(x)=f(x)+h(x)$, where $f$ is smooth and $h$ has an available prox mapping, we define the \textit{gradient mapping} $\mathcal{G}^{\lambda}_\phi(x)$ as
    \begin{equation}\label{def:grad_mapping}
        \mathcal{G}^\lambda_\phi(x)=\frac1{\lambda} (x-\prox_{\lambda h}(x-\lambda\nabla f(x))).
    \end{equation}

        
    

\section{Primal Algorithm: Restarted ACG}	\label{sec:primal}
In this section, we consider the CSCO problem~\eqref{eq:ProbIntro} under the following standard assumptions.

\begin{assumption}\label{assmp:csco} Problem~\eqref{eq:ProbIntro} satisfies the following:
\begin{enumerate}[label={\rm(\alph*)}]
    \item $f$ is proper closed, $\mu_f$-strongly convex and $L_f$-smooth on $\R^n$,
    \item the smoothness parameter $L_f$ and strong convexity parameter $\mu_f$ satisfy $L_f\geq 2\mu_f\geq 0$,\footnote{This assumption can be made without loss of generality, since the definition of $L_f$-smoothness in~\eqref{ineq:Lsmooth} implies $f$ is also $2L_f$-smooth. In view of Theorem~\ref{thm:main}, this only incurs a constant $\sqrt{2}$ factor increase in the iteration complexity of Algorithm~\ref{alg:restart}.}
    \item $h$ is proper closed convex with a simple proximal mapping.
\end{enumerate}
\end{assumption}
    

In Subsection~\ref{subsec:ACG}, we begin by introducing a variant of ACG (Algorithm~\ref{alg:ACG}) for solving a regularized version of problem~\eqref{eq:ProbIntro}, and we establish its convergence rate bound as an inner solver. Building on this result, Subsection~\ref{subsec:algorithm} proposes a Restarted ACG method that repeatedly invokes Algorithm~\ref{alg:ACG} to solve a sequence of proximal subproblems of \eqref{eq:ProbIntro}. We analyze the outer iteration complexity of this restarted method, and by combining the inner complexity of Algorithm~\ref{alg:ACG}, we derive the overall complexity of the Restarted ACG method.

\subsection{Overview of an ACG variant}\label{subsec:ACG} 
    Throughout this work, we will utilize ACG as a subroutine to solve regularized subproblems. The generic regularized subproblem we consider is of the form
	\begin{equation}\label{eq:prob}
	    \min \{\psi(x):=g(x)+h(x):x\in \R^n \},
	\end{equation}
	where $ g $ is $ \mu $-strongly convex and $ (L+\mu) $-smooth, and $ h $ is a convex and possibly nonsmooth function with a simple proximal mapping, satisfying $\dom h \subset \dom g$. We describe an ACG variant tailored to \eqref{eq:prob} and present some basic results regarding the ACG variant.

\begin{algorithm}[H]
\caption{Accelerated Composite Gradient}\label{alg:ACG}
\begin{algorithmic}
\REQUIRE given initial point $ x_0\in \dom \psi $, $L\geq 0$, and $\mu\geq0$, set $A_0=0$, $\tau_0=1$, and $ y_0=x_0$.

\FOR{$j=0,1,\cdots$}

\STATE {\bf 1.} Compute
		\begin{gather}
		    a_j=\frac{\tau_j + \sqrt{\tau_j^2+8\tau_j A_jL}}{4L}, \quad A_{j+1} = A_j + a_j,\quad \tau_{j+1}=\tau_j + \mu a_j,\label{def:acg_scalar}\\
            \tx_j=\frac{A_j}{A_{j+1}}y_j + \frac{a_j}{A_{j+1}}x_j\label{def:tx}.
		\end{gather}

\STATE {\bf 2.} Compute  
		\begin{align}
		    \ty_{j+1} &=\underset{u\in \R^n}\argmin\left\lbrace \ell_g(u;\tx_j) + h(u) + \frac{2L+\mu}{2}\|u-\tx_j\|^2\right\rbrace, \label{def:tyj} \\
            y_{j+1} &=\argmin\left\lbrace \psi(y_j),\psi(\ty_{j+1})\right\rbrace, \label{def:yj} \\
            x_{j+1} &=\frac{(2L+\mu)a_j\ty_{j+1} - \frac{2A_ja_jL}{A_{j+1}} y_j}{A_{j+1}\mu + 1}. \label{def:xj}
		\end{align}	
\ENDFOR

\end{algorithmic} 
\end{algorithm}

    The following convergence rates are standard for first-order accelerated methods. However, we provide a self-contained analysis of Algorithm~\ref{alg:ACG} based on the FLOrA framework in Appendix~\ref{appdx:acg} for completeness.
\begin{lemma}\label{lem:acg_convergence}
    Define $R_0=\min\{\|x-x_0\|:x\in X_*\}$, where $X_*$ is the set of optimal solutions to~\eqref{eq:prob}. Then, for all $j\geq 1$,
    \begin{align}
        \psi(y_{j})-\psi(x_*)&\leq\frac{R_0^2}{2A_{j}},\label{ineq:acg_primal_gap}\\
        \|\mathcal{G}^{(2L+\mu)^{-1}}_{\psi}(\tx_{j-1})\|&\leq\frac{(2L+\mu)R_0}{\sqrt{LA_{j}}}\label{ineq:acg_grad_mapping}.
    \end{align}
\end{lemma}

    The following lemma develops technical bounds in terms of the relative quantity $\|y_j-x_0\|$. These bounds will be critical in analyzing the Restarted ACG algorithm proposed in Subsection~\ref{subsec:algorithm}. The proof is deferred to Appendix~\ref{appdx:acg}.
    
	\begin{lemma}\label{lem:tech}
        For every $j\ge 1$, define
        
		\begin{align}
        & \Gamma_j(\cdot):=\ell_g(\cdot;\tx_j)+h(\cdot)+\frac{2L+\mu}{2}\|u-\tx_j\|^2, \nn \\
		& \theta_{j+1}(x):=\Gamma_j(\ty_{j+1}) - L\|\ty_{j+1}-\tx_j\|^2 + \inner{u_{j+1}}{x-\ty_{j+1}}+\frac{\mu}{2}\|x-\ty_{j+1}\|^2,  \label{def:theta_acg}\\
		& \Theta_{j+1}(x):=\frac{A_j \Theta_j(x) + a_j \theta_{j+1}(x)}{A_{j+1}}, \label{def:Theta_acg}\\
       & \hat x_j :=\underset{u\in\R^n}\argmin\left\{ \Theta_j(u)\right\},\qquad s_j:=\frac{x_0-x_j}{A_j}\in\partial \Theta_j(x_j)\label{def:hatxj_sj},
		\end{align} 
    where $\Theta_0(\cdot)=0$. Assuming that $ A_j \ge 3/\mu $, then the following statements hold for every $j\ge 1$:
		\begin{itemize}
			\item[\rm a)] 
			\begin{equation}\label{ineq:mj_mu}
			\psi (y_j)- \Theta_j( \hat x_j) \le \frac{\mu}{\mu A_j-2}\|y_j-x_0\|^2;
			\end{equation}
			
			\item[\rm b)]
			\begin{equation}\label{ineq:sj}
			    \|s_j\|\le \frac{3\|y_j-x_0\|}{2A_j}.
			\end{equation}
		\end{itemize}
	\end{lemma}

\subsection{The Restarted ACG Method}\label{subsec:algorithm}

This subsection presents the Restarted ACG method to solve \eqref{eq:ProbIntro}.
Restarted ACG requires repeatedly invoking Algorithm \ref{alg:ACG} as a subroutine within a double-loop algorithm.
This approach aligns naturally with the IPP framework, which iteratively solves a sequence of proximal subproblems using a recursive subroutine. Within each loop of the IPP framework, Algorithm~\ref{alg:ACG} is employed to solve a certain proximal subproblem, while between successive loops, an acceleration scheme is applied. Consequently, the proposed Restarted ACG method (i.e., Algorithm~\ref{alg:restart}) can be described as ``doubly accelerated."

\begin{algorithm}[H]
\caption{Restarted ACG}\label{alg:restart}
\begin{algorithmic}
\REQUIRE given initial point $ w_0\in \dom h $, $\sigma\in(0,1)$, $L_f\geq0$, $\mu_f\geq 0$, and $\lam>0$, set $B_0=0$, $\tau_0=1$, and $ v_0=w_0 $.

\FOR{$k=0,1,\cdots$}

\STATE {\bf 1.} Compute 
		\[
		b_{k}=\frac{\tau_k\lam + \sqrt{\tau_k^2\lam^2+4\tau_k\lam B_{k}}}{2}, \quad B_{k+1} = B_{k} + b_{k},\quad \tau_{k+1}=\tau_k + b_k\mu_f,\]
        \begin{equation}
            \tilde v_{k}=\frac{B_{k}}{B_{k+1}}w_{k} + \frac{b_{k}}{B_{k+1}}v_{k}.\label{def:tv}
        \end{equation}
		
\STATE {\bf 2.} Call Algorithm~\ref{alg:ACG} with
\begin{equation}\label{eq:setup}
    x_0=\tilde v_{k}, \quad \psi(\cdot)= g(\cdot)+h(\cdot), \quad g(\cdot)=f(\cdot)+\frac{1}{2\lam}\|\cdot-\tilde v_{k}\|^2, \quad \mu=\mu_f + \frac{1}{\lam}, \quad L=L_f - \mu_f
\end{equation}

and perform $j$ iterations until 
\begin{equation}
    \|\lambda s_j\|^2+ 2\lambda[\psi(y_j)-\Theta_j(x_j)]\leq \sigma\|y_j-x_0\|^2\label{ineq:lora_restart_acg},
\end{equation}
where $y_j$ and $x_j$ are the ACG iterates defined in~\eqref{def:yj} and~\eqref{def:xj}, respectively, and $\Theta_j$ and $s_j$ are defined in~\eqref{def:Theta_acg} and~\eqref{def:hatxj_sj}, respectively.

\STATE {\bf 3.} Choose $ w_{k+1}\in \Argmin\left\lbrace \phi(u): u\in \{w_{k}, y_j\}\right\rbrace$ and compute
\begin{equation}
    v_{k+1}=\frac{1}{\tau_{k+1}}\left(\tau_k v_{k} + b_k\mu_f  x_j-b_{k} \frac{A_j+\lambda}{\lambda}s_j\right),\label{def:vkp1}
\end{equation}
where $A_j$ is the ACG scalar as in~\eqref{def:acg_scalar}.

\ENDFOR

\end{algorithmic} 
\end{algorithm}

    From the ``inner loop" perspective, Algorithm~\ref{alg:restart} keeps performing ACG iterations to solve the proximal subproblem \eqref{eq:prob} with specification as in \eqref{eq:setup} until \eqref{ineq:lora_restart_acg} is satisfied, and then restarts ACG with the initialization as in \eqref{eq:setup}. 
     From the ``outer loop" perspective, Algorithm~\ref{alg:restart} is an instance of the FLOrA framework for solving \eqref{eq:ProbIntro} with ACG as its subroutine to implement Step \textbf{2} of Algorithm~\ref{alg:flora}, as we will show in Subsection~\ref{ssec:proofs_restart}. 
 


The next result combines the ``outer'' and ``inner'' complexities (see Propositions~\ref{prop:outer} and ~\ref{prop:inner_to_outer}, respectively) to obtain the total iteration-complexity of Algorithm~\ref{alg:restart}.

\begin{theorem}\label{thm:main}
    For given $\varepsilon>0$, the following statements hold:
    \begin{enumerate}[label={\rm (\alph*)}]
        \item if $\mu_f=0$ and $1/L_f\leq \lam \leq R_0^2/\varepsilon$, then the total iteration-complexity of Algorithm~\ref{alg:restart} to find an $\varepsilon$-solution is
        $\tO\left(R_0\sqrt{L_f/\varepsilon}\right)$;
        \item if $\mu_f>0$ and $1/(L_f-\mu_f)\leq \lam\leq\min\{1/\mu_f,R_0^2/\varepsilon\}$, then the total iteration-complexity of Algorithm~\ref{alg:restart} to find an $\varepsilon$-solution is $\tO(\min\{\sqrt{L_f/\mu_f}, R_0\sqrt{L_f/\varepsilon}\}))$.
    \end{enumerate}
\end{theorem}

If $\lam$ is taken to be $1/(L_f-\mu_f)$, we can show that the number of ACG iterations on each call to Algorithm~\ref{alg:ACG} is $\mathcal{O}(1)$ (see~\eqref{eq:bound} below). If $\lam$ is taken sufficiently small, then each call will only perform a single ACG iteration, effectively reducing Restarted ACG (i.e., Algorithm~\ref{alg:restart}) to standard ACG (i.e., Algorithm~\ref{alg:ACG}).

\section{Dual Algorithm: Augmented Lagrangian}\label{sec:dual}
In this section we consider the LC-CSCO problem~\eqref{eq:ProbIntro_LC} under the following standard assumptions.

\begin{assumption}\label{assmp:constrained} Problem~\eqref{eq:ProbIntro_LC} satisfies the following:
\begin{enumerate}[label={\rm(\alph*)}]
    \item $f$ is proper closed convex and $L_f$-smooth on $\R^n$,
    \item $h$ is proper closed convex with a simple proximal mapping,
    \item Slater's condition (i.e.,~\eqref{def:slater}) is satisfied,
    \item  $\dom h$ is bounded with diameter $D\geq 1$.
\end{enumerate}
\end{assumption}

The exact ALM (see~\eqref{eq:exact_alm_primal} and~\eqref{eq:exact_alm_dual}) was originally developed from the primal perspective~\cite{hestenes1969multiplier}, with the quadratic term $\rho\|Ax-b\|/2$ motivated by explicit penalty methods. However, as noted by Rockafellar~\cite{rockafellarAugmentedLagrangiansApplications1976}, the ALM can be reformulated as a proximal point method in the dual,
\begin{equation}\label{eq:exact_ppm_alm}
    \lambda_{k+1} = \underset{\lambda\in\R^m}\argmax\left\{d(\lambda)-\frac{1}{2\rho}\|\lambda-\lambda_{k}\|^2\right\},
\end{equation}
where $\rho>0$ is now the proximal stepsize. As mentioned in Section~\ref{sec:intro}, however, solving~\eqref{eq:exact_ppm_alm} (i.e.,~\eqref{eq:exact_alm_primal}) is typically intractable. Accordingly, practitioners instead adopt the I-ALM with the inexact primal step~\eqref{eq:AL-x}. Since ALM is equivalent to the proximal point method, it is only natural to suppose that I-ALM is equivalent to the IPP iteration
\[
     \lambda_{k+1} \approx \underset{\lambda\in\R^m}\argmax\left\{d(\lambda)-\frac{1}{2\rho}\|\lambda-\lambda_{k}\|^2\right\},
\]
for some suitable definition of ``inexactness''. Letting $\hat\lambda_k$ be the exact minimizer to the dual proximal problem~\eqref{eq:exact_ppm_alm}, Rockafellar~\cite[Proposition 6]{rockafellarAugmentedLagrangiansApplications1976} proved that
\begin{equation*}
    \frac{1}{2\rho}\|\lambda^{k+1}-\hat\lambda_k\|^2\leq \Lp(x_{k+1},\lambda_k)-\min_{x\in \R^n}\Lp(x,\lambda_k).
\end{equation*}
Thus, if we can ensure $\Lp(x_{k+1},\lambda_k)-\min_{x\in \R^n}\Lp(x,\lambda_k)\leq \varepsilon_k$ for some summable sequence $\{\varepsilon_k\}_{k=0}^\infty$, then we can show (see~\cite[Theorem 4]{rockafellarAugmentedLagrangiansApplications1976}) that $\lim_{k\to\infty}\lambda_k=\lambda_*$ for some $\lambda_*\in\{\lambda:d(\lambda)=\max_{\nu\in\R^m} d(\nu)\}$. This ``absolute error'' IPP perspective has persisted in several recent analyses of the I-ALM~\cite{liuNonergodicConvergenceRate2019,xuIterationComplexityInexact2021}. 

Instead of a traditional absolute error framework, we use the LOrA and FLOrA frameworks of Section~\ref{sec:framework} to provide an IPP perspective on the I-ALM, enabling us to mix relative and absolute error criteria. Subsection~\ref{ssec:baseline_alm} provides near-optimal iteration-complexity bounds for a baseline I-ALM (Algorithm~\ref{alg:al}), improving on the non-ergodic complexity bounds from~\cite{lanIterationcomplexityFirstorderAugmented2016a,liuNonergodicConvergenceRate2019}. Utilizing the FLOrA framework, Subsection~\ref{ssec:acc_alm} then proposes an accelerated ALM variant, I-FALM (Algorithm~\ref{alg:dsc_aalm}).

\subsection{Inexact Augmented Lagrangian Method}\label{ssec:baseline_alm}
In this subsection we prove near-optimal, non-ergodic $\tO(\varepsilon^{-1})$ complexity for the I-ALM (Algorithm~\ref{alg:al}) with constant penalty term $\rho>0$. For convenience, we denote the smooth part of the augmented Lagrangian~\eqref{def:augmented_lagrangian} as

\begin{equation*}
    \Psi_{\lambda}(x):=f(x)+\inner{\lambda}{Ax-b}+\frac{\rho}{2}\|Ax-b\|^2,
\end{equation*}
with smoothness constant $M_\rho$
\begin{equation}
    M_\rho:=L_f+\rho\|A\|^2.\label{def:Mrho}
\end{equation}

The termination condition for the inexact iteration~\eqref{eq:AL-x} is typically stated in terms of an $\varepsilon_k$-small objective gap, i.e., $\Lp(x_{k+1},\lam_k)-\min_{x\in \R^n}\Lp(x,\lam_k)\leq \varepsilon_k$, for some specified tolerance $\varepsilon_k>0$. In most cases, however, the exact objective gap is not computable. Supposing that $x_{k+1}$ is computed from a proximal mapping with stepsize $\eta<1/L_f$, i.e., $x_{k+1}=\prox_{\eta h}(\tx_k-\eta\nabla f(\tx_k))$, we can use an alternative termination condition based on the gradient mapping $\mathcal{G}_{\Lp(\cdot,\lam_k)}^{\eta}(\tx_k)$, defined in~\eqref{def:grad_mapping}. Applying Lemma~\ref{lem:grad_map_merged}(a) with $\tx=\tx_k$ and $x^+=x_{k+1}$, the condition $\|\mathcal{G}_{\Lp(\cdot,\lam_k)}^{\eta}(\tx_k)\|\leq \varepsilon_k/D$ implies $\Lp(x_{k+1},\lam_k)-\min_{x\in \R^n}\Lp(x,\lam_k)\leq \varepsilon_k$. Unlike the primal gap, the gradient mapping norm is an explicit and efficiently computable quantity: providing a practical inner termination condition. 

However, if the objective $\Lp(\cdot,\lam_k)$ is merely convex, then we can show that Algorithm~\ref{alg:ACG} requires $\mathcal{O}(\varepsilon_k^{-2/3})$ iterations to guarantee an $\varepsilon_k$-small gradient mapping~\cite{nesterov2013gradient}, worse than the $\mathcal{O}(\varepsilon_k^{-1/2})$ complexity needed for an $\varepsilon_k$-small primal gap. To improve the complexity of the inner call, we can instead optimize the objective
\[
    \min_{x\in\R^n}\left\{\Lp(x,\lam_k)+\frac{\varepsilon_k}{4D^2}\|x-x_k\|^2\right\},
\]
which, as shown in Proposition~\ref{prop:alm_inner_complexity} below, guarantees $\tO(D/\sqrt{\varepsilon_{k}})$ complexity for each inner iteration. This ``perturbation'' trick is common for improving the complexity of finding a gradient~\cite[Subsection 2.2.2]{nesterov2018lectures} or gradient mapping~\cite[Subsection 5.2]{nesterov2013gradient} with $\varepsilon_k$-small norm. A gradient mapping termination criterion also removes the need for post-processing routines (e.g., \cite{lanIterationcomplexityFirstorderAugmented2016a}) and improves theoretical guarantees with fixed $\rho$, as we discuss further in remarks following Theorem~\ref{thm:al_baseline_complexity}.



\begin{algorithm}[H]
\caption{Inexact Augmented Lagrangian Method}\label{alg:al}
\begin{algorithmic}
\REQUIRE given initial point $x_0\in \dom{h}$, $\rho>0$, $\varepsilon_0>0$, $\alpha\in(0,1)$, $\varepsilon > 0$, set $\lambda_0=0$, and choose $\sigma\in (0,1)$ such that $2\sigma\rho\leq D/\varepsilon$. 
\FOR{$k=0,1,\cdots$}
\STATE {\bf 1.} Set $\varepsilon_k=(\varepsilon_0\alpha^k + \sigma\rho\varepsilon^2)/2$ and call Algorithm~\ref{alg:ACG} with 
\begin{equation}
\begin{gathered}
x_0=x_k,\quad \psi(\cdot)=\Lp(\cdot,\lambda_k)+\frac{\varepsilon_k}{8D^2}\|\cdot-x_k\|^2,\quad g(\cdot)=\Psi_{\lambda_k}(\cdot)+\frac{\varepsilon_k}{8D^2}\|\cdot-x_k\|^2,\\
L=M_\rho,\quad\mu=\frac{\varepsilon_k}{4D^2},
\end{gathered}
\label{def:setup_alm}
\end{equation}
to find a $\tx_{k}$ satisfying $ \|\mathcal{G}^{(2L+\mu)^{-1}}_{\Lp(\cdot,\lambda_k)}(\tx_k)\|\leq \varepsilon_k/(2D)$
and set
\begin{equation*}
    x_{k+1}=\tx_k-(2L+\mu)^{-1}\mathcal{G}^{(2L+\mu)^{-1}}_{\Lp(\cdot,\lambda_k)}(\tx_k).
\end{equation*}
\STATE {\bf 2.} Compute 
\begin{equation}
\lambda_{k+1}=\lambda_{k}+\rho(Ax_{k+1}-b).
\label{def:lambda_alm}\end{equation}
\STATE {\bf 3.} If $\|\mathcal{G}^{(2L+\mu)^{-1}}_{\Lp(\cdot,\lambda_k)}(\tx_k)\|\leq \varepsilon/2$ and $\|Ax_{k+1}-b\|\leq \varepsilon$, then \textbf{return} $(x_{k+1},\lambda_{k+1})$.
\ENDFOR
\end{algorithmic} 
\end{algorithm}

To analyze Algorithm~\ref{alg:al}, we separately consider the ``inner'' and ``outer'' perspectives. For the inner, we apply the known iteration-complexity guarantees of ACG to achieve the termination condition in Step \textbf{2} (see Proposition~\ref{prop:alm_inner_complexity} below). For the outer, we first prove I-ALM as an instance of the LOrA framework (see Section~\ref{subsec:LORA} below), and then apply the sub-optimality guarantee of LOrA to obtain the outer iteration-complexity. Combining the two perspectives yields the following iteration-complexity bound, whose proof is deferred to Subsection~\ref{ssec:proofs_alm}.
\begin{theorem}\label{thm:al_baseline_complexity}
    Given $\varepsilon>0$, we choose $\varepsilon_0=\varepsilon$, $\sigma=1/2$, and $\rho=\varepsilon^{-1}$. Then, Algorithm~\ref{alg:al} finds an $\varepsilon$-primal-dual solution to~\eqref{eq:ProbIntro_LC} in
    \begin{equation}\label{cmplx:total_alm}
    \tO\left((1+R_\Lambda^2)\left(1+D\left(\frac{\sqrt{L_f}}{\sqrt{\varepsilon}}+\frac{\|A\|}{\varepsilon}\right)\right)\right)
    \end{equation}
    ACG iterations,
    where $R_\Lambda=\|\lambda_*-\lambda_0\|=\min\{\|\lambda-\lambda_0\|:\lambda\in\Lambda_*\}$, where $\Lambda_*$ is the set of maximizers to the dual problem~\eqref{eq:strong_duality}.
\end{theorem}

\noindent\textbf{Remark. }Lan and Monteiro~\cite{lanIterationcomplexityFirstorderAugmented2016a} analyzed I-ALM with a static $\rho$ similar to Algorithm~\ref{alg:al}, obtaining an iteration-complexity of $\mathcal{O}(\varepsilon^{-7/4})$. While there are a number of differences between Lan and Monteiro's approach and ours, the $\varepsilon$-complexity disparity can be attributed primarily to the method of ensuring $\varepsilon$-stationarity. Lan and Monteiro used a ``refinement'' final call to Algorithm~\ref{alg:ACG} with $\mathcal{O}(M_\rho^{-1}\varepsilon^2)$ accuracy, thereby requiring $\mathcal{O}(M_\rho/\varepsilon)=\mathcal{O}((L_f+\rho\|A\|^2)/\varepsilon )$ iterations. Setting $\rho=\varepsilon^{-1}$ would result in the ``refinement'' phase taking $\mathcal{O}(\varepsilon^{-2})$ iterations, requiring the authors to trade off ``main loop'' and ``refinement'' complexity. In contrast, our usage of gradient mapping norms to provide stationarity guarantees does not require a postprocessing stage, and each inner call takes $\mathcal{O}((\sqrt{L_f}+\sqrt{\rho}\|A\|)/({\varepsilon\sqrt{\rho}}))$ iterations (see~\eqref{cmplx:inner_alm} below), enabling us to take $\rho=\varepsilon^{-1}$ without adding superfluous $\varepsilon$-dependence.

Using Lemma~\ref{lem:pd_gap} from Appendix \ref{appdx:technical}, we can translate the $\varepsilon$-solution complexity in Theorem~\ref{thm:al_baseline_complexity} into the complexity to find an $\varepsilon$-primal solution in the sense of~\eqref{def:primal_lc_sol}. The proof is deferred to Appendix~\ref{appdx:deferred_alm}.

\begin{corollary} \label{cor:alm_complexity_pd}
    Under the conditions and parameter choices of Theorem~\ref{thm:al_baseline_complexity}, Algorithm~\ref{alg:al} finds an $\varepsilon$-primal solution to~\eqref{eq:ProbIntro_LC} in
    \begin{equation}
        \tO\left((1+R_\Lambda^2)\left(1+D\left(\frac{\sqrt{(R_\Lambda + D)L_f}}{\sqrt{\varepsilon}} + \frac{(R_\Lambda + D)\|A\|}{\varepsilon}\right)\right)\right)\label{cmplx:alm_pd_total}
    \end{equation}
    ACG iterations, where $R_\Lambda=\|\lambda_*-\lambda_0\|=\min\{\|\lambda-\lambda_0\|:\lambda\in\Lambda_*\}$, and $\Lambda_*$ is the set of maximizers to the dual problem~\eqref{eq:strong_duality}.
\end{corollary}

\noindent\textbf{Remark.} Comparing to the lower bounds for primal convergence established in~\cite[Theorem 3.1]{ouyangLowerComplexityBounds2021}, the complexity of Corollary~\ref{cor:alm_complexity_pd} is optimal (up to logarithmic terms) in $\varepsilon$, $L_f$, and $\|A\|$. However, it is suboptimal in $R_\Lambda$ ($\mathcal{O}(R_\Lambda^3)$ vs $\mathcal{O}(R_\Lambda)$) and $D$ ($\mathcal{O}(D^2)$ vs $\mathcal{O}(D)$). The discrepancy may be due to our choice of optimality measure in~\eqref{def:approximate_kkt}. Lu and Zhou~\cite{luIterationComplexityFirstOrderAugmented2023} also obtained optimal complexity in terms of $L_f$, $\|A\|$, and $\varepsilon$, but similarly incurred additional $R_\Lambda$ and $D$ dependence when converting to a primal gap bound.



\subsection{Inexact Fast Augmented Lagrangian Method}\label{ssec:acc_alm}
While Algorithm~\ref{alg:al} is near-optimal, numerical experiments in Section~\ref{sec:numerical} below show that it is often outperformed by more advanced methods such as the linearized proximal ALM (LPALM)~\cite{liAcceleratedAlternatingDirection2019,ouyangAcceleratedLinearizedAlternating2015}, particularly when $\rho=\varepsilon^{-1}$. In this subsection, we utilize FLOrA (see Subsection~\ref{sec:flora}) to accelerate the outer loop, accelerating dual maximization and leading to a more performant algorithm. 


As discussed in the last subsection, accelerated methods typically require $\mathcal{O}(\varepsilon^{-2/3})$ iterations to guarantee an $\varepsilon$-small gradient mapping for a merely convex objective. However, adding a small, strongly convex perturbation improves the iteration complexity to $\mathcal{O}(\varepsilon^{-1/2})$. Since the criterion~\eqref{def:approximate_kkt} can be interpreted as finding $\varepsilon$-small primal/dual subgradients, we add strongly convex (concave) perturbations to the primal (dual) problems to improve the iteration-complexity. First, we define the perturbed primal problem

\begin{equation}\label{eq:ProbIntro_LC_pert}
	\tilde{\phi}_*:=\min_{x\in \R^n}\left\{ \phi(x)+\frac{\gamma_p}{2}\|x-x_0\|^2: Ax = b\right\}
\end{equation}
where $\phi(x)$ is as in~\eqref{eq:ProbIntro_LC} and $x_0\in\dom h$ is an arbitrary point. The associated Lagrangian is then
\begin{equation*}
    \L^{\gamma_p}(x,\lambda)=\phi(x) + \frac{\gamma_p}{2}\|x-x_0\|^2 + \inner{\lambda}{Ax-b},
\end{equation*}
with the augmented form
\begin{equation*}
    \Lp^{\gamma_p}(x,\lambda)=\phi(x) + \frac{\gamma_p}{2}\|x-x_0\|^2 + \inner{\lambda}{Ax-b} + \frac{\rho}{2}\|Ax-b\|^2.
\end{equation*}
Extending the idea of perturbations to the dual problem, we define the \emph{symmetrically perturbed} (augmented) Lagrangian
\begin{align*}
        \tL(x,\lambda):=\L^{\gamma_p}(x,\lambda)- \frac{\gamma_d}{2}\|\lambda-\lambda_0\|^2,\qquad \Lpt(x,\lambda):=\Lp^{\gamma_p}(x,\lambda)- \frac{\gamma_d}{2}\|\lambda-\lambda_0\|^2;
\end{align*}
where $\lambda_0\in\R^m$ is arbitrary. The associated perturbed dual problem is
\begin{equation}\label{def:pert_dual}
    \td(\lambda) = \underset{u\in\R^n}\min\tL(u,\lambda)=\underset{u\in\R^n}\min \L^{\gamma_p}(u,\lambda) - \frac{\gamma_d}{2}\|\lambda-\lambda_0\|^2,
\end{equation}
which is $\gamma_d$-strongly concave, hence $-\td(\lambda)$ is $\gamma_d$-strongly convex. Accordingly,~\eqref{def:pert_dual} has a unique maximizer $\tilde{\lambda}_*$ with $\dtL:=\|\tilde{\lambda}_*-\lambda_0\|$.
    By the definition of $\td$ and the superadditivity of $\min$, we have
    \begin{align}
    \td(\lambda)&\stackrel{\eqref{def:pert_dual}}=\min_{x\in\R^n}\left\{\L(x,\lambda)+\frac{\gamma_p}{2}\|x-x_0\|^2\right\} -\frac{\gamma_d}{2}\|\lambda-\lambda_0\|^2 \nn \\
        &\geq \min_{x\in\R^n}\L(x,\lambda)+\overbrace{\min_{x\in\R^n}\frac{\gamma_p}{2}\|x-x_0\|^2}^{=0} -\frac{\gamma_d}{2}\|\lambda-\lambda_0\|^2\stackrel{\eqref{eq:strong_duality}}{=}d(\lambda)-\frac{\gamma_d}{2}\|\lambda-\lambda_0\|^2. \label{ineq:pert_dual_ub}
    \end{align}
%
As before, we define $\Psi_{\lambda}^{\gamma_p}(\cdot)$ as the smooth, primal part of $\Lpt(\cdot)$
\[
\Psi_{\lambda}^{\gamma_p}(x):=f(x)+\frac{\gamma_p}{2}\|x-x_0\|^2 + \inner{\lambda}{Ax-b} + \frac{\rho}{2}\|Ax-b\|^2,
\]
which is $\gamma_p$-strongly convex and $(M_\rho+\gamma_p)$-smooth on $\R^n$.

For sufficiently small $\gamma_p$, an approximate solution to the perturbed problem implies an approximate solution to the original target problem. The following result is elementary, and similar lemmas have been used in previous works~\cite{lanIterationcomplexityFirstorderAugmented2016a}, and we therefore defer its proof to Appendix~\ref{appdx:deferred_alm}.

\begin{lemma}[Perturbation Solution]\label{lem:perturb_solution}
    Set ${\gamma_p}={\varepsilon}/({2D})$ and suppose that for the pair $(x,\lambda)$ there exists $v\in\partial \tL(\cdot,\lambda)(x)$ satisfying $\|v\|\leq \varepsilon/2$. Then, there exists $v'\in \partial \L(\cdot,\lambda)(x)$ satisfying $\|v'\|\leq \varepsilon$.
\end{lemma}

Primal perturbations have been leveraged in several existing works~\cite{lanIterationcomplexityFirstorderAugmented2016a,luIterationComplexityFirstOrderAugmented2023} with strong relations to proximal ALM schemes~\cite{meloProximalAugmentedLagrangian2024a}. Dual perturbation, on the other hand, has been less explored, while it has appeared in previous works to improve the iteration-complexity of the outer ALM loop~\cite{patrascuAdaptiveInexactFast2017}. However, as far as we are aware, the two ideas have not been used in tandem. As noted in~\cite{luIterationComplexityFirstOrderAugmented2023}, the distance from $\lambda_0$ to the optimum of the perturbed dual (defined as $\dtL$) depends implicitly on $\gamma_p^{-1}\propto\varepsilon^{-1}$. Interestingly, adding dual regularization removes this hidden dependence, as we will show in Lemma~\ref{lem:double_perturbation_distance} below.

Incorporating the acceleration scheme into the outer loop, as well as the perturbations, we obtain the I-FALM, shown in Algorithm~\ref{alg:dsc_aalm}.


\begin{algorithm}[H]
\caption{Inexact Fast Augmented Lagrangian Method}\label{alg:dsc_aalm}
\begin{algorithmic}
\REQUIRE given initial $ x_0\in \dom h $, $\rho>0$, $\gamma_d> 0$, $\varepsilon>0$, and $\varepsilon_0\geq \varepsilon$, set $B_0=0$, $\tau_0=1$, $\gamma_p = \varepsilon/(2D)$, and $\lam_0=\nu_0=0$, and choose $\sigma\in(0,1)$ such that $4\sigma\rho\varepsilon\leq 1$, and $\alpha\geq 0$ satisfying $\alpha< (1+\sqrt{\gamma_d\rho})^{-2}$.
\FOR{$k=0,1,\cdots$}
\STATE {\bf{1.}} Set $\varepsilon_k= (7\varepsilon_0\alpha^k+\sigma\rho\varepsilon^2)/8$ and compute
		\begin{gather}
		    \quad b_k=\frac{\rho\tau_k + \sqrt{\rho^2\tau_k^2+4\rho\tau_kB_{k}}}{2},\quad
            B_{k+1} = B_k + b_k,\quad \tau_{k+1}=\tau_k + b_k\gamma_d; \nn \\
            \tnu_k=\frac{B_k}{B_{k+1}}\lam_k + \frac{b_k}{B_{k+1}}\nu_k.\label{def:tnu}
		\end{gather}

\STATE {\bf 2.} Call Algorithm~\ref{alg:ACG} with 
\begin{equation}
\begin{gathered}
x_0=x_k,\quad \psi(\cdot)=\Lpt(\cdot,\tnu_k)+\frac{\varepsilon_k}{8D^2}\|\cdot-x_k\|^2,\quad g(\cdot)=\Psi_{\tnu_k}^{\gamma_p}+\frac{\varepsilon_k}{8D^2}\|\cdot-x_k\|^2,\\ L=M_\rho,\quad \mu=\gamma_p+\frac{\varepsilon_k}{4D^2},\label{def:setup_alm_acc}
\end{gathered}
\end{equation}
to find a point $\tx_k$ satisfying $\|\mathcal{G}^{{(2L+\mu)^{-1}}}_{\Lpt(\cdot,\tnu_k)}(\tx_k)\|\leq \varepsilon_k/(2D)$ and set 
\begin{equation*}
    x_{k+1}=\tx_k-(2L+\mu)^{-1}\mathcal{G}^{(2L+\mu)^{-1}}_{\Lpt(\cdot,\tnu_k)}(\tx_k).
\end{equation*}
\STATE {\bf 3.} Compute 
\begin{align}
    \lam_{k+1}=\tnu_k+\rho(Ax_{k+1}-b)\label{eq:lambda_acc}.
\end{align}
\STATE {\bf 4.} If $\|\mathcal{G}^{(2L+\mu)^{-1}}_{\Lpt(\cdot,\tnu_k)}(\tx_k)\|\leq {\varepsilon}/{4}$ and $\|Ax_{k+1}-b\|\leq \varepsilon$, then return $(x_{k+1},\lambda_{k+1})$ and \textbf{terminate}; otherwise, compute
    \begin{equation}
        \nu_{k+1}=\frac{1}{\tau_{k+1}}\left(\tau_{k}\nu_k+b_k\gamma_d\frac{\lam_{k+1}}{1+\gamma_d\rho}-\frac{b_{k}}{\rho}\left(\tnu_k-\frac{\lam_{k+1}}{1+\gamma_d\rho}\right)\right),\label{def:nu_k}
    \end{equation}
    and continue.
\ENDFOR
\end{algorithmic} 
\end{algorithm}

Unlike Algorithm~\ref{alg:al}, the $\varepsilon_k/(8D^2)$ addition to $\psi$ in Algorithm~\ref{alg:dsc_aalm} is not necessary to obtain $\varepsilon^{-1}$ complexity, as shown in Proposition~\ref{prop:alm_inner_complexity_sc} below. However, the increased strong convexity modulus improves empirical performance, particularly in early iterations when $\varepsilon_k\gg \varepsilon$.

Theorem~\ref{thm:acc_alm_complexity_1}, whose proof is deferred to Subsection~\ref{ssec:proofs_falm}, states our main complexity results for Algorithm~\ref{alg:dsc_aalm}. 

\begin{theorem}\label{thm:acc_alm_complexity_1}
Let $\varepsilon>0$ satisfy $\varepsilon\leq \|A\|^2 / L_f$. Choose $\rho= L_f/\|A\|^2$, $\varepsilon_0=\rho^{-1}$, $\sigma=1/4$, $\gamma_p=\varepsilon/(2D)$, $\gamma_d = \sigma^{3/2}\varepsilon/(\sqrt{3}\mathcal{R})$,
where 
\begin{equation}\label{def:R-R}
        \mathcal{R}:=\bardtL(1 + \sqrt{2\varepsilon_0C})\left(\frac{2}{\sqrt{1-\sigma}}+1\right), \quad \bardtL:=\max\{1,\|\tilde\lambda_*-\lambda_0\|\}, \quad C:=\sum_{i=0}^\infty B_{i+1}\alpha^i<\infty,
    \end{equation}
    with $\tilde\lambda_*$ defined as the unique maximizer of~\eqref{def:pert_dual}. Furthermore, assume $\alpha$ satisfies
    \begin{equation}\label{ineq:alpha_conds}
        \alpha\leq \min\left\{\frac{9}{10}(1+\sqrt{\rho\gamma_d})^{-2},\left(\frac{15D\varepsilon}{28\varepsilon_0}\right)^{\sqrt{\rho\varepsilon/D}}\right\}.
    \end{equation}
        Then, Algorithm~\ref{alg:dsc_aalm} finds an $\varepsilon$-primal-dual solution to~\eqref{eq:ProbIntro_LC} in
    \begin{equation}\label{cmplx:total_aalm}
        \tO\left(1+\frac{\sqrt{D^2+D \hat{R}_\Lambda}\|A\|}{\varepsilon}+\frac{\sqrt{D+\hat{R}_\Lambda}\|A\|}{\sqrt{   L_f\varepsilon}}+\frac{\sqrt{DL_f}}{\sqrt{\varepsilon}}\right)
    \end{equation}
    total ACG iterations, 
    where \begin{equation}\label{def:hatR}
        \hat{R}_\Lambda=\max\{1,\|\lambda_*-\lambda_0\|\},
    \end{equation} with $\lambda_*=\argmin\{\|\lambda-\lambda_0\|:\lambda\in\Lambda_*\}$ and $\Lambda_*$ is the set of maximizers to the dual problem~\eqref{eq:strong_duality}.

\end{theorem}

Combining the previous complexity results with Lemma~\ref{lem:pd_gap} in Appendix \ref{appdx:technical}, we can state complexity results for obtaining an $\varepsilon_g$-primal solution (see~\ref{def:primal_lc_sol}). As in the previous subsection, the proof is deferred to Appendix~\ref{appdx:deferred_alm}. 
\begin{corollary}\label{cor:aalm_complexity_pd}
    Let $\varepsilon_g>0$ satisfy $\varepsilon_g\leq  2\|A\|^2(D+\mathcal{R})/L_f$. Then, using the parameter settings of Theorem~\ref{thm:acc_alm_complexity_1} with $\varepsilon=\varepsilon_g/(2(D+\hat R_\Lambda))$, Algorithm~\ref{alg:dsc_aalm} finds an $\varepsilon_{g}$-primal solution to~\eqref{eq:ProbIntro_LC} in
        
        \begin{equation}
            \tO\left(\sqrt{\hat R_\Lambda+D}\left(
            1+\frac{\sqrt{D}(D+\hat R_\Lambda)\|A\|}{\varepsilon_{g}}+\frac{\sqrt{D+\hat R_\Lambda}\|A\|}{\sqrt{L_f\varepsilon_{g}}}+\frac{\sqrt{ DL_f}}{\sqrt{\varepsilon_{g}}}\right)\right)\label{cmplx:total_pd_alm_1}
        \end{equation}
        ACG iterations,
    where $\mathcal{R}$ is as in~\eqref{def:R-R} and $\hat R_\Lambda$ is as in~\eqref{def:hatR}.
\end{corollary}

Comparing to the lower bound in \cite[Theorem 3.1]{ouyangLowerComplexityBounds2021}, Algorithm~\ref{alg:dsc_aalm} is therefore optimal up to logarithmic terms in $\varepsilon_g$, $\|A\|$, and $L_f$. Again, however, it is sub-optimal in $\hat R_\Lambda$ ($\mathcal{O}(\hat R_\Lambda^{3/2})$ vs. $\mathcal{O}(\hat R_\Lambda)$) and $D$ ($\mathcal{O}(D^2)$ vs. $\mathcal{O}(D)$). As noted following Corollary~\ref{cor:alm_complexity_pd}, this is likely due to our method of analysis: reducing from approximate stationarity to a primal gap instead of directly bounding a gap function as in~\cite{ouyangAcceleratedLinearizedAlternating2015}.

\noindent\textbf{Remark. }To prove complexity for the case where $\varepsilon L_f\geq \|A\|^2$, we can simply rescale the objective $\phi(\cdot)$ by a scalar $\chi = \|A\|^2/(\varepsilon L_f) \le 1$ and find a $\chi\varepsilon$-solution using the settings of Theorem~\ref{thm:acc_alm_complexity_1}. Furthermore, we can show by elementary algebra that if $(x_*,\chi\lambda_*)$ is an optimal pair for the rescaled problem, then $(x_*,\lambda_*)$ is an optimal pair for the original problem~\eqref{eq:ProbIntro_LC}. Since $\|\chi\lambda\|\leq  \|\lambda\|$, the distance to $\Lambda_*$ from $\lambda_0=0$ does not increase, and we only need to consider the effects on $L_f$ and $\varepsilon$. Therefore, Theorem~\ref{thm:acc_alm_complexity_1} provides complexity bounds for~\eqref{eq:ProbIntro_LC} without loss of generality.

\noindent\textbf{Remark. }Focusing on the regime where $L_f\geq \varepsilon_g$ (true of most problems of interest), Corollary~\ref{cor:aalm_complexity_pd} implies that, omitting $\hat R_\Lambda$ and $D$ dependence, Algorithm~\ref{alg:dsc_aalm} has an iteration-complexity of
    $\tO\left(\sqrt{L_f/\varepsilon_g} + \|A\|/\varepsilon_g\right)$.
    In the case where $2(D+\mathcal{R})\|A\|^2\leq \varepsilon_g L_f$, following a similar rescaling argument as in the previous remark, we establish an iteration-complexity of
    $\tO\left(\sqrt{L_f/\varepsilon_g} + L_f/\|A\|\right)$.

\section{Frameworks for Generic Convex Optimization}\label{sec:framework}

In this section, we consider the generic optimization problem
\begin{equation}\label{eq:general_convex_prob}
\Phi_*=\min\{\Phi(x):x\in \R^n\},
\end{equation}
where $\Phi$ is proper, lower semi-continuous, and $\mu$-strongly convex for some $\mu\ge 0$ (with $\mu=0$ corresponding to the merely convex case). Motivated by IPP frameworks \cite{monteiro2010complexity,solodov1999hybrid1}, we propose two general schemes for solving \eqref{eq:general_convex_prob}: a baseline (unaccelerated) framework and an accelerated counterpart in the spirit of accelerated gradient methods.
Both frameworks rely on an abstract subroutine that prescribes the accuracy to which each proximal subproblem is solved. Under the assumption that such a subroutine is available, the main results of this section are the sub-optimality guarantees for the two frameworks. 
These guarantees will be used in Section~\ref{sec:proof} to establish the iteration-complexity bounds of Restarted ACG, I-ALM, and I-FALM, described in Sections~\ref{sec:primal} and~\ref{sec:dual}, which are special instances of the frameworks in primal and dual spaces.


\subsection{Lower Oracle Approximation Framework}\label{subsec:LORA}
This subsection presents the baseline framework, LOrA, given in Algorithm~\ref{alg:lora} below. We mark LOrA iterates (resp., parameters) with a superscript (resp., subscript) ``$\mathrm{L}$" to distinguish from those of specific implementations. For simplicity of presentation and analysis, we assume for this subsection that $\Phi$ is merely convex in \eqref{eq:general_convex_prob}.

\begin{algorithm}[H]
    \caption{LOrA Framework}\label{alg:lora}
    \begin{algorithmic}
        \REQUIRE given initial point $x^\lora_0 \in \dom\Phi$, $\sigma_\lora\in(0,1)$, $\lambda_\lora>0$, set $y^\lora_0=x^\lora_0$.
        \FOR{$k=0,1,\dots$}
        \STATE {\bf 1.} Choose $\delta^\lora_k>0$.
        \STATE {\bf 2.} Find $(y^\lora_{k+1}, \Gamma^\lora_k)\in \dom \Phi \times\CConv{1/\lam_\lora}{\dom\Phi}$ such that 
        \begin{equation}\label{ineq:Gamma}
            \Gamma_k^\lora(\cdot)\leq \Phi(\cdot)+\frac{1}{2\lambda_\lora}\|\cdot-x^\lora_k\|^2,
        \end{equation}
        \begin{equation}\label{ineq:lora_base}
            \|\lam_\lora \hat u^\lora_{k+1}\|^2 + 2\lambda_\lora\left[\Phi(y^\lora_{k+1}) +\frac{1}{2\lambda_\lora}\|y^\lora_{k+1}-x^\lora_k\|^2 - \Gamma^\lora_k(x^\lora_{k+1})\right]\le \sigma_\lora\|y^\lora_{k+1}-x^\lora_k\|^2 + 2\lambda_\lora\delta_k^\lora,
        \end{equation}
        where for some $\A^\lora_k\in(0,\infty]$,
        \begin{equation} \label{eq:x_lora}
            x_{k+1}^\lora=\underset{x\in\R^n}{\argmin}\left\{\Gamma^\lora_k(x) + \frac{1}{2\A^\lora_k}\|x-x^\lora_k\|^2\right\}, \quad \hat u^\lora_{k+1}=\frac{x^\lora_k-x^\lora_{k+1}}{\A^\lora_k}.
        \end{equation}
        \ENDFOR
    \end{algorithmic}
\end{algorithm}

LOrA can be understood as an iterative procedure of finding $\{y_{k+1}^\lora\}$ via certain subroutines satisfying \eqref{ineq:lora_base}, which provides sub-optimality guarantees as shown below in Theorem \ref{thm:lora_complexity}. From the perspective of IPP, $y_{k+1}^\lora$ is obtained by inexactly solving the proximal subproblem $\min_{z\in \R^n} \{\Phi(z)+\|z-x^\lora_k\|^2/(2\lambda_\lora)\}$, where the solution accuracy is controlled by the sum of a relative error and an absolute error as on the right-hand side of \eqref{ineq:lora_base}. Moreover, $\{x_{k+1}^\lora\}$ is an auxiliary sequence obtained as in \eqref{eq:x_lora} by (approximately) solving the surrogate function $\Gamma^\lora_k$, which approximates $\Phi+\|\cdot-x^\lora_k\|^2/(2\lambda_\lora)$ from below (see \eqref{ineq:Gamma}).

The following result formalizes the connection to IPP, that is, $y_{k+1}^\lora$ is an approximate solution to the proximal subproblem. For brevity of the main text, we defer the proof to Appendix~\ref{appdx:lora_framework}.
\begin{proposition}\label{prop:prox_point_lora}
    Let $\hat{x}_*^\lora$ be the minimizer of the proximal subproblem at iteration $k$,
    \[
    \hat{x}^\lora_* =\underset{z\in\R^n}\argmin \left\{\Phi(z)+\frac{1}{2\lambda_\lora}\|z-x^\lora_k\|^2\right\}.
    \]
    Then, $y_{k+1}^\lora$ obtained by the LOrA framework satisfies
    \[
        \Phi(y^\lora_{k+1})+\frac{1}{2\lambda_\lora}\|y^\lora_{k+1}-x^\lora_k\|^2 - \Phi(\hat{x}_*^\lora) - \frac{1}{2\lambda_\lora}\|\hat{x}_*^\lora-x^\lora_k\|^2 \le \frac{\sigma_\lora}{2\lambda_\lora}\|y^\lora_{k+1}-x^\lora_k\|^2+\delta^\lora_k.
    \]
\end{proposition}

LOrA is a generic framework for convex optimization that includes many first-order methods for solving smooth and nonsmooth problems as instances. 
In Subsection~\ref{ssec:proofs_alm}, we will show that I-ALM is an instance of LOrA. 
Here we provide two other concrete instances of the LOrA framework: the proximal gradient method and the modern proximal bundle (MPB) method~\cite{liang2021proximal,liang2024unified}.

\begin{example}[Proximal Gradient Method]
    Consider solving problem \eqref{eq:general_convex_prob} with  $\Phi(\cdot)=f(\cdot)+h(\cdot)$ where $f$ is convex and $L_f$-smooth and $h$ is convex and simple, and choose a stepsize $\eta\leq 1/L_f$, the proximal gradient method is
    \begin{equation}\label{def:pgm}
        x_{k+1}=\prox_{\eta h}(x_k-\eta \nabla f(x_k)).
    \end{equation}
    It is straightforward to verify that PGM is an instance of LOrA with the correspondence
    \begin{eqbox}
    \begin{equation}
    \begin{gathered}
        \Phi(\cdot)=f(\cdot)+h(\cdot),\quad\Gamma_k^\lora(\cdot) = \ell_f(\cdot;x_k)+h(\cdot)+\frac{1}{2\eta}\|\cdot-x_k\|^2,\quad \lambda_\lora = \eta,\quad \sigma_\lora=\eta L_f;\\
   \A^\lora_k=\infty,\quad\delta_k^\lora=0,\quad y_{k+1}^\lora=x_{k+1}^\lora=x_{k+1},\quad \hat{u}^\lora_{k+1}=0.\label{def:gd_corresp}
    \end{gathered}
    \end{equation}
    \end{eqbox}
\end{example}
    
We can easily verify \eqref{ineq:Gamma}-\eqref{eq:x_lora}.
    First, the inequality~\eqref{ineq:Gamma} follows trivially by the convexity of $f$. 
    Second, it is easy to verify that \eqref{def:pgm} indicates that
    \[
    x_{k+1}=\underset{x\in\R^n}{\argmin}\left\{\ell_f(x;x_k) +h(x)+ \frac{1}{2\eta}\|x-x_k\|^2\right\},
    \]
    which in view of the choices of $\Gamma_k^\lora$, $\A_k^\lora$, and $x^\lora_k$ in \eqref{def:gd_corresp} implies that the first relation in 
    \eqref{eq:x_lora} holds.
    Moreover, the second relation in \eqref{eq:x_lora} also simply follows from \eqref{def:gd_corresp}.
    Finally, we only need to show~\eqref{ineq:lora_base}. 
    Using the choices of $y_{k+1}^\lora$, $x_{k+1}^\lora$, and $\Gamma_k^\lora$ in \eqref{def:gd_corresp}, we have
    \begin{align*}
        &2\lambda_\lora\left[\Phi(y^\lora_{k+1})+\frac{1}{2\lambda_\lora}\|y^\lora_{k+1}-x^\lora_k\|^2-\Gamma^\lora_k(x^\lora_{k+1})\right]\\
        \stackrel{\eqref{def:gd_corresp}}=&2\eta[f(x_{k+1})-\ell_f(x_{k+1};x_k)]
        \leq \eta L_f\|x_{k+1}-x_k\|^2
        \stackrel{\eqref{def:gd_corresp}}=\sigma_\lora\|y^\lora_{k+1}-x^\lora_k\|^2.
    \end{align*}
    where the inequality follows from the $L_f$-smoothness of $f$ and the final identity follows from the choice of $\sigma_\lora$ in~\eqref{def:gd_corresp}.

\begin{example}[MPB Method]
    Consider the composite nonsmooth convex optimization problem $\min_x \{\phi(x):=f(x)+h(x)\}$ where $f$ is convex and Lipschitz continuous and $h$ is convex and simple. One method to solve such problem is the MPB method~\cite{liang2021proximal,liang2024unified}.
A key distinction of MPB from classical proximal bundle methods \cite{lemarechal1975extension,lemarechal1978nonsmooth,mifflin1982modification,wolfe1975method} lies in its incorporation of the IPP framework. MPB approximately solves a sequence of proximal subproblem of the form
\begin{equation}\label{eq:phi-lora}
    \min_{u \in \R^n} \left\{\psi(u) :=\phi(u)+\frac{1}{2 \lambda}\left\|u-x^\lora_{k}\right\|^2\right\}.
\end{equation}
Letting $x_0=x^\lora_{k}$ be the initial point of the subroutine for solving \eqref{eq:phi-lora}, MPB iteratively solves 
\begin{equation}\label{eq:PB-xj}
    x_j=\underset{u \in \R^n}{\argmin}\left\{\Gamma_j(u) + h(u) +\frac{1}{2 \lam}\|u-x_0\|^2\right\},
\end{equation}
where $\Gamma_j$ is a bundle model underneath $f$. Details about various models and a unifying framework underlying them are discussed in \cite{liang2024unified}. MPB keeps refining $\Gamma_j$ and solving $x_j$ through \eqref{eq:PB-xj}, until a criterion $t_j =\psi(\tx_j) - m_j \le \delta$ is met, where
\begin{equation}\label{def:tj}
	m_j = \Gamma_j(x_j) + h(x_j) +\frac{1}{2 \lam}\|x_j-x_0\|^2, \quad \tx_j \in \Argmin \{\psi(u): u \in\{x_0,x_1, \ldots, x_j\}\}.
	\end{equation}
As explained in \cite{liang2025primal}, the criterion $t_j \le \delta$ indicates that a primal-dual solution to \eqref{eq:phi-lora} with primal-dual gap bounded by $\delta$ is obtained. It also implies that $\tx_j$ is a $\delta$-solution to \eqref{eq:phi-lora} (see also \cite{liang2021proximal}).
Once the condition $t_j \le \delta$ is met, MPB updates the prox center to $x^\lora_{k+1}=x_j$, resets the bundle model $\Gamma_j$ from scratch, and proceeds to solve \eqref{eq:phi-lora} with $x^\lora_{k}$ replaced by $x^\lora_{k+1}$. We refer to iterations where the prox center $x^\lora_k$ is updated (and hence $t_j\leq \delta$) as serious steps. Otherwise, a step is referred to as a null step.
\end{example}

Let $j_k$ be the iteration corresponding to serious step $k$. We will show that for all serious steps $k \geq 0$, MPB is an instance of the LOrA framework with the correspondence

\begin{eqbox}
    \begin{equation}
    \begin{gathered}
        \Phi(\cdot)=\phi(\cdot),\quad\Gamma^\lora_k(\cdot) = \Gamma_{j_k}(\cdot) + h(\cdot) +\frac{1}{2 \lam}\|\cdot-x^\lora_{k}\|^2,\quad \lambda_\lora = \lambda,\quad \sigma_\lora=0;\\
   \A_k^\lora=\infty,\quad\delta_k^\lora=\delta,\quad y^\lora_{k+1}=\tx_{j_k},\quad x_{k+1}^\lora=x_{j_k},\quad \hat u^\lora_{k+1}=0.\label{def:mpb_corresp}
    \end{gathered}
    \end{equation}
\end{eqbox}

    Inequality~\eqref{ineq:Gamma} follows from the fact that $\Gamma_{j_k}\le f$. The first relation in \eqref{eq:x_lora} follows from \eqref{eq:PB-xj} with $j=j_k$ and the correspondence~\eqref{def:mpb_corresp}. The second relation in \eqref{eq:x_lora} trivially follows from \eqref{def:mpb_corresp}. Finally, in view of \eqref{def:mpb_corresp}, condition~\eqref{ineq:lora_base} is exactly the serious/null criterion $t_{j_k}\le \delta$, i.e.,
    \begin{equation*}
        \Phi(y^\lora_{k+1})+\frac{1}{2\lam_\lora}\|y^\lora_{k+1}-x^\lora_{k}\|^2-\Gamma^\lora_k(x^\lora_{k+1})\stackrel{\eqref{eq:phi-lora},\eqref{def:tj}}=\psi(\tx_{j_k})-m_{j_k}=t_{j_k}\leq \delta.
    \end{equation*}

The following theorem presents two sub-optimality guarantees of LOrA. Its proof is deferred to Appendix~\ref{appdx:lora_framework}.

\begin{theorem}\label{thm:lora_complexity}
    Let $X_*$ be the set of optimal solutions to \eqref{eq:general_convex_prob}. Define $R^\lora_0:=\|x^\lora_0-x_*\|=\min\{\|x^\lora_0-x\|:x\in X_*\}$ and $\bar{\delta}^\lora_{k}:=k^{-1}\sum_{i=0}^{k-1}\delta^\lora_i$. Suppose that $\A_k^\lora=\infty$ at all iterations. Then, for every $k\geq 1$, we have
     \begin{equation}
         \min_{1\leq i\leq k}\|y^\lora_i-x^\lora_{i-1}\|\leq \frac{R^\lora_0}{\sqrt{1-\sigma_\lora}\sqrt{k}} + \sqrt{\frac{2\lambda_\lora\bar{\delta}^\lora_{k}}{1-\sigma_\lora}}.\label{ineq:lora_a}
     \end{equation}
    Moreover,
     \begin{equation}
         \min_{1\leq i\leq k}\Phi(y^\lora_i)-\Phi(x_*)\leq \frac{\left(R^\lora_0\right)^2}{2\lambda_\lora k}+\bar{\delta}^\lora_{k}.\label{ineq:lora_b}
     \end{equation}

\end{theorem}

\subsection{Fast Lower Oracle Approximation Framework}\label{sec:flora}

This subsection presents the FLOrA framework for (strongly) convex minimization. 
We mark FLOrA iterates (resp., parameters) with a superscript (resp., subscript) ``$\mathrm{F}$" to distinguish from those of specific implementations.
Incorporating a Nesterov-type acceleration scheme into LOrA, FLOrA achieves better sub-optimality guarantees, which are comparable to other accelerated IPP frameworks \cite{marquesalvesVariantsAHPELargestep2022b,MonteiroSvaiterAcceleration}.

\begin{algorithm}[H]
    \caption{FLOrA Framework}
    \label{alg:flora}
    \begin{algorithmic}
        \REQUIRE given initial point $x^\flora_{0}\in\dom\Phi$, $\mu_\flora\geq 0$, $\sigma_\flora\in(0,1]$, $\lambda_\flora > 0$, $\tau_0=1$, $\delta^\flora_0\geq0$, set $B_0=0$ and $y^\flora_0=x^\flora_0$ and choose an $\alpha_\flora\geq 0$ satisfying $\alpha_\flora<(1+\sqrt{\lambda_\flora\mu_\flora})^{-2}$.
        \FOR{$k=0,1,\dots$}
        \STATE {\bf 1.} Set $\delta^\flora_{k}=\delta^\flora_0(\alpha_\flora)^k$ and compute
        \begin{gather}
            b_{k}=\frac{\lambda_\flora\tau_k+\sqrt{\lambda_\flora^2\tau_k^2+4\lambda_\flora \tau_k B_k}}{2}\label{eq:bk_update},\quad
            B_{k+1}=B_k + b_{k},\quad\tau_{k+1} = \tau_{k} + b_{k}\mu_\flora,\\
             \xt^\flora_k = \frac{B_k}{B_{k+1}}y^\flora_k+\frac{b_{k}}{B_{k+1}}x^\flora_k\label{eq:tx_flora}.
        \end{gather}
        \STATE {\bf 2.} Find $(\ty^\flora_{k+1}, \Gamma^\flora_k)\in \dom\Phi\times\underset{\mu_\flora+\lambda_\flora^{-1}}{\overline{\text{Conv}}}(\dom \Phi)$ such that
        \begin{equation}\label{ineq:Gamma_acc}
            \Gamma^\flora_k(\cdot)\leq \Phi(\cdot)+\frac{1}{2\lambda_\flora}\|\cdot-\tx^\flora_k\|^2,
        \end{equation}
        \begin{equation}\label{ineq:flora_cond}
            \|\lam_\flora \hat{u}^\flora_{k+1}\|^2 + 2\lam_\flora\left[\Phi(\ty^\flora_{k+1}) +\frac{1}{2\lambda_\flora}\|\ty^\flora_{k+1}-\tx^\flora_{k}\|^2 - \Gamma^\flora_k(z^\flora_{k+1})\right]\le \sigma_\flora\|\ty^\flora_{k+1}-\tx^\flora_{k}\|^2+2\lambda_\flora\delta_{k}^\flora,
        \end{equation}
        where for some $\A_k^\flora \in (0,\infty]$,
        \begin{align}
        z^\flora_{k+1}=\underset{v\in\R^n}\argmin\left\{\Gamma^\flora_k(v)+\frac{1}{2\A_k^\flora}\|v-\tx^\flora_k\|^2\right\},\quad  \hat{u}^\flora_{k+1}=\frac{\tx^\flora_k-z^\flora_{k+1}}{\A_k^\flora}\label{eq:zkp1_hatuk}.
        \end{align}
        \STATE {\bf 3.} Choose $y^\flora_{k+1}$ satisfying $\Phi(y^\flora_{k+1})\leq \Phi(\ty^\flora_{k+1})$ and compute
        \begin{equation}\label{def:xk_flora}
            u^\flora_{k+1}=\hat{u}^\flora_{k+1}+\frac{\tx^\flora_k-z^\flora_{k+1}}{\lambda_\flora},\quad x^\flora_{k+1}=\frac{1}{\tau_{k+1}}\left(\tau_kx^\flora_k+b_k\mu_\flora z^\flora_{k+1}-b_ku^\flora_{k+1}\right).
        \end{equation}
        \ENDFOR
    \end{algorithmic}
\end{algorithm}

Similar to LOrA, FLOrA does not specify the subroutine used in Step \textbf{2} to find $\ty_{k+1}^\flora$ and instead describes the requirement \eqref{ineq:flora_cond} on the subroutine to establish sub-optimality guarantees.
In addition to LOrA (which is close to Step~\textbf{2} in FLOrA), FLOrA employs the necessary computation (i.e., Steps~\textbf{1} and~\textbf{3}) for Nesterov's acceleration to enable better guarantees.
Hence, FLOrA is considered as a generic framework consisting of accelerated methods as special instances.
More specifically, we will show that ACG, Restarted ACG, and I-FALM are instances of FLOrA in Appendix~\ref{appdx:acg}, Subsection~\ref{ssec:proofs_restart}, and Subsection~\ref{ssec:proofs_falm}, respectively.

As an accelerated version of LOrA, FLOrA naturally admits an accelerated IPP interpretation. Prior accelerated IPP frameworks focus on more restricted settings: \cite{MonteiroSvaiterAcceleration} studies \eqref{eq:general_convex_prob} in the purely convex case, while \cite{marquesalvesVariantsAHPELargestep2022b} considers the composite form $\Phi=f+h$ with $f$ being convex and $h$ being strongly convex. In contrast, by including $z_{k+1}^\flora$  in \eqref{def:xk_flora}, FLOrA accommodates strong convexity in $\Phi$ directly, without imposing any particular decomposition or structural assumptions on $\Phi$.


The following theorem presents three sub-optimality guarantees of FLOrA. Its proof is deferred to Appendix~\ref{appdx:flora_framework}.

\begin{theorem}\label{thm:flora_complexity} 
    Let $X_*$ be the set of optimal solutions to \eqref{eq:general_convex_prob}. Define $R^\flora_0:=\|x^\flora_0-x_*\|=\min\{\|x^\flora_0-x\|:x\in X_*\}$. Then, for every $k\geq 0$,
        \begin{equation}\label{ineq:flora_a}
            \Phi(y^\flora_{k+1})-\Phi_*\leq \frac{\left(R_0^\flora\right)^2}{2B_{k+1}}+\frac{\delta^\flora_0C_\flora}{B_{k+1}},
        \end{equation}
        where $C_\flora:=\sum_{i=0}^\infty B_{i+1}(\alpha_\flora)^i < \infty$.
        Furthermore, if $\sigma_\flora < 1$, then for every $k\geq 0$, we have
        \begin{align}
            \|\ty^\flora_{k+1}-\tx^\flora_k\| &\leq\frac{\sqrt{\lambda_\flora}R_0^\flora+\sqrt{2\lambda_\flora\delta^\flora_0C_\flora}}{\sqrt{(1-\sigma_\flora)B_{k+1}}}, \label{ineq:flora_b} \\
            \min_{0\leq i\leq k}\|\ty^\flora_{i+1}-\tx^\flora_i\| &\leq\frac{\sqrt{\lambda_\flora}R_0^\flora+\sqrt{2\lambda_\flora\delta^\flora_0C_\flora}}{\sqrt{(1-\sigma_\flora) \sum_{i=1}^{k+1}B_{i}}}. \label{ineq:flora_c}
        \end{align}
\end{theorem}

    

\section{Proofs of Main Complexity Results} \label{sec:proof}

This section is devoted to the complexity analysis of the three algorithms studied in this paper: Restarted ACG, I-ALM, and I-FALM. The three subsections provide proofs of the main results for each method, namely Theorems~\ref{thm:main}, \ref{thm:al_baseline_complexity}, and \ref{thm:acc_alm_complexity_1}.

\subsection{Proof of Theorem~\ref{thm:main}}\label{ssec:proofs_restart}
To prove Theorem~\ref{thm:main}, we will first show that, assuming the call to Algorithm~\ref{alg:ACG} terminates in Step \textbf{2}, Algorithm~\ref{alg:restart} is an instance of the FLOrA framework with only relative error (i.e., $\delta^\flora_k=0$ for every $k \ge 0$). Theorem~\ref{thm:flora_complexity} will then imply the outer complexity. We then bound the number of ``inner'' iterations required by Algorithm~\ref{alg:ACG} in Step \textbf{2} to satisfy~\eqref{ineq:lora_restart_acg}. Combining the outer and inner complexities gives Theorem~\ref{thm:main}. For brevity of the main text, we defer the proofs of intermediate results to Appendix~\ref{appdx:deferred_restart}.

Let $j$ be the final (inner) iteration of Algorithm~\ref{alg:ACG} when invoked in Step~\textbf{2}, $y_j$ and $x_j$ be the final ACG iterates as in~\eqref{def:yj} and~\eqref{def:xj}, respectively, $s_j$ be as defined in~\eqref{def:hatxj_sj}, $A_j$ be the ACG scalar as in \eqref{def:acg_scalar}, and $\Theta_j$ be the aggregate function defined in~\eqref{def:Theta_acg}. Then we will show that Algorithm~\ref{alg:restart} is an instance of the FLOrA framework with the correspondence

\begin{eqbox}  
\begin{equation}
\begin{gathered}
    \Phi(\cdot)=\phi(\cdot),\quad\Gamma^\flora_k(\cdot) = \Theta_j(\cdot),\quad \A_k^\flora=A_j,\quad\delta^\flora_k=\alpha_\flora=0, \quad \mu_\flora = \mu_f,\quad \sigma_\flora=\sigma,\quad\lambda_\flora=\lambda;\\
    y^\flora_{k} = w_k,\,\,
    x^\flora_k = v_k,\,\, \tx^\flora_k=\tilde{v}_k,\,\, \ty^\flora_{k+1}=y_j,\,\, 
   z_{k+1}^\flora= x_j,\,\,  u^\flora_{k+1}=\frac{A_j+\lambda}{\lambda}s_j,\,\,  \hat u^\flora_{k+1}=s_j.\label{def:restart_corresp}
\end{gathered}
\end{equation}
\end{eqbox}

    \begin{lemma}\label{lem:restart_Gamma_eq}
        Assume for all $k\geq 0$, the call to Algorithm~\ref{alg:ACG} in Step \textbf{2} terminates. Then, with the correspondence~\eqref{def:restart_corresp}, Algorithm~\ref{alg:restart} is an instance of the FLOrA framework.
    \end{lemma}
    
    Since Algorithm~\ref{alg:restart} is an instance of FLOrA, the following ``outer'' sub-optimality guarantee holds by Lemma~\ref{lem:b_seq}(c) in Appendix~\ref{appdx:flora_framework} and Theorem~\ref{thm:flora_complexity} (see~\eqref{ineq:flora_a} with $\delta_0^\flora=0$).
    \begin{proposition}\label{prop:outer}
        For every $k\ge 1$, the function value gap $\phi(w_k) - \phi_*$ satisfies\[
        \phi(w_k) - \phi_* \le \min\left\{\frac{2R_0^2}{\lam k^2},\quad \frac{R_0^2}{2\lam}\left(1+\frac{\sqrt{\lambda\mu_f}}{2}\right)^{-2(k-1)}\right\},
        \]
        where $R_0$ denotes the distance from initial point $w_0$ to solution set $X_*$, i.e.,
\[
R_0=\|w_0-x_*\|= \min \{\|w_0-x\| : x \in X_*\}. 
\]
    \end{proposition}
    
    The following lemma provides a bound on the complexity of Algorithm~\ref{alg:ACG} to satisfy~\eqref{ineq:lora_restart_acg}, which connects the ``inner'' and ``outer'' perspectives.
    \begin{proposition}\label{prop:inner_to_outer}
        Assume that $\lambda\geq1/(L_f-\mu_f)$. Then in each call to ACG in Step \textbf{2} of Algorithm~\ref{alg:ACG}, after at most
        \begin{equation}\label{eq:bound}
          1+\left\lceil\min\left\{2\sqrt{10\sigma^{-1}\lambda(L_f-\mu_f)},\left(\frac{1}{4}+\frac{1}{2}\sqrt{\frac{2\lam (L_f-\mu_f)}{1+\lam \mu_f}}\right)\ln\left(10\sigma^{-1}\lambda(L_f-\mu_f)\right)\right\}\right\rceil.
        \end{equation}
        ACG iterations, the condition~\eqref{ineq:lora_restart_acg} is satisfied.
    \end{proposition}

\vspace{1em}

We are now ready to prove Theorem \ref{thm:main}.

\vspace{1em}

\noindent
{\bf Proof of Theorem \ref{thm:main}:}
   Recall that by Proposition~\ref{prop:inner_to_outer}, the inner complexity of Algorithm~\ref{alg:ACG} in Step \textbf{2} is
    \begin{equation}\label{cmplx:inner}
        \tO(1+\sqrt{\lambda(L_f-\mu_f)}), 
    \end{equation}
    and by Proposition~\ref{prop:outer}, the outer complexity of Algorithm~\ref{alg:restart} to find an $\varepsilon$-solution is
    \begin{equation}\label{cmplx:outer}
        \tO\left(1+\min\left\lbrace
            \frac{R_0}{\sqrt{\lambda \varepsilon} },
            \frac{1}{\sqrt{\mu_f\lambda}}\right\rbrace\right).
    \end{equation}

    a) In the case $\mu_f=0$, the outer complexity is $\tO(1+R_0/\sqrt{\lambda\varepsilon})$, hence the total complexity is
    \[
        \tO\left(\left(1+\sqrt{\lambda L_f}\right)\left(1+\frac{R_0}{\sqrt{\lambda \varepsilon}}\right)\right),
    \]
    which becomes $\tO\left(R_0\sqrt{L_f/\varepsilon}\right)$ under the assumption that $1/L_f\leq \lam \leq R_0^2/\varepsilon$.

    b) In the case $\mu_f>0$, the total complexity immediately follows from \eqref{cmplx:inner}, \eqref{cmplx:outer}, and the assumption that $1/(L_f-\mu_f)\leq \lam\leq\min\{1/\mu_f,R_0^2/\varepsilon\}$.
\QEDA

\subsection{Proof of Theorem~\ref{thm:al_baseline_complexity}}\label{ssec:proofs_alm}
We consider two perspectives to prove Theorem~\ref{thm:al_baseline_complexity}: ``inner'' and ``outer''.  First, we bound the number of inner iterations needed to satisfy the termination criterion $\|\mathcal{G}_{\Lp(\cdot,\lambda_k)}^{(2L+\mu)^{-1}}\|\leq\varepsilon_k/2$ in Step~\textbf{1} of Algorithm~\ref{alg:al}. The inner bound is a direct result of the ACG convergence rate in Lemma~\ref{lem:acg_convergence}, and we therefore defer the proof to Appendix~\ref{appdx:deferred_alm}.

\begin{proposition}\label{prop:alm_inner_complexity}
    The number of ACG iterations required in the call to Algorithm~\ref{alg:ACG} in Step \textbf{1} of Algorithm~\ref{alg:al} is at most
   \begin{equation}\label{cmplx:inner_alm}
       \tO\left(1+\frac{D(\sqrt{L_f}+\sqrt{\rho}\|A\|)}{\sqrt{\sigma\rho}\varepsilon}\right).
   \end{equation}
\end{proposition}

To bound the outer complexity, we will show that Algorithm~\ref{alg:al} implements the LOrA framework, i.e., Algorithm \ref{alg:lora}, with the correspondence

\begin{eqbox}
\begin{equation}
\begin{gathered}
    \Phi(\cdot)=-d(\cdot),\quad\Gamma_k^\lora(\cdot) = -\L(x_{k+1},\cdot)+\frac{1}{2\rho}\|\cdot-\lambda_k\|^2,\quad\mu_\lora=0,\quad \lambda_\lora=\rho,\quad \sigma_\lora=\sigma;\\
   \A_k^\lora=\infty,\quad \delta^\lora_k=\varepsilon_0\alpha^k,\quad y^\lora_{k} =x^\lora_k=\lambda_k,\quad \hat{u}^\lora_{k+1}=0,\quad \alpha_\lora=\alpha.
\end{gathered}\label{def:alm_corresp}
\end{equation}
\end{eqbox}

We begin by showing that our choice of $\Gamma^\lora_k$, $x^\lora_k$, and $\hat u^\lora_k$ satisfy~\eqref{ineq:Gamma} and~\eqref{eq:x_lora}. 

\begin{lemma}\label{lem:gamma_dual_props_alm}
    Consider the sequences $\{\lambda_{k+1}\}$ and $\{x_{k+1}\}$ produced by Algorithm~\ref{alg:al}. Then, for every $k\geq 0$, the following statements hold:
    \begin{enumerate}[label={\rm \alph*)}]
        \item for every $\nu\in\R^m$, we have
        \begin{equation}\label{ineq:gamma_dual_lb}
            -\L(x_{k+1},\nu)+\frac{1}{2\rho}\|\nu-\lambda_k\|^2\leq -d(\nu) + \frac{1}{2\rho}\|\nu-\lambda_k\|^2;
        \end{equation}
        
        \item \begin{equation}\label{def:lambda_min}
            \lambda_{k+1}=\underset{\nu\in\R^m}\argmin\left\{-\L(x_{k+1},\nu)+\frac{1}{2\rho}\|\nu-\lambda_k\|^2\right\}.
        \end{equation}
        
    \end{enumerate}
    Moreover, in light of \eqref{def:alm_corresp}, \eqref{ineq:gamma_dual_lb} and \eqref{def:lambda_min} correspond to \eqref{ineq:Gamma} and \eqref{eq:x_lora}, respectively.
\end{lemma}

We now prove that on all iterations of Algorithm~\ref{alg:al}, either the inequality~\eqref{ineq:lora_base} holds or $(x_{k+1},\lam_{k+1})$ is an $\varepsilon$-primal-dual solution to~\eqref{eq:ProbIntro_LC} and the outer loop terminates. Combined with Lemma~\ref{lem:gamma_dual_props_alm}, we therefore guarantee that, until termination, Algorithm~\ref{alg:al} is an instance of the LOrA framework (i.e., Algorithm~\ref{alg:lora}).

\begin{proposition}\label{prop:alm_contrapositive}
    For every $k\geq 0$, we have either
     \begin{align}
       -d(\lambda_{k+1})+\frac{1}{2\rho}\|\lambda_{k+1}-\lambda_k\|^2- \Gamma^\lora_k(\lambda_{k+1})
       \leq \varepsilon_0\alpha^k+\frac{\sigma}{2\rho}\|\lambda_{k+1}-\lambda_k\|^2,\label{ineq:alm_flora_ineq}
    \end{align}
    which corresponds to \eqref{ineq:lora_base} in view of \eqref{def:alm_corresp},
     or $(x_{k+1},\lambda_{k+1})$ is an $\varepsilon$-primal-dual solution to~\eqref{eq:ProbIntro_LC}.
\end{proposition}

\begin{proof}
    Setting $L$ and $\mu$ as in \eqref{def:setup_alm} and applying Proposition~\ref{prop:lag_subgradient} to \eqref{eq:ProbIntro_LC} with $(\tx, x^+, \lambda, \lambda^+)=(\tx_k, x_{k+1}, \lambda_k, \lambda_{k+1})$ and $\eta=(2L+\mu)^{-1}$, then the inner termination condition in Step \textbf{1}, i.e., $\|\mathcal{G}^{(2L+\mu)^{-1}}_{\Lp(\cdot,\lambda_k)}(\tx_k)\|\leq \varepsilon_k/(2D)$, implies that there exists some $v\in \partial\L(\cdot,\lambda_{k+1})(x_{k+1})$ satisfying $\|v\|\leq \varepsilon_k/D$. It then follows by the definition of $\Gamma_k$ in \eqref{def:alm_corresp}, the Cauchy-Schwarz inequality, and Assumption~\ref{assmp:constrained}(d) that we have
    \begin{align}
        -d(\lambda_{k+1})+&\frac{1}{2\rho}\|\lambda_{k+1}-\lambda_k\|^2-\Gamma^\lora_k(\lambda_{k+1})
        \stackrel{\eqref{def:alm_corresp}}=\L(x_{k+1},\lambda_{k+1})-d(\lambda_{k+1}) \nn \\
        \leq& \inner{v}{x_{k+1}-u(\lambda_{k+1})} \leq \|v\|\|x_{k+1}-u(\lambda_{k+1})\| \leq \varepsilon_k=\frac{\varepsilon_0\alpha^k}{2} + \frac{\sigma\rho\varepsilon^2}{2},\label{ineq:Ld}
    \end{align}
    where the first inequality follows from $v\in \partial\L(\cdot,\lambda_{k+1})(x_{k+1})$ and $u(\lambda_{k+1})=\underset{x\in\R^n}\argmin\L(x,\lambda_{k+1})$, and the last identity follows from the choice of $\varepsilon_k$ in Step \textbf{1}. 

    We now consider three cases to prove the proposition: 1) if $\|Ax_{k+1}-b\|\geq \varepsilon$, then we show \eqref{ineq:alm_flora_ineq} holds; 2) if $\|\mathcal{G}^{(2L+\mu)^{-1}}_{\Lp(\cdot,\lambda_k)}(\tx_k)\|\geq \varepsilon/2$, then we show \eqref{ineq:alm_flora_ineq} holds; and 3) if both conditions are violated, then we show that $(x_{k+1},\lambda_{k+1})$ is an $\varepsilon$-primal-dual solution to~\eqref{eq:ProbIntro_LC}.

    {\bf Case 1)} If $\|Ax_{k+1}-b\|\geq \varepsilon$, then $\rho\varepsilon^2\leq \rho\|Ax_{k+1}-b\|^2\stackrel{\eqref{def:lambda_alm}}=\rho^{-1}\|\lam_{k+1}-\lam_k\|^2$. Then \eqref{ineq:Ld} and \eqref{def:lambda_alm} imply that
    \[
        -d(\lambda_{k+1})+\frac{1}{2\rho}\|\lambda_{k+1}-\lambda_k\|^2-\Gamma^\lora_k(\lambda_{k+1}) \stackrel{\eqref{ineq:Ld}}\leq \frac{\varepsilon_0\alpha^k}{2}+\frac{\sigma\rho\varepsilon^2}{2}\leq \frac{\varepsilon_0\alpha^k}{2}+\frac{\sigma}{2\rho}\|\lambda_k-\lambda_{k+1}\|^2,
    \]
    which satisfies~\eqref{ineq:alm_flora_ineq}.

    {\bf Case 2)} If $\|\mathcal{G}^{(2L+\mu)^{-1}}_{\Lp(\cdot,\lambda_k)}(\tx_k)\|\geq \varepsilon/2$, then the termination condition of the inner solver, i.e., $\|\mathcal{G}^{(2L+\mu)^{-1}}_{\Lp(\cdot,\lambda_k)}(\tx_k)\|\leq \varepsilon_k/(2D)$, implies $\varepsilon_k\geq D\varepsilon$. By the condition on $\sigma$ in Algorithm~\ref{alg:al}, we have $\sigma \rho\varepsilon\leq D/2$. Then,
    \[D\varepsilon\leq \varepsilon_k=\frac{\varepsilon_0\alpha^k}{2}+\frac{\sigma\rho\varepsilon^2}{2}\leq \frac{\varepsilon_0\alpha^k}{2}+\frac{\varepsilon D}{4},\]
    which implies \begin{equation}
        \frac{\varepsilon_0\alpha^k}{2}\geq \frac{3D\varepsilon}{4}\geq \frac{\sigma\rho\varepsilon^2}{2}.\label{ineq:err_bound_case_2}
    \end{equation}
    Thus, by~\eqref{ineq:Ld}, we obtain
    \[
        -d(\lambda_{k+1})+\frac{1}{2\rho}\|\lambda_{k+1}-\lambda_k\|^2-\Gamma^\lora_k(\lambda_{k+1})\stackrel{\eqref{ineq:Ld}}\leq  \frac{\varepsilon_0\alpha^k}{2}+\frac{\sigma\rho\varepsilon^2}{2}\stackrel{\eqref{ineq:err_bound_case_2}}\leq \varepsilon_0\alpha^k,
    \]
    which satisfies~\eqref{ineq:alm_flora_ineq}.

    {\bf Case 3)} We now consider the third case, where the conditions for the first two cases fail to hold, that is, Algorithm \ref{alg:al} terminates in Step \textbf{3}. Now that $\|\mathcal{G}^{(2L+\mu)^{-1}}_{\Lp(\cdot,\lambda_k)}(\tx_k)\|\leq \varepsilon/2$, then Proposition~\ref{prop:lag_subgradient}  with $(\tx, x^+, \lambda, \lambda^+)=(\tx_k, x_{k+1}, \lambda_k, \lambda_{k+1})$ and $\eta=(2L+\mu)^{-1}$ implies that there exists a $v\in\partial \L(\cdot,\lambda_{k+1})(x_{k+1})$ satisfying $\|v\|\leq \varepsilon$. Hence, $\|Ax_{k+1}-b\|\leq \varepsilon$ and $\|\mathcal{G}^{(2L+\mu)^{-1}}_{\Lp(\cdot,\lambda_k)}(\tx_k)\|\leq \varepsilon/2$ indicates that $(x_{k+1},\lambda_{k+1})$ is an $\varepsilon$-primal-dual solution by~\eqref{def:approximate_kkt}.

    Therefore, we complete the proof.
\end{proof}

\vspace{1em}

We are now ready to prove Theorem \ref{thm:al_baseline_complexity}.

\vspace{1em}

\noindent
\textbf{Proof of Theorem~\ref{thm:al_baseline_complexity}:}
Recall from Proposition~\ref{prop:alm_inner_complexity} that
    the inner complexity to satisfy the inner termination condition $\|\mathcal{G}^{(2L+\mu)^{-1}}_{\Lp(\cdot,\lambda_k)}(\tx_k)\|\leq \varepsilon_k/(2D)$ is as in \eqref{cmplx:inner_alm}.

   Observe that $\sigma= 1/2$ and $\rho=\varepsilon^{-1}$ satisfy the requirement $2\sigma\rho=1/\varepsilon\leq D/\varepsilon$ (see initialization in Algorithm \ref{alg:al}) in view of Assumption~\ref{assmp:constrained}(d). Set $L$ and $\mu$ as in~\eqref{def:setup_alm}. By the inner termination condition, our choice $\varepsilon_0=\varepsilon$, and the condition $\rho\sigma\leq D/(2\varepsilon)$, for all iterations $k$ we have
    \[
        \|\mathcal{G}_{\Lp(\cdot,\lambda_k)}^{(2L+\mu)^{-1}}(\tx_k)\|\leq \frac{\varepsilon_k}{2D}\leq \frac{\varepsilon_0}{4D}+\frac{\rho\sigma\varepsilon^2}{4D}\leq \frac{\varepsilon}{2}.
    \]
    Then, Proposition~\ref{prop:lag_subgradient} applied to problem~\eqref{eq:ProbIntro_LC} with $(\tx, x^+, \lambda, \lambda^+)=(\tx_k, x_{k+1}, \lambda_k, \lambda_{k+1})$ and $\eta=(2L+\mu)^{-1}$ implies that for all iterations $k\geq0$, there exists a $v\in \partial\L(\cdot,\lambda_{k+1})(x_{k+1})$ satisfying $\|v\|\leq \varepsilon$.

    By Lemma~\ref{lem:gamma_dual_props_alm} and Proposition~\ref{prop:alm_contrapositive}, Algorithm~\ref{alg:al} is an instance of the LOrA framework (i.e., Algorithm~\ref{alg:lora}) until termination with the correspondence~\eqref{def:alm_corresp}. Then using Theorem~\ref{thm:lora_complexity} with $R^\lora_0=R_\Lambda$, we have \[
    \rho\min_{1\leq i\leq k}\|Ax_{i}-b\| \stackrel{\eqref{def:lambda_alm}}=\min_{1\leq i\leq k}\|\lam_{i}-\lam_{i-1}\|\stackrel{\eqref{def:alm_corresp}}=\min_{1\leq i\leq k}\|y^\lora_{i}-x^\lora_{i-1}\|\stackrel{\eqref{ineq:lora_a},\eqref{def:alm_corresp}}\leq \frac{R_\Lambda}{\sqrt{1-\sigma}\sqrt{k}} + \sqrt{\frac{2\rho\bar{\delta}^\lora_{k}}{1-\sigma}}.
    \]
    It follows from the definition of $\bar{\delta}_k^\lora$ in Theorem~\ref{thm:lora_complexity} and $\delta_k^\lora=\varepsilon_0\alpha^k=\varepsilon\alpha^k$ from \eqref{def:alm_corresp} that
    \[
        \bar{\delta}^\lora_k=\frac{\sum_{i=0}^{k-1}\varepsilon \alpha^i}{k}\leq \frac{\varepsilon}{(1-\alpha)k}.
    \]
    The above two inequalities immediately imply the outer complexity to guarantee near feasibility $\min_{1\leq i\leq k}\|Ax_{i}-b\|\le \varepsilon$ is 
    \begin{equation}\label{cmplx:outer_alm}
        \mathcal{O}\left(1 + \frac{R_\Lambda^2 + \rho\varepsilon}{(1-\sigma)\rho^2\varepsilon^2}\right)
    \end{equation}
    outer iterations, since at least one outer iteration is needed to ensure stationarity.

    Combining the inner complexity from~\eqref{cmplx:inner_alm} and outer complexity from \eqref{cmplx:outer_alm}, 
   and substituting $\rho=\varepsilon^{-1}$, we obtain the total complexity as in \eqref{cmplx:total_alm}.\QEDA


\subsection{Proof of Theorem~\ref{thm:acc_alm_complexity_1}}\label{ssec:proofs_falm}
We prove Theorem~\ref{thm:acc_alm_complexity_1} by following the same approach as in Subsection~\ref{ssec:proofs_alm}. First, we will provide a bound on the inner complexity in each call to Algorithm~\ref{alg:ACG} in Step \textbf{2}, then we will bound the outer complexity by proving that Algorithm~\ref{alg:dsc_aalm} is an instance of FLOrA (i.e., Algorithm~\ref{alg:flora}). However, the inclusion of dual perturbations requires more care in the outer analysis than in the prior subsection. Our choice of the auxiliary point $z^\flora_{k+1}$ in the FLOrA analysis will play a crucial role in our argument.

Before providing complexity bounds, we show that our primal-dual perturbations in~\eqref{def:pert_dual} do not add dependence on $\varepsilon^{-1}$, as observed for primal-only perturbations in~\cite[Appendix A]{luIterationComplexityFirstOrderAugmented2023}. Instead, when both $\gamma_p$ and $\gamma_d$ are $\mathcal{O}(\varepsilon)$, the dependence on $\varepsilon^{-1}$ disappears. The proof of the following lemma is deferred to Appendix \ref{proof:double_perturbation_distance}.

\begin{lemma}\label{lem:double_perturbation_distance}
Let $\Lambda_*=\{\lambda:d(\lambda)=d_*\}$ be the set of optimal multipliers for the original problem~\eqref{eq:ProbIntro_LC}. Define $R_{\Lambda}:=\|\lambda_0-\lambda_*\|=\min\{\|\lambda_0-\lambda\|:\lambda\in\Lambda_*\}$ and $\dtL:=\|\tilde\lambda_*-\lambda_0\|$ where $\tilde\lambda_*$ is the unique minimizer of $-\td(\cdot)$ and $-\td(\cdot)$ is as in~\eqref{def:pert_dual}. Suppose $\gamma_p=\varepsilon/(2D)$ and $\gamma_d= C_0\varepsilon/(\dtL)$ for some $C_0>0$, then we have
    \begin{equation}
        \dtL\leq R_{\Lambda}+\frac{D}{4C_0}.\label{ineq:pert_distance_propto}
    \end{equation}
\end{lemma}

With the perturbed bound proven, we proceed with our proof of Theorem~\ref{thm:acc_alm_complexity_1} by bounding the inner complexity in Step \textbf{2}. The proof is nearly identical to that of Proposition~\ref{prop:alm_inner_complexity}, and is likewise deferred to Appendix~\ref{proof:alm_inner_complexity_sc}.

\begin{proposition}\label{prop:alm_inner_complexity_sc}
    Choosing $\gamma_p=\varepsilon/(2D)$, then the number of ACG iterations required in the call to Algorithm~\ref{alg:ACG} in Step \textbf{2} of Algorithm~\ref{alg:dsc_aalm} is at most
   \begin{equation}
       \tO\left(1+\frac{\sqrt{D}(\sqrt{L_f}+\sqrt{\rho}\|A\|)}{\sqrt{\varepsilon}}\right). \label{cmplx:inner_aalm}
       \end{equation}
\end{proposition}

\noindent
We now switch to the ``outer'' perspective. Define the point
\begin{equation}\label{def:lambda_hat}
    \hat{\lambda}_{k+1}=\frac{\lambda_{k+1}}{1+\gamma_d\rho}
\end{equation} and the function
\begin{align}
    \Gamma_k^\lambda(\cdot)=-\tL(x_{k+1},\lambda_{k+1})+\frac{1}{2\rho}\|\lambda_{k+1}-\tnu_k\|^2
    +\inner{\gamma_d\lambda_{k+1}}{\cdot-\lambda_{k+1}}+\frac{1+\gamma_d\rho}{2\rho}\|\cdot-\lambda_{k+1}\|^2,\label{def:dual_Gamma_sc}
\end{align}
which is a $(\rho^{-1}+\gamma_d)$-strongly convex approximation of $-\tL(x_{k+1},\cdot)+\|\cdot-\tnu_k\|^2/(2\rho)$ at $\lambda_{k+1}$.

We will show that Algorithm~\ref{alg:dsc_aalm} is an instance of the FLOrA framework with the correspondence

\begin{eqbox}
\begin{equation}\label{def:acc_alm_corresp}
    \begin{gathered}
    \Phi(\cdot)=-\td(\cdot), \,\, \Gamma^\flora_k(\cdot) = \Gamma_k^{\lambda}(\cdot),\,\, \A_k^\flora=\infty,\,\,  \alpha_\flora=\alpha, \,\, \mu_\flora = \gamma_d,\,\, \sigma_\flora=\sigma,\,\, \lambda_\flora=\rho;\\
    \delta^\flora_k=\varepsilon_0\alpha^k,\,\, y_{k}^\flora = \ty^\flora_k=\lambda_k,\,\,
   z^\flora_{k}= \hat{\lambda}_k,\,\,
    x^\flora_k = \nu_k, \,\, \tx^\flora_k=\tnu_k,\,\, u^\flora_{k}=\rho^{-1}({\tnu_{k-1}-\lambda_{k}}),\,\, \hat u^\flora_{k}=0.
    \end{gathered}
\end{equation}
\end{eqbox}
First, we show that the conditions~\eqref{ineq:Gamma_acc}, \eqref{eq:zkp1_hatuk}, and \eqref{def:xk_flora} are satisfied, along with a summability bound related to the absolute error sequence $\{\varepsilon_0\alpha^k\}$.

\begin{lemma}\label{lem:gamma_bounds_dual_sc}
    The following statements hold for every $k\ge 0$,
    \begin{enumerate}[label={\rm\alph*)}]
        \item for every $\nu\in\R^m$ 
        \begin{equation*}
            \Gamma_k^\lambda(\nu)\leq -\td(\nu) + \frac{1}{2\rho}\|\nu-\tnu_k\|^2;
        \end{equation*}
        
        \item $\hat\lambda_{k+1}=\underset{\nu\in\R^m}\argmin\Gamma_k^\lambda(\nu)$ and
        \begin{equation}\label{eqn:min_GammaL_value}
            \min_{\nu\in\R^m}\Gamma_k^\lambda(\nu)=-\tL(x_{k+1},\lambda_{k+1})+\frac{1}{2\rho}\|\lambda_{k+1}-\tnu_k\|^2-\frac{\gamma_d^2\rho}{2(1+\gamma_d\rho)}\|\lambda_{k+1}\|^2.
        \end{equation}

        \item letting $u_{k+1}=\rho^{-1}(\tnu_{k}-\hat\lambda_{k+1})$, we can rewrite~\eqref{def:nu_k} as
        \[
            \nu_{k+1}=\frac{1}{\tau_{k+1}}\left(\tau_k\nu_k+b_k\gamma_d\hat\lambda_{k+1}-b_ku_{k+1}\right).
        \]
        \item defining $\beta=\sqrt{\alpha} (1+\sqrt{\rho\gamma_d})<1$, we have $C\leq \rho(1-\beta)^{-4}<\infty$, where $C$ is as in~\eqref{def:R-R}.
    \end{enumerate}
    Moreover, in light of \eqref{def:acc_alm_corresp}, statements a), b), and c) correspond to \eqref{ineq:Gamma_acc}, \eqref{eq:zkp1_hatuk}, and \eqref{def:xk_flora}, respectively, and therefore \eqref{def:tnu} is equivalent to \eqref{eq:tx_flora}.
\end{lemma}


Analyzing Algorithm~\ref{alg:dsc_aalm} as an instance of Algorithm~\ref{alg:flora} now requires that we show~\eqref{ineq:flora_cond} holds with the correspondence~\eqref{def:acc_alm_corresp}. The following proposition is the analogue of Proposition~\ref{prop:alm_contrapositive} from the prior subsection, retaining the same ``three-case'' structure while adapting the analysis to the dual perturbations.
\begin{proposition}\label{prop:OR}
     Suppose $\gamma_d>0$ satisfies
    \begin{equation}
        \gamma_d\leq\min\left\{\frac{\sqrt{\sigma}}{2\sqrt{3}\rho},\frac{\sqrt{\sigma}\varepsilon}{4\sqrt{3}\mathcal{R}}\right\},\label{ineq:gamma_d_req}
    \end{equation}
    where $\mathcal{R}$ is as in~\eqref{def:R-R}. Then, for every $k\geq 0$, we have either
     \begin{equation}\label{ineq:key}
         -\td(\lambda_{k+1})+\frac{1}{2\rho}\|\lambda_{k+1}-\tnu_k\|^2-\Gamma_{k}^\lambda(\hat\lambda_{k+1})\leq \frac{\sigma}{2\rho}\|\lambda_{k+1}-\tnu_k\|^2+\varepsilon_0\alpha^k,
     \end{equation}
     which corresponds to \eqref{ineq:flora_cond} in view of \eqref{def:acc_alm_corresp},
     or $(x_{k+1},\lambda_{k+1})$ is an $\varepsilon$-primal-dual solution to~\eqref{eq:ProbIntro_LC}.
\end{proposition}
\begin{proof}
     We prove the proposition by induction. Throughout, let $L$ and $\mu$ be as in~\eqref{def:setup_alm_acc}.
     First, we note that Lemma~\ref{lem:gamma_bounds_dual_sc}(b) and the definition of $\Gamma_k^\lam$ in \eqref{def:dual_Gamma_sc} imply that
    \begin{equation}
        -\td(\lambda_{k+1})+\frac{1}{2\rho}\|\lambda_{k+1}-\tnu_k\|^2-\Gamma_{k}^\lambda(\hat\lambda_{k+1})\stackrel{\eqref{eqn:min_GammaL_value}}=\tL(x_{k+1},\lambda_{k+1})-\td(\lambda_{k+1})+\frac{\gamma_d^2\rho}{2(\gamma_d\rho+1)}\|\lambda_{k+1}\|^2.\label{eq:Ld_sc}
    \end{equation}
     Applying Proposition~\ref{prop:lag_subgradient} to~\eqref{eq:ProbIntro_LC_pert} with $(\tx, x^+, \lambda, \lambda^+)=(\tx_k, x_{k+1}, \tnu_k, \lambda_{k+1})$, $\eta=(2L+\mu)^{-1}$, and $f$ replaced by 
     $f+{\gamma_p}\|\cdot-x_0\|^2/2$, then the inner termination condition in Step \textbf{2}, i.e., $\|\mathcal{G}_{\Lpt(\cdot,\tnu_k)}^{(2L+\mu)^{-1}}(\tx_k)\|\leq \varepsilon_k/2D$, implies that there exists a subgradient $v\in\partial\tL(\cdot,\lambda_{k+1})(x_{k+1})$ satisfying $\|v\|\leq \varepsilon_k/D$. It then follows from the Cauchy-Schwarz inequality and Assumption~\ref{assmp:constrained}(d) that
     \begin{equation}\label{ineq:lag_error_bound}
         \tL(x_{k+1},\lambda_{k+1})-\td(\lambda_{k+1})\leq \inner{v}{x_{k+1}-\tilde u(\lambda_{k+1})}\leq \|v\|D\leq \varepsilon_k =\frac{7\varepsilon_0\alpha^k}{8}+\frac{\sigma\rho\varepsilon^2}{8},
     \end{equation}
    where the first inequality follows from $v\in \partial\tL(\cdot,\lambda_{k+1})(x_{k+1})$ and $\tilde u(\lambda_{k+1})=\underset{x\in\R^n}\argmin\tL(x,\lambda_{k+1})$, and the last identity follows from the choice of $\varepsilon_k$ in Step \textbf{1}.

    By our choice $\lambda_0=\nu_0=0$, for the base case we have $\tnu_0=0$ and so $\|\lambda_{1}\|=\|\tnu_0-\lambda_1\|$. It follows from the initialization in Algorithm \ref{alg:dsc_aalm} that $\sigma\rho\varepsilon^2 \leq \varepsilon/4\leq \varepsilon_0/4$, which together with \eqref{ineq:lag_error_bound} implies that
    \[
    \tL(x_{1},\lambda_{1})-\td(\lambda_{1})+\frac{\gamma_d^2 \rho}{2(\gamma_d\rho+1)}\|\lambda_{1}\|^2\stackrel{\eqref{ineq:lag_error_bound}}\leq \frac{7\varepsilon_0}{8}+\frac{\sigma\rho\varepsilon ^2}{8}+\frac{\sigma}{2\rho}\|\tnu_0-\lambda_1\|^2\leq \varepsilon_0+\frac{\sigma}{2\rho}\|\tnu_0-\lambda_1\|^2.
    \]
    In view of \eqref{eq:Ld_sc}, the above inequality proves \eqref{ineq:key} with $k=0$, which is the base case of the proposition.
    
    Now we assume the proposition holds for iterations $0\leq n\le k-1$.
    Without loss of generality, we assume that \eqref{ineq:key} holds with $k$ replaced by $k-1$, otherwise $(x_k,\lambda_k)$ is already an $\varepsilon$-primal-dual solution to~\eqref{eq:ProbIntro_LC}.
    Hence, \eqref{ineq:key} and Lemma~\ref{lem:gamma_bounds_dual_sc} imply that Algorithm~\ref{alg:dsc_aalm} is an instance of the FLOrA framework (i.e., Algorithm~\ref{alg:flora}) under the correspondence~\eqref{def:acc_alm_corresp}. Then, using Lemma~\ref{lem:distance_bound_flora} with $\mathcal{R}_\flora=\mathcal{R}$ (which is defined in \eqref{def:R-R}) and the correspondence \eqref{def:acc_alm_corresp}, we have
    \begin{equation}
        \|\tnu_k-\tilde\lambda_*\| \stackrel{\eqref{def:acc_alm_corresp}}= \|\tx_k^\flora-x_*\|\stackrel{\eqref{ineq:tx_dist_bound}}\leq \mathcal{R}.\label{ineq:tnu_k_bound}
    \end{equation}
    It thus follows from the triangle inequality and the Cauchy-Schwarz inequality that
    \begin{align*}
        \|\lambda_{k+1}\|^2
       & \leq (\|\lambda_{k+1}-\tnu_k\| + \|\tnu_k-\tilde\lambda_*\| + \|\tilde\lambda_*\|)^2 \\
       & \leq 3 (\|\lambda_{k+1}-\tnu_k\|^2 + \|\tnu_k-\tilde\lambda_*\|^2 + \|\tilde\lambda_*\|^2) \stackrel{\eqref{ineq:tnu_k_bound}}\leq 3 \left(\|\lambda_{k+1}-\tnu_k\|^2 + 2\mathcal{R}^2\right),
    \end{align*}
    where the last inequality is due to \eqref{ineq:tnu_k_bound} and the fact that $\|\tilde\lam_*\|=\|\tilde\lam_*-\lam_0\|\leq \bardtL\leq \mathcal{R}$ in view of \eqref{def:R-R}.
    The above inequality and the requirement on $\gamma_d$ in~\eqref{ineq:gamma_d_req} further imply that
    \begin{equation}
        \frac{\gamma_d^2\rho}{2(\gamma_d\rho+1)}\|\lambda_{k+1}\|^2
        \leq \frac{3 \gamma_d^2 \rho}{2}\left(\|\lambda_{k+1}-\tnu_k\|^2 + 2\mathcal{R}^2\right) \stackrel{\eqref{ineq:gamma_d_req}}\leq \frac{\sigma}{8\rho}\|\lambda_{k+1}-\tnu_k\|^2 + \frac{\sigma\rho\varepsilon^2}{16}. \label{ineq:lamb_bound_lambk}
    \end{equation}
    Putting together \eqref{eq:Ld_sc}, \eqref{ineq:lag_error_bound}, and \eqref{ineq:lamb_bound_lambk}, we obtain
    \begin{equation}
        -\td(\lambda_{k+1})+\frac{1}{2\rho}\|\lambda_{k+1}-\tnu_k\|^2-\Gamma_{k}^\lambda(\hat\lambda_{k+1}) \leq \frac{7\varepsilon_0\alpha^k}{8}+\frac{3 \sigma\rho\varepsilon^2}{16} +\frac{\sigma}{8\rho}\|\lambda_{k+1}-\tnu_k\|^2.\label{ineq:inductive_step_bound}
    \end{equation}
    
    We now consider three cases to prove the proposition: 1) if $\|Ax_{k+1}-b\|\geq \varepsilon$, then we show that~\eqref{ineq:key} holds; 2) if $\|\mathcal{G}^{(2L+\mu)^{-1}}_{\Lpt(\cdot,\tnu_k)}(\tx_k)\|\geq \varepsilon/4$, then we show \eqref{ineq:key} holds; and 3) if both conditions are violated, then we show that $(x_{k+1},\lambda_{k+1})$ is an $\varepsilon$-primal-dual solution to~\eqref{eq:ProbIntro_LC}.

    {\bf Case 1)} Since $\|Ax_{k+1}-b\|\geq \varepsilon$, it follows from \eqref{eq:lambda_acc} that
    \[
        \frac{3 \sigma\rho\varepsilon^2}{16}\leq\frac{3 \sigma\rho}{16}\|Ax_{k+1}-b\|^2\stackrel{\eqref{eq:lambda_acc}}=\frac{3 \sigma}{16\rho}\|\lambda_{k+1}-\tnu_k\|^2,
    \]
    which together with \eqref{ineq:inductive_step_bound} implies that
     \[
         -\td(\lambda_{k+1})+\frac{1}{2\rho}\|\lambda_{k+1}-\tnu_k\|^2-\Gamma_{k}^\lambda(\hat\lambda_{k+1})
         \leq \frac{7\varepsilon_0\alpha^k}{8}+\frac{5\sigma}{16\rho}\|\lambda_{k+1}-\tnu_k\|^2.
     \]
    Hence, \eqref{ineq:key} immediately follows.

     {\bf Case 2)} Since $\|\mathcal{G}_{\Lpt(\cdot,\tnu_k)}^{(2L+\mu)^{-1}}(\tx_k)\|\geq \varepsilon/4$, the inner termination condition in Step \textbf{2} of Algorithm~\ref{alg:dsc_aalm} implies that $\varepsilon_k\geq D\varepsilon/2$.  It thus follows from the choice of $\varepsilon_k$ in Step \textbf{1} that
     \[
         \frac{D\varepsilon}{2}\leq \varepsilon_k=\frac{7\varepsilon_0\alpha^k}{8}+\frac{\sigma\rho\varepsilon^2}{8}\leq \frac{7\varepsilon_0\alpha^k}{8}+\frac{D\varepsilon}{32},
     \]
    where the inequality is due to $\sigma\rho\leq 1/(4 \varepsilon) \leq D/(4 \varepsilon)$ by the initialization of Algorithm~\ref{alg:dsc_aalm} and Assumption~\ref{assmp:constrained}(d).
    The above inequality thus indicates that
     \[
         \frac{7\varepsilon_0\alpha^k}{8}\geq \frac{15D \varepsilon}{32}\geq \frac{15\sigma\rho\varepsilon^2}{8}\implies \frac{3\sigma\rho\varepsilon^2}{16}\leq \frac{7\varepsilon_0\alpha^k}{80}.
     \]
     Plugging the above bound into \eqref{ineq:inductive_step_bound}, we obtain
     \[
         -\td(\lambda_{k+1})+\frac{1}{2\rho}\|\lambda_{k+1}-\tnu_k\|^2-\Gamma_{k}^\lambda(\hat\lambda_{k+1}) 
         \leq \frac{77\varepsilon_0\alpha^k}{80}+\frac{\sigma}{8\rho}\|\lambda_{k+1}-\tnu_k\|^2.
     \]
    Hence, \eqref{ineq:key} immediately follows.
     
     {\bf Case 3)} We now consider the third case, where the conditions for the first two cases fail to hold, that is, Algorithm \ref{alg:dsc_aalm} terminates in Step \textbf{4}. 
     Now that $\|\mathcal{G}^{(2L+\mu)^{-1}}_{\Lpt(\cdot,\tnu_k)}(\tx_k)\|\leq \varepsilon/4$, then Proposition~\ref{prop:lag_subgradient} applied to problem~\eqref{eq:ProbIntro_LC_pert} with $(\tx, x^+, \lambda, \lambda^+)=(\tx_k, x_{k+1}, \tnu_k, \lambda_{k+1})$, $\eta=(2L+\mu)^{-1}$, and $f$ replaced by $f(\cdot)+\gamma_p\|\cdot-x_0\|^2/2$ implies that there exists a $v\in\partial \tL(\cdot,\lambda_{k+1})(x_{k+1})$ such that $\|v\|\leq \varepsilon/2$. Using the initialization $\gamma_p= \varepsilon/(2D)$ in Algorithm~\ref{alg:dsc_aalm}, Lemma~\ref{lem:perturb_solution} with $(x,\lam)=(x_{k+1},\lambda_{k+1})$ implies that there exists a $v'\in\partial\L(\cdot,\lambda_{k+1})(x_{k+1})$ such that $\|v'\|\leq \varepsilon$.
     Hence $\|Ax_{k+1}-b\|\leq \varepsilon$ and $\|\mathcal{G}^{(2L+\mu)^{-1}}_{\Lpt(\cdot,\tnu_k)}(\tx_k)\|\leq \varepsilon/4$ indicates that $(x_{k+1},\lambda_{k+1})$ is an $\varepsilon$-primal-dual solution by~\eqref{def:approximate_kkt}.

    Therefore, we finish the inductive proof and thus complete the proof of the lemma. 
\end{proof}

\vspace{1em}

We are now ready to prove Theorem~\ref{thm:acc_alm_complexity_1}.

\vspace{1em}

\noindent\textbf{Proof of Theorem~\ref{thm:acc_alm_complexity_1}:} 
Recall from Proposition~\ref{prop:alm_inner_complexity_sc} that the inner complexity to satisfy the inner termination condition $\|\mathcal{G}_{\Lpt(\cdot,\tnu_k)}^{(2L+\mu)^{-1}}(\tx_k)\|\leq\varepsilon_k/(2D)$ is as in~\eqref{cmplx:inner_aalm}. Then, we simply need to bound the complexity of the outer loop. 

It is trivial to show that the parameter choices satisfy the conditions in the initialization of Algorithm~\ref{alg:dsc_aalm},
\begin{equation}
    4\rho\sigma\varepsilon\leq 1,\quad \varepsilon_0\geq \varepsilon,\quad 0\leq \alpha< (1+\sqrt{\gamma_d\rho})^{-2}.\label{ineq:param_choice_bound}
\end{equation}
We then proceed to bound the outer iteration complexity to satisfy each of the termination criteria in Step \textbf{4} of Algorithm~\ref{alg:dsc_aalm}. Combining the outer complexity with the inner complexity in~\eqref{cmplx:inner_aalm} will then yield the total complexity in~\eqref{cmplx:total_aalm}. 





First, we bound the complexity to satisfy the stationarity condition $\|\mathcal{G}_{\Lpt(\cdot,\tnu_k)}^{(2L+\mu)^{-1}}(\tx_k)\|\leq \varepsilon/4$, where $L$ and $\mu$ are as in~\eqref{def:setup_alm_acc}. By the inner termination condition in Step \textbf{2}, i.e., $\|\mathcal{G}_{\Lpt(\cdot,\tnu_k)}^{(2L+\mu)^{-1}}(\tx_k)\|\leq \varepsilon_k/(2D)$, the inequality $\|\mathcal{G}_{\Lpt(\cdot,\tnu_k)}^{(2L+\mu)^{-1}}(\tx_k)\|\leq \varepsilon/4$ is satisfied by any iteration with $\varepsilon_k\leq D\varepsilon/2$. By Assumption~\ref{assmp:constrained}(d) and~\eqref{ineq:param_choice_bound}, we obtain $\sigma\rho\varepsilon^2\leq \varepsilon/4\le D\varepsilon/4$. Then, from the choice of $\varepsilon_k$ in Step \textbf{1} of Algorithm~\ref{alg:dsc_aalm}, the condition
\[\varepsilon_k=\frac{7}{8}\varepsilon_0\alpha^k+\frac{\sigma\rho\varepsilon^2}{8}\leq \frac{7}{8}\varepsilon_0\alpha^k+\frac{D\varepsilon}{32}\leq \frac{D\varepsilon}{2}\]
is satisfied when $7\varepsilon_0\alpha^k/8\leq 15D\varepsilon/32$,
which occurs in 
    \begin{equation}
        N_\alpha=\left\lceil\frac{\log (15D\varepsilon/32)-\log(7\varepsilon_0/8)}{\log\alpha}\right\rceil\leq 1+\frac{\log (15D\varepsilon/32)-\log(7\varepsilon_0/8)}{\log\alpha}\stackrel{\eqref{ineq:alpha_conds}}\leq 1+\sqrt{\frac{D}{\rho\varepsilon}}\label{def:Nalpha}
    \end{equation}
    outer iterations, where the second inequality follows by the second condition on $\alpha$ in~\eqref{ineq:alpha_conds}.
    Therefore, $N_{\alpha}=\mathcal{O}(1+\sqrt{D/(\rho\varepsilon)})$.
    
    Next, we bound the outer iteration complexity required to satisfy the termination condition $\|Ax_{k}-b\|\leq\varepsilon$. Combining $\sigma\varepsilon\leq 1/{(4\rho)}$ from~\eqref{ineq:param_choice_bound} with the choice of $\gamma_d$ and $\sigma=1/4$, we can show that the condition~\eqref{ineq:gamma_d_req} in Proposition~\ref{prop:OR} is satisfied. Therefore, Lemma~\ref{lem:gamma_bounds_dual_sc} and Proposition~\ref{prop:OR} imply that Algorithm~\ref{alg:dsc_aalm} is an instance of the FLOrA framework (i.e., Algorithm~\ref{alg:flora}) with the correspondence~\eqref{def:acc_alm_corresp}.
    
    Then, using Theorem~\ref{thm:flora_complexity} with $R^\flora_0=\|\tilde\lambda_*\|\leq\bardtL$ (see~\eqref{def:R-R}) and $C_\flora=C$, we have
    \[
    \rho\|Ax_{k}-b\|\stackrel{\eqref{eq:lambda_acc}}=\|\lambda_{k}-\tnu_{k-1}\|\stackrel{\eqref{def:acc_alm_corresp}}=\|\ty^\flora_{k}-\tx^\flora_{k-1}\|\stackrel{\eqref{ineq:flora_b},\eqref{def:acc_alm_corresp}}\leq \frac{\sqrt{\rho}\bardtL+\sqrt{2\rho\varepsilon_0C}}{\sqrt{(1-\sigma)B_k}}.
    \]
    Combining the above inequality with Lemma~\ref{lem:b_seq}(c) immediately implies the outer complexity to guarantee near feasibility $\|Ax_{k+1}-b\|\leq \varepsilon$ is $\tO(1+1/\sqrt{\rho\gamma_d})$. Using Lemma~\ref{lem:gamma_bounds_dual_sc}(d) with $\beta=\sqrt{9/10}$ from~\eqref{ineq:alpha_conds} and the choice $\varepsilon_0=\rho^{-1}$, we can show that $\gamma_d = \mathcal{O}(\varepsilon/\bardtL)$ and $\mathcal{R}=\mathcal{O}(\bardtL)$. Then by Lemma~\ref{lem:double_perturbation_distance}, we have that $\mathcal{R}=\mathcal{O}(\bardtL)=\mathcal{O}(\hat R_\Lambda + D)$.

    Accordingly, the outer iteration count $k$ to satisfy $\|Ax_{k}-b\|\leq\varepsilon$ is
    \begin{equation}\label{cmplx:outer_aalm}
        \tO\left(1+\frac{1}{\sqrt{\rho\gamma_d}}\right)=\tO\left(1+\frac{\sigma^{3/4}\sqrt{\mathcal{R}}}{\sqrt{\rho\varepsilon}}\right)=\tO\left(1+\frac{\sqrt{\hat R_\Lambda+D}}{\sqrt{\rho\varepsilon}}\right),
    \end{equation}
    which is of the same order as $N_\alpha$ in~\eqref{def:Nalpha}.

    Combining the inner complexity from~\eqref{cmplx:inner_aalm} and the outer complexity from~\eqref{cmplx:outer_aalm}, substituting $\rho= L_f/\|A\|^2$, and using $D\geq 1$, we obtain the total complexity as in~\eqref{cmplx:total_aalm}.
    \QEDA

\section{Numerical Experiments}\label{sec:numerical}
In this section we provide numerical illustrations of the proposed primal and dual methods.  All code is implemented in Julia and is publicly available\footnote{\href{https://github.com/mxburns2022/PrimalDualRestart}{https://github.com/mxburns2022/PrimalDualRestart}}. Details of numerical experiments (problem generation, libraries, etc.) can be found in Appendix~\ref{appdx:numerical_details}. In-depth experimental analysis is beyond the scope of this work, and these tests should be taken as preliminary illustrations.

\subsection{Primal Methods: Restarted ACG}
We compare Restarted ACG (Algorithm~\ref{alg:restart}) to baseline ACG (Algorithm~\ref{alg:ACG}, ``None'') as well as two prominent restart schemes from literature: ``gradient'' restarting~\cite{odonoghueAdaptiveRestartAccelerated2015} and ``speed'' restarting~\cite{su2016differential}. 

Gradient restarting is a heuristic scheme that restarts the ACG solver whenever the gradient mapping forms an acute angle with the update direction, i.e., $\inner{\tx_k-y_{k+1}}{y_{k+1}-y_k}>0$.

For speed restarting, we restart the acceleration whenever the distance between adjacent iterates decreases, $\|y_{k+1}-y_k\|<\|y_{k}-y_{k-1}\|$, motivated by the continuous-time limit of ACG~\cite{su2016differential}. To prevent the speed scheme from restarting too often, we only allow restarts at most every $k_{\min}$ iterations. As in~\cite{su2016differential}, we set $k_{\min}=10$.

We focus on the sparse linear regression/LASSO problem~\cite{tibshiraniRegressionShrinkageSelection1996}
\begin{equation}\label{def:lasso}
        \phi_*:=\min_{x\in\R^n}\left\{\phi(x):={\frac{1}{2}\|Ax-b\|_2^2} + {\gamma\|x\|_1}\right\}.
\end{equation}
where $A\in\R^{m\times n}$, $b\in\R^m$, and $\gamma > 0$. We use $n=1000$, $m=500$, and $\gamma=1/2$.

Fig.~\ref{fig:restart_figures} shows the estimated function value gap $\phi(y_k)-\phi_*$, where $\phi_*$ is the best solution found by any solver, versus the number of ACG iterations. All of the restart methods are significantly faster than baseline ACG (``None''). Speed restarting and Restarted ACG behave quite similarly, while gradient restarting has the most rapid convergence.  


\begin{figure}[t]
    \centering
    \includegraphics[width=0.7\linewidth]{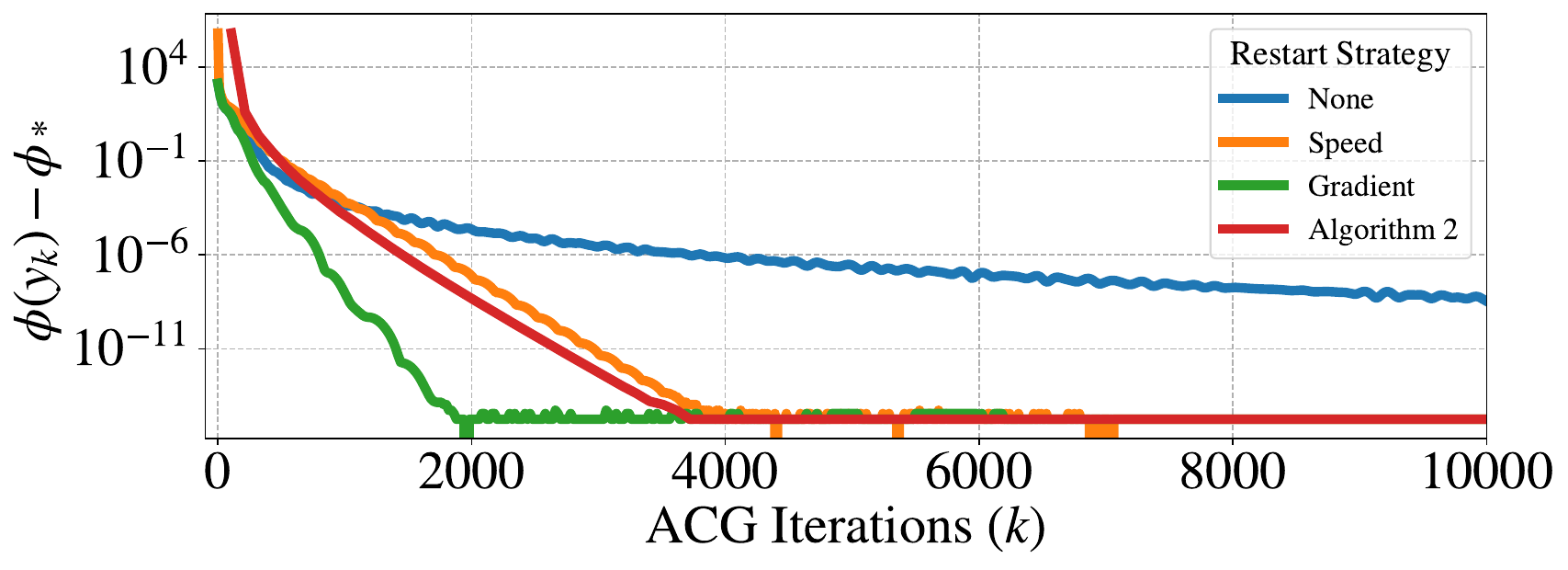}
    \caption{Numerical results for Restarted ACG algorithms.}
    \label{fig:restart_figures}
\end{figure}
\subsection{Dual Methods: Augmented Lagrangian}

\begin{figure}[t]
    \centering

\begin{subfigure}[t]{0.6\textwidth}
    \includegraphics[width=\linewidth]{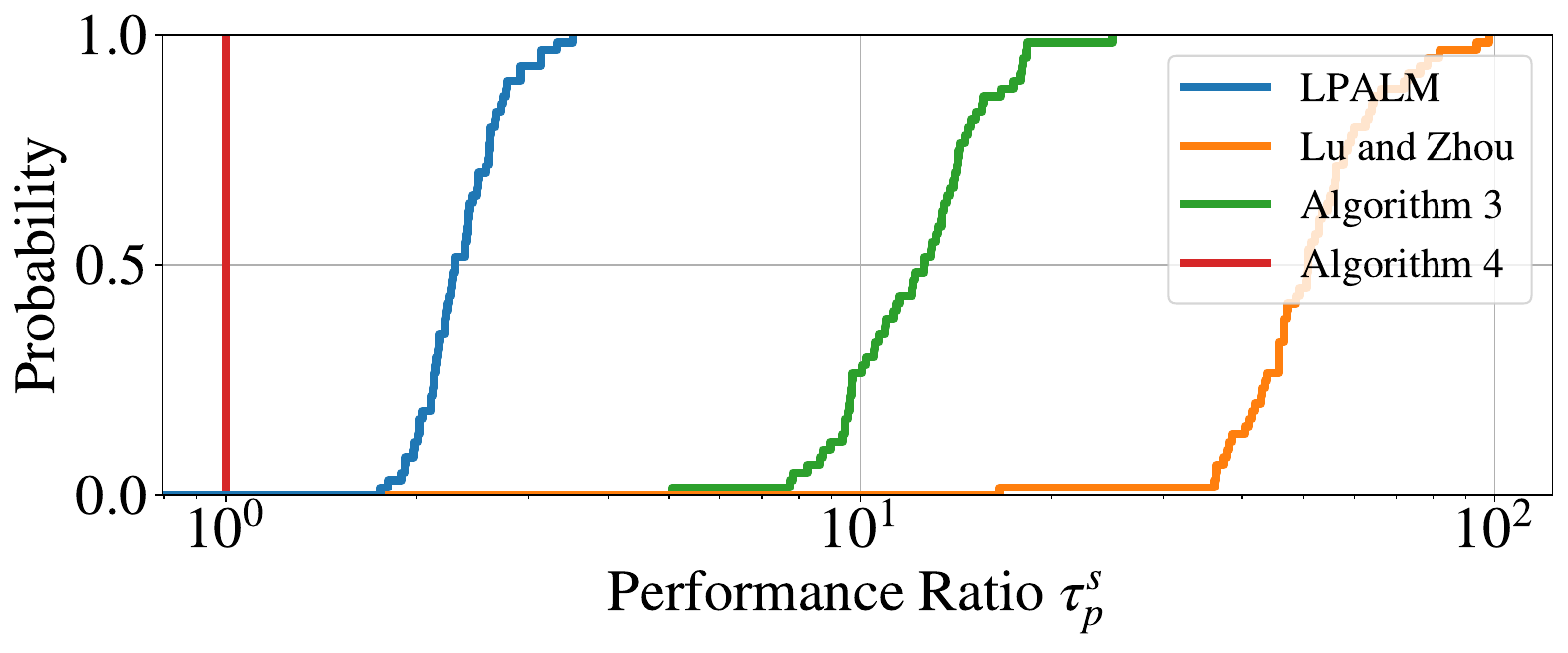}
    \subcaption{Performance profile of I-ALM algorithms in $n=200$, $m=100$ LCQP instances with $\varepsilon=10^{-3}$}
    \label{fig:alm_performance_profile}
\end{subfigure}
\begin{subfigure}[t]{0.4\textwidth}
    \centering
    \includegraphics[width=\linewidth]{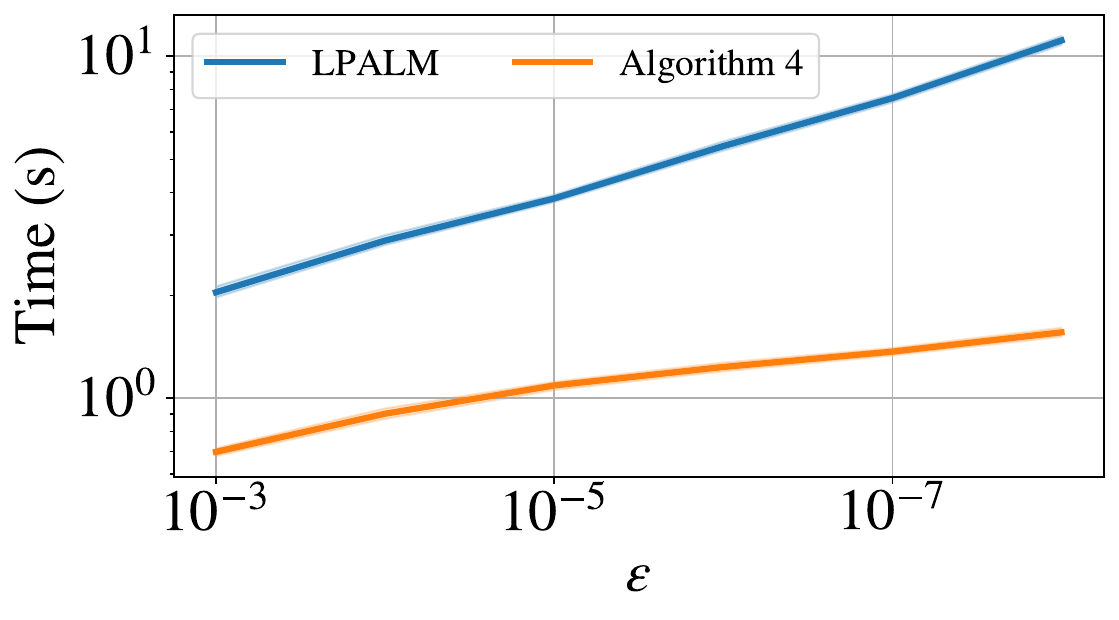}
    \subcaption{Wall time scaling of Algorithm~\ref{alg:dsc_aalm} and~\cite[Algorithm 1]{liAcceleratedAlternatingDirection2019} with varying {accuracy} $\varepsilon^{-1}$.}
    \label{fig:alm_absolute_1000}
\end{subfigure}
\begin{subfigure}[t]{0.4\textwidth}
    \centering
    \includegraphics[width=\linewidth]{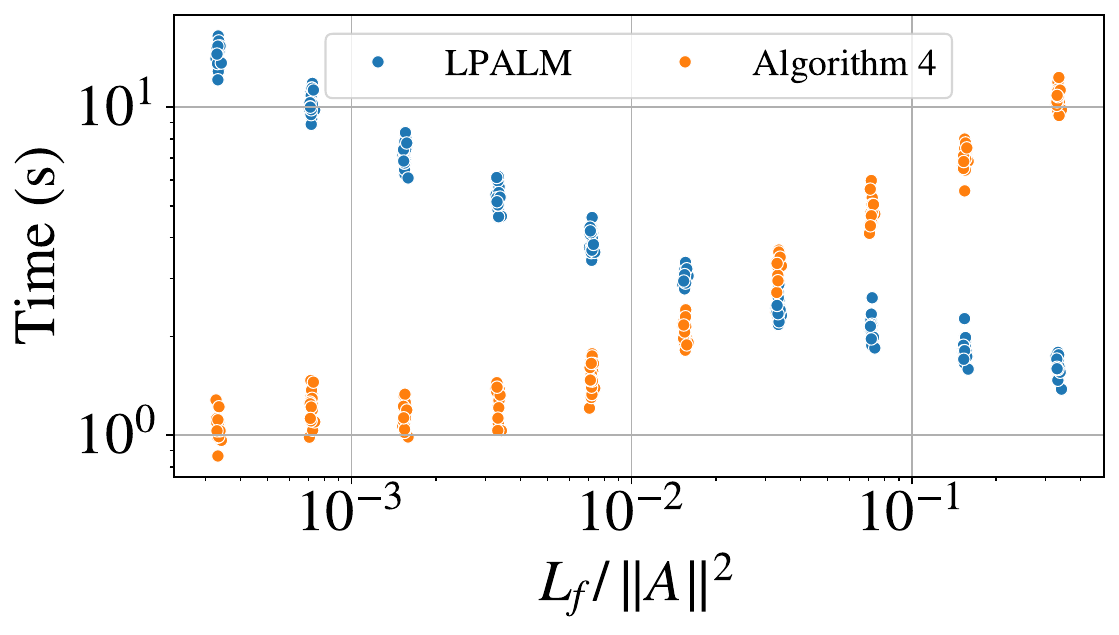}
    \subcaption[t]{Wall time scaling of Algorithm~\ref{alg:dsc_aalm} and~\cite[Algorithm 1]{liAcceleratedAlternatingDirection2019} with varying $L_f/\|A\|^2$ with relative accuracy $L_f/\varepsilon$ and optimizer held constant.}
    \label{fig:alm_relative_1000}
\end{subfigure}
    \caption{Numerical experiments for ALM variants tested.}
    \label{fig:alm_1000}
\end{figure}
In this subsection, we compare several proposed ALM variants in linearly-constrained quadratic programming (LCQP). The LCQP problem is given by
\begin{equation*}
    \hat\phi_*:=\min_{x\in\R^n} \left\{\phi(x):=\frac{1}{2}x^\top Mx + c^\top x + \delta_{\mathrm{Q}}(x):Ax=b\right\},
\end{equation*}
where $M\in\mathbb{R}^{n\times n}$ is positive semi-definite, $c\in \R^n$, $A\in \R^{m\times n}$ is full rank, $b\in \R^m$, and $m\leq n$. For the nonsmooth component, we choose the indicator function $\delta_Q(\cdot)$, where $Q$ is the $n$-dimensional box with side length 20 centered at 0, $Q=\{x:-10\leq x_i\leq 10 \text{ for all } 1\leq i\leq n\}$. The LCQP problem is a staple of model predictive control~\cite{frisonHPIPMHighperformanceQuadratic2020} and as a subproblem in algorithms for nonsmooth optimization~\cite[Lemma 10.8]{bonnans2006numerical}, motivating its inclusion.

In addition to Algorithms~\ref{alg:al} and~\ref{alg:dsc_aalm},  we compare against two $\tO(\varepsilon^{-1})$ ALM proposals from the literature: the decreasing error/decreasing penalty scheme from~\cite[Algorithm 2]{luIterationComplexityFirstOrderAugmented2023} (``Lu and Zhou'') and the linearized proximal ALM (``LPALM'') with static $\rho$. LPALM is derived from the LPADMM method of~\cite[Algorithm 1]{liAcceleratedAlternatingDirection2019} by setting one block to zero. The inner loop of Lu and Zhou is terminated based on the absolute error condition $\Lp(x_{k+1},\lambda_k)-\min_{x\in \R^n}\Lp(x,\lambda_k)\leq \varepsilon_k$, which is estimated by the gradient mapping norm using Lemma~\ref{lem:grad_map_merged}(a). We tested Algorithm~\ref{alg:al} with a variety of $\rho$ values, finding that $\rho=1.0$ was the most performant in practice. For Algorithm~\ref{alg:dsc_aalm}, we set $\rho=\sqrt{m}\|M\|/\|A\|^2$ (i.e., $\sqrt{m}L_f/\|A\|^2$).

Figure~\ref{fig:alm_performance_profile} shows the performance profile~\cite{dolanBenchmarkingOptimizationSoftware2002b} of the four algorithms across 60 randomly generated LCQP instances with $n=200$ and $m=100$. We solve each instance to $\varepsilon=10^{-3}$ accuracy (using the definition in~\eqref{def:approximate_kkt}), terminating when an $\varepsilon$-primal-dual solution is detected. The x-axis is the ``performance ratio'' $\tau^s_p=t^s_p / \min_s t^s_p$, where $t^s_p$ is the elapsed wall-time needed for solver $s$ to reach the target accuracy on problem $p$, i.e., $\tau_p^s=1$ if solver $s$ was first to achieve $\varepsilon\leq10^{-3}$ on problem $p$. The y-axis of Figure~\ref{fig:alm_performance_profile} shows the cumulative distribution of $\tau_p^s$ for each solver across the 60 instances tested. Algorithm~\ref{alg:dsc_aalm} shows a clear advantage, with LPALM placing second. 

Focusing on the two most performant methods, Figs.~\ref{fig:alm_absolute_1000} and~\ref{fig:alm_relative_1000} compare LPALM and Algorithm~\ref{alg:dsc_aalm} across 20 LCQP instances with $n=1000$ and $m=500$. Fig~\ref{fig:alm_absolute_1000} fixes the problem set ($L_f=1.0$, $\|A\|\approx 17$) and varies the target accuracy $\varepsilon$. Both methods appear to scale similarly. However Algorithm~\ref{alg:dsc_aalm} is over $5\times$ faster, with the gap widening in the high-accuracy regime.

Fig.~\ref{fig:alm_relative_1000} examines the performance impact of the ratio $L_f/\|A\|^2$ in Algorithm~\ref{alg:dsc_aalm} and LPALM. We fix 20 problems $\{(M_i,c_i,A_i,b_i)\}$, then rescale each problem $(M_i,c_i,A_i,b_i)$ to $(\chi M_i,\chi c_i,A_i,b_i)$ and solve to $\chi 10^{-6}$ accuracy for some $\chi>0$\footnote{We hold the primal feasibility target constant, only adjusting the tolerance for the primal subgradient norm.}. It is worth emphasizing again that the problems, minimizers, and relative accuracy are constant: the only variable is the rescaled ratio. We tested 10 values of $\chi\in[0.1, 100]$. As seen in Fig.~\ref{fig:alm_relative_1000}, LPALM and Algorithm~\ref{alg:dsc_aalm} are effective in two very different regimes. For $L_f\ll \|A\|^2$, Algorithm~\ref{alg:dsc_aalm} is over 10$\times$ faster. However, the methods meet when $L_f\sim 0.03\|A\|^2$, and LPALM significantly overtakes Algorithm~\ref{alg:dsc_aalm} in the regime $L_f\geq 0.1\|A\|^2$. These findings suggest that the ``rescaling'' discussed in the remarks after Corollary~\ref{cor:aalm_complexity_pd} is more than a theoretical convenience: Algorithm~\ref{alg:dsc_aalm} performs significantly better when $L_f\ll \|A\|^2$, even when that requires decreasing $\varepsilon$.


\section{Concluding Remarks}\label{sec:conclusion}
This paper proposes the Restarted ACG method (Algorithm~\ref{alg:restart}), I-ALM (Algorithm~\ref{alg:al}), and I-FALM (Algorithm~\ref{alg:dsc_aalm}). Our improved analysis of all three methods is grounded in a unified IPP perspective, making use of the LOrA and FLOrA frameworks proposed in Section~\ref{sec:framework}.
Using the FLOrA framework, we show that Algorithm~\ref{alg:restart} achieves optimal global complexity for solving~\eqref{eq:ProbIntro} in both convex and strongly convex settings, which, to our knowledge, is a novel result in the restarted ACG literature. Similarly, we utilize the LOrA framework to prove that  Algorithm~\ref{alg:al} achieves near-optimal, non-ergodic complexity for solving~\eqref{eq:ProbIntro_LC} with constant regularization, a novel result in the ALM literature to our knowledge. Finally, we combine the analysis of Algorithm~\ref{alg:al} with the FLOrA framework to develop an accelerated variant, I-FALM (Algorithm~\ref{alg:dsc_aalm}), which also achieves near-optimal non-ergodic complexity for solving~\eqref{eq:ProbIntro_LC}. Both Algorithms~\ref{alg:al} and~\ref{alg:dsc_aalm} utilize gradient mapping-based termination criteria for the inner ACG solver, which are both efficiently computable and remove the need for the postprocessing used in previous ALM literature~\cite{lanIterationcomplexityFirstorderAugmented2016a}.
Numerical experiments validate the empirical performance of the proposed algorithms, with Algorithm~\ref{alg:dsc_aalm} significantly outperforming competing ALM variants.

Several related questions merit future investigation. First, Restarted ACG attains optimal complexity for strongly convex optimization if the modulus $\mu_f$ is provided. However, in the absence of prior knowledge about $\mu_f$, one must rely on universal methods such as \cite{guiguesUniversalSubgradientProximal2026b,sujanani2025efficient}, which achieve complexity bounds in terms of $\mu_f$ as good as those obtained when $\mu_f$ is known in advance. 
Second, Assumption~\ref{assmp:constrained}(d) plays a crucial role in our analysis throughout Section~\ref{sec:dual}, and it remains an open question whether optimal I-ALM variants can be designed with inexact subroutines without boundedness. It is also of interest to design an algorithm that does not require an estimate of $R_\Lambda$ as input, since an estimate may not be available a priori. Third, another related pursuit would be to obtain (near)-optimal complexities for Algorithm~\ref{alg:dsc_aalm} without primal-dual perturbations (i.e., $\gamma_d=0$ and $\gamma_p=0$), which may remove the explicit need for an $R_\Lambda$ estimate.
    
	\bibliographystyle{plain}
	\bibliography{ref_cleaned}
	
	\appendix


\section{Details of Numerical Experiments}\label{appdx:numerical_details}
All experiments are run on a workstation desktop running Ubuntu 24.04.3 LTS with an Intel i9-13900k CPU and 64 GB of RAM. Proximal operator implementations are taken from the ProximalOperators.jl\footnote{\url{https://github.com/JuliaFirstOrder/ProximalOperators.jl}} package.
\subsection{Restarted ACG Experimental Details}
Recall that our problem of interest is the sparse linear regression/LASSO problem ~\eqref{def:lasso}.
For testing we set $n=1000$, $m=500$, and $\gamma=1/2$. $A$ is set to 20\% density, with nonzero entries generated IID normal. The vector $b$ is randomly generated with IID uniform entries over $[0,1]$. We start each solver from the origin $x_0=0$ and use $L_f=\|A\|^2$. After some brief parameter tuning, we set $\lambda=0.2$ in Algorithm~\ref{alg:restart}. The function value and number of restart steps are logged on every restart for each algorithm. Note the location of data points along the x-axis of Fig.~\ref{fig:restart_figures} is therefore non-uniform, since the number of steps between each restart differ for each algorithm.
\subsection{I-ALM Experimental Details}
Recall that the problem class used for I-ALM testing is the linearly constrained quadratic program

\begin{align*}
    &\min_{x\in\R^n} \frac{1}{2}x^\top Mx + c^\top x\\
    &\text{s.t.  } Ax=b\\
    &\quad x_\ell\leq x_i \leq x_u \text{ for all } i\in\{1,\dots,n\}.
\end{align*}

Fixing $n, m,r\in\mathbb{N}$ with $n\geq \max\{r,m\}$ and $\delta\in[0,1]$, we generate problem structures using the following procedure:
\begin{itemize}
    \item $\tilde{M}=RR^T$ with $R\in\R^{n\times r}$, $R_{ij}\sim\mathcal{N}(0,1)$. We then rescale $M_{ij}=\tilde{M}_{ij}/\|M\|$ to ensure that $L_f=\|M\|=1$,
    \item $c\in\R^n$, $c_i\sim \mathcal{N}(0,1)$,
    \item $A\in\R^{m\times n}$, $A_{ij}\sim \text{Bern}(\delta)\cdot\mathcal{N}(0,1)$,
    \item $b\in\R^{m}$, $b_{i}\sim \mathcal{N}(0,1)$,
    \item $x_\ell=-10$, $x_u=10$.
\end{itemize} 
$M$ is therefore an $n\times n$ matrix with rank $r$, $c$ and $b$ are entry-wise normally-distributed vectors, and $A$ is a normally distributed matrix with density $\delta=0.1$. We them compute the diameter $D$ as $D=\sqrt{n}(x_\ell-x_u)$. We estimate $\hat{R}_\Lambda=1000$ for all problems, which worked as a sufficient upper bound in practice. This procedure was only used for testing. In general, one could use a ``guess-and-check'' procedure as proposed in~\cite{monteiro2016adaptive} which only adds $\mathcal{O}(\log\bardtL)$ complexity.
For ``Lu and Zhou", we set $\varepsilon_k= \varepsilon_0\alpha^k$ with $\alpha=0.7$, $\varepsilon_0=0.1$, and $\rho_0=10$ after preliminary tuning. For LPALM we set $\rho=\max\{\sqrt{L_f}/\|A\|,L_f/\|A\|^2\}$, which was a performant heuristic in our limited numerical testing. We set $\alpha=0.7$, $\varepsilon_0=100$ for Algorithm~\ref{alg:al} and $\alpha=0.85$, $\varepsilon_0=\rho^{-1}$ for Algorithm~\eqref{alg:dsc_aalm}.

\section{Technical Results}\label{appdx:technical}
The gradient mapping (defined in~\eqref{def:grad_mapping}) serves a critical role in Subsections~\ref{ssec:proofs_alm} and~\ref{ssec:proofs_falm}, as well as in the numerical experiments in Section~\ref{sec:numerical}. The following lemma provides several technical results for the gradient mapping of a convex composite function.
\begin{lemma}\label{lem:grad_map_merged}
    Consider problem~\eqref{eq:ProbIntro}, which we assume satisfies Assumption~\ref{assmp:csco}. Additionally assume that $\dom h$ is bounded with diameter $D$.  Given $\eta\leq L_f^{-1}$ and $\tx\in\dom h$, define the gradient mapping $\mathcal{G}^\eta_\phi(\tx)$ as in~\eqref{def:grad_mapping} and set 
    \begin{equation}\label{eq:tx+}
        x^+=\tx-\eta\mathcal{G}^\eta_\phi(\tx).
    \end{equation} 
    Then, the following statements hold:
    \begin{itemize}
        \item[{\rm a)}] denoting $\phi_*=\min_{x\in\dom h}\phi(x)$, we have
    \[
        \phi(x^+)-\phi_*\leq D\|\mathcal{G}^\eta_\phi(\tx)\|-\frac{\eta}{2}\|\mathcal{G}^\eta_\phi(\tx)\|^2;
    \]
        \item[{\rm b)}] if $\|\mathcal{G}^\eta_\phi(\tx)\|\leq \varepsilon$, then there exists a subgradient $v\in\partial \phi(x^+)$ satisfying $\|v\| \le 2\varepsilon$;
        \item[{\rm c)}] given $\gamma>0$ and $\bar{x}\in \dom h$, define $\phi_\gamma(x)= \phi(x)+\gamma\|x-\bar{x}\|^2/2$. Suppose that $\gamma\leq \varepsilon/(2D)$ and $\|\mathcal{G}^{\eta}_{\phi_\gamma}(\tx)\|\leq \varepsilon/2$, where the proximal mapping in $\mathcal{G}^{\eta}_{\phi_\gamma}(\cdot)$ is still with respect to $h$, then  $\|\mathcal{G}^{\eta}_\phi(\tx)\|\leq \varepsilon$.
    \end{itemize}
\end{lemma}
\begin{proof}
    a) Applying Lemma 2.3 of \cite{beck2009fast} and using the definition of $x^+$ in \eqref{eq:tx+}, we have for every $y\in\dom h$,
    \begin{align*}
        \phi(y)-\phi(x^+)&\geq \frac{1}{2\eta}\|x^+-\tx\|^2+\eta^{-1}\inner{\tx-y}{x^+-\tx} \stackrel{\eqref{eq:tx+}}=\frac{\eta}{2}\|\mathcal{G}^\eta_\phi(\tx)\|^2-\inner{\tx-y}{\mathcal{G}^\eta_\phi(\tx)}\\
        &\geq \frac{\eta}{2}\|\mathcal{G}^\eta_\phi(\tx)\|^2-\|\tx-y\|\|\mathcal{G}^\eta_\phi(\tx)\|
        \geq \frac{\eta}{2}\|\mathcal{G}^\eta_\phi(\tx)\|^2-D\|\mathcal{G}^\eta_\phi(\tx)\|,
    \end{align*}
    where the second inequality is due to the Cauchy-Schwarz inequality and the last inequality follows from the boundedness of $\dom h$. The statement follows by taking $y=x_*$ for any $x_*\in\{x\in\dom h:\phi(x)=\phi_*\}$.    

    b) In view of \eqref{def:grad_mapping}, the definition of $x^+$ in \eqref{eq:tx+} can be rewritten as
    \[
    x^+ = \prox_{\eta h}(\tx-\eta\nabla f(\tx)),
    \]
    whose optimality condition yields that
    \[
        0\in \frac{x^+-\tx}{\eta}+\nabla f(\tx)+\partial h(x^+).
    \]
    Rearranging terms and adding $\nabla f(x^+)$ to both sides, we have
    \begin{equation*}
        v:= \frac{\tx-x^+}{\eta}-\nabla f(\tx)+\nabla f(x^+)\in \partial h(x^+)+\nabla f(x^+)= \partial \phi(x^+).
    \end{equation*}
    Using the triangle inequality and the smoothness of $f$, we have
    \begin{align*}
        \|v\|\leq \frac{1}{\eta}\|{\tx-x^+}\|+ L_f\|x^+ - \tx\| \le \frac{2}{\eta}\| \tx-x^+\| \stackrel{\eqref{eq:tx+}}=2\|\mathcal{G}_\phi^\eta(\tx)\|\leq 2\varepsilon,
    \end{align*}
    where the second inequality is due to the fact that $L_f \le 1/\eta$.
    Hence, we prove the statement.

    c) It follows from Lemma 3.1(iii) of \cite{itoNearlyOptimalFirstOrder2021} that for any $\gamma >0$,
    \begin{equation}\label{ineq:gradient_map_diff}
        \|\mathcal{G}^{\eta}_\phi(\tx)-\mathcal{G}^{\eta}_{\phi_\gamma}(\tx)\|\leq \gamma\|\tx-\bar{x}\|.
    \end{equation}
    Using the above inequality and the triangle inequality, we have
    \[
        \|\mathcal{G}^{\eta}_\phi(\tx)\|\leq \|\mathcal{G}^{\eta}_{\phi_\gamma}(\tx)\|+\|\mathcal{G}^{\eta}_{\phi_\gamma}(\tx)-\mathcal{G}^{\eta}_{\phi}(\tx)\|
        \stackrel{\eqref{ineq:gradient_map_diff}}\leq \|\mathcal{G}^{\eta}_{\phi_\gamma}(\tx)\|+\gamma\|\tx-\bar{x}\|\leq \|\mathcal{G}^{\eta}_{\phi_\gamma}(\tx)\|+\gamma D\leq \varepsilon,
    \]
    where the third inequality follows from boundedness, and the final inequality follows from the assumptions on $\|\mathcal{G}^{\eta}_{\phi_\gamma}(\tx)\|$ and $\gamma$.
\end{proof}

The following result connects the gradient mapping of the augmented Lagrangian function $\Lp$ to the subdifferential of the Lagrangian $\L$. The lemma is used to prove Propositions~\ref{prop:alm_contrapositive} and~\ref{prop:OR}.
\begin{proposition}\label{prop:lag_subgradient}
    Consider problem~\eqref{eq:ProbIntro_LC}, which we assume satisfies Assumption~\ref{assmp:constrained}. Given $\lambda\in\R^m$ and $\rho>0$, let $\Lp(\cdot,\lambda)$ be the augmented Lagrangian in~\eqref{def:augmented_lagrangian}. Define $M_\rho$ as in~\eqref{def:Mrho}, let $\eta\leq M_\rho^{-1}$, and
    suppose $\tx$ is a point satisfying $\|\mathcal{G}^{\eta}_{\Lp(\cdot,\lambda)}(\tx)\|\leq \varepsilon/2$. Set 
    \begin{equation}\label{eq:lora_update_appdx_b}
        x^+=\tx-M_\rho^{-1}\mathcal{G}^{\eta}_{\Lp(\cdot,\lambda)}(\tx),\quad \lambda^+=\lambda+\rho(Ax^+-b).
    \end{equation}
    Then, there exists a subgradient $v\in\partial \L(\cdot,\lambda^+)(x^+)$ satisfying $\|v\|\leq \varepsilon$.
\end{proposition}
\begin{proof}
    By Lemma~\ref{lem:grad_map_merged}(b) applied to $\Lp(\cdot,\lambda)$ with gradient mapping stepsize $\eta$, there exists a $v\in \partial \Lp(\cdot,\lambda) (x^+)$ satisfying $\|v\|\leq \varepsilon$. It follows from the definition of $\Lp(\cdot,\lambda)$ in \eqref{def:augmented_lagrangian} and subdifferential calculus that
    \begin{align*}
        \partial \Lp(\cdot,\lambda)(x^+)&=\partial\left(\L(\cdot,\lambda)+\frac{\rho}{2}\|A\cdot -b\|^2\right)(x^+)
        =\partial \L(\cdot,\lambda)(x^+)+\rho A^\top(Ax^+-b)\\
        &=\partial(\L(\cdot,\lambda)+\inner{\rho(Ax^+-b)}{A\cdot-b})(x^+)
        \stackrel{\eqref{eq:strong_duality},\eqref{eq:lora_update_appdx_b}}=\partial \L(\cdot,\lambda^+) (x^+),
    \end{align*}
    where the last identity follows from \eqref{eq:lora_update_appdx_b} and the definition of $\L(\cdot,\lambda)$ in \eqref{eq:strong_duality}.
    Therefore, we prove $v\in \partial \L(\cdot,\lambda^+) (x^+)$ and thus conclude the proof.
\end{proof}

We can show that an $\varepsilon$-primal-dual solution to~\eqref{eq:ProbIntro_LC} in the sense of~\eqref{def:approximate_kkt} implies an $\mathcal{O}(\varepsilon)$ bound on the absolute primal gap $|\phi(x)-\hat\phi_*|$. The result has been used in prior works~\cite{luIterationComplexityFirstOrderAugmented2023}, though we repeat the proof for completeness.

\begin{lemma}\label{lem:pd_gap}
    Suppose $(x,\lam)$ is an $\varepsilon$-primal-dual solution to~\eqref{eq:ProbIntro_LC} in the sense of~\eqref{def:approximate_kkt}. Then the absolute value of the primal gap is bounded by
    \begin{equation}\label{ineq:primal_gap_ul}
        |\phi(x)-\hat\phi_*|\leq \varepsilon\max\{\|\lambda_*\|,\|\lam\|+D\},
    \end{equation}
    where $\lam_*\in\Lambda_*=\{\lambda:d(\lambda)=d_*\}$ is an optimal dual solution to~\eqref{eq:strong_duality}.
\end{lemma}

\begin{proof}
    We start by proving
    \begin{equation}\label{ineq:pd_gap}
        \phi(x)-d(\lam)\leq \varepsilon(\|\lam\|+D).
    \end{equation}
    Since $(x,\lam)$ is an $\varepsilon$-primal-dual solution to~\eqref{eq:ProbIntro_LC},
    \begin{equation}\label{ineq:assumption_pd_sol}
        v\in \partial\mathcal{L}(\cdot,\lam)(x),\quad\|v\|\leq \varepsilon,\quad \|Ax-b\|\leq \varepsilon,
    \end{equation}
    for some $v\in \R^n$.
    Define $u(\lam)=\underset{x\in\R^n}\argmin \L(x,\lambda)$. Then by the definition of the subdifferential $\partial\mathcal{L}(\cdot,\lam)(x)$ we have
    \[
        -d(\lambda)\stackrel{\eqref{eq:strong_duality}}=-\L(u(\lam),\lam)
        \leq -\L(x,\lam)-\inner{v}{u(\lam)-x}
        \stackrel{\eqref{eq:strong_duality}}=- \phi(x)-\inner{\lam}{Ax-b}-\inner{v}{u(\lam)-x}.
    \]
Rearranging and using the Cauchy-Schwarz inequality gives
\begin{align*}
    \phi(x)&-d(\lambda_r)\leq -\inner{\lam}{Ax-b}-\inner{v}{u(\lam)-x}\\
    &\leq \|\lam\|\|Ax-b\|+\|v\|\|u(\lam)-x\|
    \stackrel{\eqref{ineq:assumption_pd_sol}}\leq \varepsilon\|\lam\|+\varepsilon D,
\end{align*}
where the last inequality follows from~\eqref{ineq:assumption_pd_sol} and Assumption~\ref{assmp:constrained}(d).

    We now prove~\eqref{ineq:primal_gap_ul}. Clearly, the upper bound
    \[
    \phi(x)-\hat\phi_*\leq \varepsilon(\|\lam\|+D)
    \]
    follows from~\eqref{ineq:pd_gap} and strong duality $d(\lambda)\leq d_*=\hat\phi_*$ by~\eqref{eq:strong_duality}. Let $x_*\in \{x\in\dom h: \phi(x)=\hat\phi_*, \,\,Ax=b\}$ be an optimizer of~\eqref{eq:ProbIntro_LC}. Note that since $(x_*,\lam_*)$ is a saddle-point of $\L$ in~\eqref{eq:strong_duality}, we have 
    \begin{align*}
       0&\leq \L(x,\lam_*)-\L(x_*,\lam_*)=\phi(x)+\inner{\lambda_*}{Ax-b}-\hat\phi_*-\inner{\lambda_*}{Ax_*-b}\\
       &=\phi(x)+\inner{\lambda_*}{Ax-b}-\hat\phi_*.
    \end{align*}
    where the last equality follows by the feasibility of $x_*$. Rearranging, we have
    \begin{align*}
        \phi(x)-\hat\phi_*\geq -\inner{\lambda_*}{Ax-b}\geq -\|\lambda_*\|\|Ax-b\|\stackrel{\eqref{ineq:assumption_pd_sol}}\geq -\|\lambda_*\|\varepsilon,
    \end{align*}
    where the second inequality follows from the Cauchy-Schwarz inequality and the final inequality follows from the assumption that $(x,\lambda)$ is an $\varepsilon$-primal-dual solution to~\eqref{eq:ProbIntro_LC}. The inequality~\eqref{ineq:primal_gap_ul} follows from combining the upper and lower bounds.
\end{proof}

\section{Analysis of Frameworks in Section \ref{sec:framework}}\label{appdx:framework}

This section develops the analysis of LOrA and FLOrA frameworks introduced in Section \ref{sec:framework} and finally proves the two main results on sub-optimality guarantees, namely Theorems~\ref{thm:lora_complexity} and~\ref{thm:flora_complexity}.

\subsection{LOrA Analysis}\label{appdx:lora_framework}

We begin the analysis of LOrA by providing the proof of Proposition~\ref{prop:prox_point_lora}. For simplicity of notation, we omit the superscripts $\cdot^\lora$ and subscripts $\cdot_\lora$ in all proofs.

\vspace{1em}

\noindent
{\bf Proof of Proposition~\ref{prop:prox_point_lora}:}
We first note from \eqref{eq:x_lora} that
\begin{equation}\label{incl:optcond}
    \hat u_{k+1} \in \partial \Gamma_k(x_{k+1}).
\end{equation}
    Using \eqref{ineq:Gamma}, \eqref{incl:optcond}, and the fact that $\Gamma_k$ is $\lam^{-1}$-strongly convex, we have 
\[
\Phi(\hat{x}_*) + \frac{1}{2\lam}\|\hat{x}_*-x_k\|^2 \stackrel{\eqref{ineq:Gamma}}{\ge} \Gamma_k(\hat{x}_*) \stackrel{\eqref{incl:optcond}}\ge 
\Gamma_k(x_{k+1}) + \inner{\hat u_{k+1}}{\hat{x}_*-x_{k+1}} + \frac{1}{2\lam}\|\hat{x}_*-x_{k+1}\|^2.
\]
Rearranging the terms and adding $\lam\|\hat u_{k+1}\|^2/2$ to both sides, we have
\begin{align*}
    \frac{\lam}{2}\|\hat u_{k+1}\|^2 - \Gamma_k(x_{k+1}) + \Phi(\hat{x}_*) + \frac{1}{2\lam}\|\hat{x}_*-x_k\|^2&\ge \frac{\lam}{2}\|\hat u_{k+1}\|^2 + \inner{\hat u_{k+1}}{\hat{x}_*-x_{k+1}} + \frac{1}{2\lam}\|\hat{x}_*-x_{k+1}\|^2 \\
    &= \frac{1}{2\lam}\|\lambda\hat u_{k+1} + \hat{x}_*-x_{k+1}\|^2 \ge 0,
\end{align*}
and hence
\begin{equation}\label{ineq:inter}
    \frac{\lam}{2}\|\hat u_{k+1}\|^2 - \Gamma_k(x_{k+1})\geq -\Phi(\hat{x}_*) - \frac{1}{2\lam}\|\hat{x}_*-x_k\|^2.
\end{equation}
Combining the above inequality with \eqref{ineq:lora_base} yields
\begin{align*}
    \frac{\sigma}{2\lam}\|y_{k+1}-x_k\|^2 +\delta_k&\stackrel{\eqref{ineq:lora_base}}\geq \frac{\lam}{2}\|\hat u_{k+1}\|^2+\Phi(y_{k+1})+\frac{1}{2\lambda}\|y_{k+1}-x_k\|^2-\Gamma_k(x_{k+1})\\
    &\stackrel{\eqref{ineq:inter}}\ge \Phi(y_{k+1})+\frac{1}{2\lambda}\|y_{k+1}-x_k\|^2 - \Phi(\hat{x}_*) - \frac{1}{2\lambda}\|\hat{x}_*-x_k\|^2,
\end{align*}
proving the claim.
\QEDA

We next present a technical lemma that is useful in the analysis of LOrA. The first statement is a single-step bound on the primal gap, while the second statement provides a uniform upper bound for the iterate distance. Note that if $\delta_k^\lora=0$ for all $k\geq 0$, then the sequence $\|x_k^\lora-x_*\|$ is non-increasing. However, uniformly bounding the distance with $\delta_k^\lora>0$ requires summability of the absolute error terms.

\begin{lemma}\label{lem:single_step_lora}

    Let $X_*$ be the set of optimal solutions to \eqref{eq:general_convex_prob}. Define $x_*=\argmin\{\|x^\lora_0-x\|:x\in X_*\}$ and suppose $\A^\lora_k=\infty$. Then, for every $k\ge 0$, we have
    \begin{equation}
        2\lambda_\lora[\Phi(y^\lora_{k+1})-\Phi(x_*)]+(1-\sigma_\lora)\|y^\lora_{k+1}-x^\lora_k\|^2
        \leq \|x^\lora_k-x_*\|^2-\|x^\lora_{k+1}-x_*\|^2 + 2\lambda\delta^\lora_k.\label{ineq:lora_one_step}
    \end{equation}
    Moreover, if $\{\delta^\lora_k\}$ is summable with $C_\delta:=\sum_{i=0}^{\infty}\delta^\lora_k<\infty$, we have for every $k\geq 0$,
    \begin{equation}
        \|x^\lora_{k}-x_*\|\leq R_0^\lora + \sqrt{2\lambda_\lora C_\delta},\label{ineq:lora_xkdist}
    \end{equation}
    where $R_0^\lora=\|x_0^\lora-x_*\|$.
\end{lemma}

\begin{proof}
    Using \eqref{ineq:lora_base} and the fact that the objective in \eqref{eq:x_lora} is $(\A_k^{-1} + \lam^{-1})$-strongly convex, we have for every $v\in \R^n$,
    \begin{align*}
        \Gamma_k(v) + \frac{1}{2\A_k}\|v-x_k\|^2- \frac{\lambda+\A_k}{2\lambda \A_k}\|v-x_{k+1}\|^2 &+ \delta_k \stackrel{\eqref{eq:x_lora}}\ge \Gamma_k(x_{k+1}) + \frac{1}{2\A_k}\|x_{k+1}-x_k\|^2 + \delta_k \\
        &\stackrel{\eqref{ineq:lora_base}}\ge \frac{1}{2\lambda}\|\lambda\hat u_{k+1}\|^2+\Phi(y_{k+1})+\frac{1-\sigma}{2\lambda}\|y_{k+1}-x_k\|^2.
    \end{align*}
    The above inequality together with \eqref{ineq:Gamma} implies that
    \[
    \Phi(v) + \frac{\lambda+\A_k}{2\A_k\lambda}\|v-x_k\|^2-\frac{\lambda + \A_k}{2\A_k\lambda}\|v-x_{k+1}\|^2 + \delta_k \ge
      \frac{1}{2\lambda}\|\lambda\hat u_{k+1}\|^2+\Phi(y_{k+1})+\frac{1-\sigma}{2\lambda}\|y_{k+1}-x_k\|^2.
    \]
    Hence, \eqref{ineq:lora_one_step} immediately follows by taking $v=x_*$ and $\A_k = \infty$ and rearranging the terms.
    
    Rearranging~\eqref{ineq:lora_one_step} and discarding non-negative terms, we have for every $k\ge 0$,
    \[
        \|x_{k+1} - x_*\|^2 \le \|x_k-x_*\|^2 + 2\lambda\delta_k.
    \]
    Summing both sides from $0$ to $k-1$ gives
    \[
        \|x_{k}-x_*\|^2 \le \|x_0-x_*\|^2 + 2\lambda\sum_{i=0}^{k-1}\delta_0\leq 2\lambda C_\delta,
    \]
    which proves \eqref{ineq:lora_xkdist} using the fact that $\sqrt{a+b} \le 
    \sqrt{a} + \sqrt{b}$ for $a, b \ge 0$.
\end{proof}

\vspace{1em}

We are now ready to prove Theorem~\ref{thm:lora_complexity}, which follows directly from the single-step claim of Lemma~\ref{lem:single_step_lora}.

\vspace{1em}

\noindent
{\bf Proof of Theorem~\ref{thm:lora_complexity}:}
    Summing \eqref{ineq:lora_one_step} from $k=0$ to $k-1$ we have
    \[
       2\lam \sum_{i=0}^{k-1}[\Phi(y_{i+1})-\Phi(x_*)] + (1-\sigma)\sum_{i=0}^{k-1}\|y_{i+1}-x_i\|^2 \leq \sum_{i=0}^{k-1}\left(\|x_i-x_*\|^2-\|x_{i+1}-x_*\|^2\right)+ 2\lambda_L \sum_{i=0}^{k-1}\delta_i.
    \]
    Using the definitions of $R_0$ and $\bar \delta_k$ given in the theorem, we obtain
    \[
        2\lambda k\min_{1\leq i\leq k}[\Phi(y_{i})-\Phi(x_*)] + (1-\sigma)k\min_{1\leq i\leq k}\|y_{i}-x_{i-1}\|^2 \leq \|x_0-x_*\|^2+2\lambda k \bar{\delta}_k,
    \]
    and hence conclude the claims.
\QEDA


\subsection{FLOrA Analysis}\label{appdx:flora_framework}

This subsection is devoted to the analysis of FLOrA, which is introduced in Subsection~\ref{sec:flora}. For simplicity of notation, we omit the superscripts $\cdot^\flora$ and subscripts $\cdot_\flora$ in all proofs.

Many of the following results are analogous to those in the analysis of accelerated first-order methods, however our inclusion of the absolute error sequence $\{\delta_k^\flora\}$ in Algorithm \ref{alg:flora} requires modification of some statements and/or proofs.

\begin{lemma}\label{lem:b_seq}
    For every $k\geq 0$, the following statements hold:
        \begin{itemize}
            \item[{\rm a)}]
                $b_k^2=\tau_k\lambda_\flora B_{k+1}$;
            \item[{\rm b)}] $\tau_{k}=1+\mu_\flora B_{k}$;
            
            \item[{\rm c)}]\begin{equation*}
               B_{k+1}\geq \lambda_\flora\max\left\{\frac{(k+1)^2}{4}, \left(1+\frac{\sqrt{\lambda_\flora\mu_\flora}}{2}\right)^{2k}\right\};
            \end{equation*}
            
            \item[{\rm d)}]  recall $C_\flora=\sum_{i=0}^\infty B_{i+1}\alpha_\flora^i$ defined in Theorem \ref{thm:flora_complexity}, then $C_\flora < \infty$. Furthermore,
            denoting $\beta_\flora=\sqrt{\alpha_\flora}(1+\sqrt{\lambda_\flora\mu_\flora})<1$, we have $C_\flora\leq {\lambda_\flora}/{(1-\beta_\flora)^4}$;
            
            \item[{\rm e)}] define the sequence $\{\Delta^\flora_{k}\}$ as
            \begin{equation}\label{def:Delta}
                \Delta^\flora_{-1}=0,\quad\Delta^\flora_{k}=\frac{B_k}{B_{k+1}}\Delta^\flora_{k-1} + \delta^\flora_{k},
            \end{equation}
            then we have
        \begin{equation}\label{ineq:Delta_bound}
                \Delta^\flora_k= \frac{\delta^\flora_0}{B_{k+1}} \sum_{i=0}^kB_{i+1} (\alpha_\flora)^i \leq \frac{\delta_0^\flora C_\flora}{B_{k+1}}.
            \end{equation}
        \end{itemize} 
\end{lemma}

\begin{proof}
a) It is easy to verify that $b_k$ in \eqref{eq:bk_update} is the root of equation $b_k^2-\lambda\tau_k b_k -\lambda\tau_k B_k=0$, which is equivalent to statement a) in view of the second identity in \eqref{eq:bk_update}.

b) This statement immediately follows from the second and last equations in \eqref{eq:bk_update} and $B_0=0$.

c) This statement can be easily shown in a way similar to the proof of Proposition 1(c) of \cite{monteiro2016adaptive} and hence we omit the proof.

d) Using \eqref{eq:bk_update} and the fact that $\sqrt{a+b} \le \sqrt{a} + \sqrt{b}$ for $a, b \ge 0$, we have
    \begin{align*}
        B_{k+1}&=B_k+b_k\stackrel{\eqref{eq:bk_update}}\leq B_k+\lambda\tau_k+\sqrt{\lambda \tau_k B_k}
        \leq (\sqrt{B_k}+\sqrt{\lambda \tau_k})^2 \\
        &=(\sqrt{B_k}+\sqrt{\lambda(1+\mu B_k)})^2
        \le [(1+\sqrt{\lambda\mu})\sqrt{B_k}+\sqrt{\lambda}]^2,
    \end{align*}
    where the last identity is due to  statement b).
    We thus have the recursion
    \[
        \sqrt{B_{k+1}}\leq (1+\sqrt{\lambda\mu})\sqrt{B_k}+\sqrt{\lambda}.
    \]

    Note that $\beta:=\sqrt{\alpha}(1+\sqrt{\lambda\mu})<1$ by the requirement $\alpha < (1+\sqrt{\lambda\mu})^{-2}$ from the initialization of Algorithm~\ref{alg:flora}. Then, we obtain for all $k\geq 0$
    \[
        \sqrt{\alpha^{k+1}B_{k+2}}\leq \beta\sqrt{\alpha^{k}B_{k+1}}+\sqrt{\alpha}^{k+1}\sqrt{\lambda}.
    \]
    Unrolling with initial element $\alpha^0B_1=\lambda$, we obtain the upper bound
    \[
        \sqrt{\alpha^{k}B_{k+1}}\leq\sqrt{\lambda}\sum_{i=0}^{k}\beta^{k-i}\sqrt{\alpha}^i.
    \]
    Then, summing from 0 to $\infty$, we obtain
    \[
        \sum_{k=0}^\infty\sqrt{\alpha^{k}B_{k+1}}\leq\sqrt{\lambda}\sum_{k=0}^\infty\sum_{i=0}^{k}\beta^{k-i}\sqrt{\alpha}^i=\sqrt{\lambda}\sum_{i=0}^\infty\sqrt{\alpha}^i\sum_{k=i}^{\infty}\beta^{k-i}
        =\frac{\sqrt{\lambda}}{1-\beta}\sum_{i=0}^\infty\sqrt{\alpha}^i=\frac{\sqrt{\lambda}}{(1-\sqrt{\alpha})(1-\beta)}.
    \]
    Hence, we have
    \[
        \sum_{k=0}^\infty\alpha^{k}B_{k+1} \le \left( \sum_{k=0}^\infty\sqrt{\alpha^{k}B_{k+1}}\right)^2 \le \frac{\lambda}{(1-\sqrt{\alpha})^2(1-\beta)^2}\leq \frac{\lambda}{(1-\beta)^4},
    \]
    where the last inequality follows from $\beta\geq \sqrt{\alpha}$.


    e) It follows from \eqref{def:Delta} and $\delta_{k}=\delta_0\alpha^k$ (see Step \textbf{1} of Algorithm \ref{alg:flora}) that
    \[
    B_{k+1}\Delta_k=\sum_{i=0}^kB_{i+1}\delta_i=\delta_0\sum_{i=0}^kB_{i+1} \alpha^i.
    \]
   Hence, we prove \eqref{ineq:Delta_bound} in view of statement d).
\end{proof}

\begin{lemma}\label{lem:theta_props}
    For every $k\geq 0$, define
\begin{align}
        \theta_{k+1}(x)&=\Gamma^\flora_k(z^\flora_{k+1})-\frac{1}{2\lambda_\flora}\|z^\flora_{k+1}-\tx^\flora_{k}\|^2+\inner{u^\flora_{k+1}}{x-z^\flora_{k+1}} + \frac{\mu_\flora}{2}\|x-z^\flora_{k+1}\|^2\label{def:theta},\\
        \Theta_{k+1}(x)&=\frac{B_{k}\Theta_{k}(x) + b_{k}\theta_{k+1}(x)}{B_{k+1}}+\delta^\flora_{k}\label{def:Theta},
\end{align}
with $\Theta_0\equiv 0$.
Then, for every $k\geq 0$, the following statements hold:
\begin{enumerate}[label={\rm\alph*)}]
    \item $\theta_{k+1}$ and $\Theta_{k+1}$ are $\mu_\flora$-strongly convex quadratic functions;
    \item $x^\flora_{k} = \underset{u\in\R^n}\argmin\{B_{k}\Theta_{k}(u)+\|u-x^\flora_0\|^2/2\}$;
    \item for all $x\in\dom\Phi$, $\theta_{k+1}(x)\leq \Phi(x)$ and $\Theta_{k+1}(x)\leq \Phi(x)+\Delta^\flora_{k}$.
\end{enumerate}
\end{lemma}
\begin{proof}
    a) The statement simply follows from the definitions of $\theta_{k+1}$ and $\Theta_{k+1}$ in~\eqref{def:theta} and \eqref{def:Theta}, respectively.

    b) We prove the statement by induction. Since $B_0=0$ and $\Theta_0\equiv 0$, we trivially have the base case $k=0$.
    We assume the statement holds for some $k\geq 0$. Since $\Theta_{k}$ is a  quadratic function with $\nabla^2 \Theta_k=\mu I$, Taylor expansion around the minimizer $x_{k}$ yields the following equality for all $x\in\R^n$,
    \[ B_k\Theta_k(x)+\frac{1}{2}\|x_0-x\|^2=B_k\Theta_k(x_k) +\frac{1}{2}\|x_0-x_k\|^2+ \frac{1+B_k\mu}{2}\|x_k-x\|^2. 
    \]
    Hence, it follows from \eqref{def:Theta} that
    \[
        B_{k+1}\Theta_{k+1}(x)+\frac{1}{2}\|x_0-x\|^2=B_{k}\Theta_k(x_k)+\frac{1}{2}\|x_0-x_k\|^2 + \frac{1+B_k\mu}{2}\|x_k-x\|^2 + b_k\theta_{k+1}(x) + B_{k+1} \delta_k.
    \]
    In view of \eqref{def:theta}, the stationarity condition of $B_{k+1}\Theta_{k+1}(x)+ \|x_0-x\|^2/2$ is
    \begin{align*}
        0&=(1+B_{k}\mu)(x-x_k) + b_k u_{k+1}+ b_k\mu(x-z_{k+1}).
        \end{align*}
    It is thus straightforward to verify that $x_{k+1}$ in \eqref{def:xk_flora} is the solution to the above equation using Lemma \ref{lem:b_seq}(b). Therefore, we complete the proof by induction and conclude the statement.

    c) Noting from \eqref{eq:zkp1_hatuk} that $\hat u_{k+1} \in \Gamma_k(z_{k+1})$, which together with the first identity in \eqref{def:xk_flora} implies that
    \[  u_{k+1}\in\partial\left(\Gamma_k(\cdot)-\frac{1}{2\lam}\|\cdot-\tx_k\|^2\right).
    \]
    It follows from \eqref{ineq:Gamma_acc} and the fact that $\Gamma_k(\cdot)-\frac{1}{2\lam}\|\cdot-\tx_k\|^2$ is $\mu$-strongly convex that for every $x\in\dom\Phi$,
    \begin{align*}
    \Phi(x) &\stackrel{\eqref{ineq:Gamma_acc}}\ge \Gamma_k(x)-\frac{1}{2\lambda}\|x-\tx_k\|^2 \\
    &\ge \Gamma_k(z_{k+1})-\frac{1}{2\lambda}\|z_{k+1}-\tx_k\|^2 + \inner{u_{k+1}}{x-z_{k+1}}+\frac{\mu}{2}\|x-z_{k+1}\|^2 \stackrel{\eqref{def:theta}}= \theta_{k+1}(x),
    \end{align*}
    where the identity is due to the definition of $\theta_{k+1}$ in \eqref{def:theta}. We have thus proved the first claim.
    
    Using the definitions of $\Theta_{k+1}$ and $\Delta_k$ in \eqref{def:Theta} and \eqref{def:Delta}, respectively, we can show by induction that 
    \[
        \Theta_{k+1}(x)=\frac{\sum_{i=0}^{k}b_i\theta_{i+1}(x)}{B_{k+1}}+\frac{B_k}{B_{k+1}}\Delta_{k-1}+\delta_{k}\leq \Phi(x)+\Delta_{k},
    \]
    where the inequality follows from the first claim $\theta_{k+1}\leq \Phi$.
\end{proof}

\begin{lemma}\label{lem:min_linearization}
    For every $k\geq 0$, we have
    \begin{align}
    &\min_{x\in\R^n}\left\{\Gamma^\flora_{k}(z^\flora_{k+1})-\frac{1}{2\lambda_\flora}\|z^\flora_{k+1}-\tx^\flora_k\|^2+\inner{u^\flora_{k+1}}{x-z^\flora_{k+1}}+\frac{1}{2\lambda_\flora}\|x-\tx^\flora_k\|^2+{{\delta}^\flora_{k}}\right\} \nn \\
        \ge&
        \Phi(\ty^\flora_{k+1})+\frac{1-\sigma_\flora}{2\lambda_\flora}\|\ty^\flora_{k+1}-\tx^\flora_{k}\|^2. \label{ineq:min_linearization}
    \end{align}
\end{lemma}

\begin{proof}
In view of \eqref{def:xk_flora}, we first observe that the minimizer of the left-hand side of~\eqref{ineq:min_linearization} is
    \begin{equation}
        \hat{x}:=\tx_k - \lambda u_{k+1}\stackrel{\eqref{def:xk_flora}}=z_{k+1}-\lambda \hat u_{k+1}.\label{eq:xhat_min}
    \end{equation} 
    Using the second equation in \eqref{eq:zkp1_hatuk}, it is easy to verify that
    \begin{equation}
        \inner{u_{k+1}}{\hat{x}-z_{k+1}}\stackrel{\eqref{eq:xhat_min}}=-\inner{\hat u_{k+1}+\lambda^{-1}(\tx_k-z_{k+1})}{\lambda\hat{u}_{k+1}}\stackrel{\eqref{eq:zkp1_hatuk}}=-\frac{1}{\lambda}\|\lambda\hat u_{k+1}\|^2 - \frac{1}{\A_k}\|\tx_k-z_{k+1}\|^2.\label{eq:inner_u_min}
    \end{equation}
    Using the above relations and \eqref{eq:zkp1_hatuk} and \eqref{def:xk_flora}, we have
    \begin{align*}
        &\Gamma_{k}(z_{k+1})-\frac{1}{2\lambda}\|z_{k+1}-\tx_k\|^2+\inner{u_{k+1}}{\hat{x}-z_{k+1}}+\frac{1}{2\lambda}\|\hat{x}-\tx_k\|^2\\
        &\stackrel{\eqref{eq:xhat_min},\eqref{eq:inner_u_min}}=\Gamma_{k}(z_{k+1})-\frac{1}{2\lambda}\|z_{k+1}-\tx_k\|^2-\frac{1}{\lambda}\|\lambda\hat u_{k+1}\|^2 - \frac{1}{\A}\|\tx_k-z_{k+1}\|^2+\frac{(\A_k+\lambda)^2}{2\lambda \A_k^2}\|z_{k+1}-\tx_k\|^2\\
        &\stackrel{\eqref{eq:zkp1_hatuk}}=\Gamma_{k}(z_{k+1})-\frac{1}{2\lambda}\|\lambda \hat{u}_{k+1}\|^2.
    \end{align*}
    It follows from \eqref{ineq:flora_cond} that
    \begin{equation*}
        {{\delta}_{k}}+\Gamma_{k}(z_{k+1})-\frac{1}{2\lambda}\|\lambda \hat u_{k+1}\|^2\stackrel{\eqref{ineq:flora_cond}}\geq \Phi(\ty_{k+1})+\frac{1-\sigma}{2\lambda}\|\ty_{k+1}-\tx_k\|^2.
    \end{equation*}
    Finally, \eqref{ineq:min_linearization} immediately follows from combining the above two relations.
\end{proof}


\begin{lemma}\label{lem:Theta_ub}
    Let $X_*$ be the set of optimal solutions to \eqref{eq:general_convex_prob}. Define $R^\flora_0:=\|x^\flora_0-x_*\|=\min\{\|x^\flora_0-x\|:x\in X_*\}$. Then, for every $k\geq 0$, the following statements hold:
    \begin{itemize}
        \item[{\rm a)}] 
    \begin{equation}\label{ineq:Theta_ub}
        \min_{x\in\R^n}\left\{B_{k}\Theta_{k}(x)+\frac{1}{2}\|x-x^\flora_0\|^2\right\} \ge B_{k}\Phi(y^\flora_{k}) + \sum_{i=1}^k\frac{B_{i}(1-\sigma_\flora)}{2\lambda_\flora}\|\ty^\flora_{i}-\tx^\flora_{i-1}\|^2;
\end{equation}
        \item[{\rm b)}] 
        \begin{equation}\label{ineq:xk_bound_flora}
            \|x^\flora_k-x_*\| \le R^\flora_0 + \sqrt{2\delta^\flora_0C_\flora},
        \end{equation} where $C_\flora$ is as in Theorem \ref{thm:flora_complexity}.
    \end{itemize}
\end{lemma}

\begin{proof} 
    a) The statement follows by induction. Relation \eqref{ineq:Theta_ub} trivially holds for $k=0$ since $B_0=0$. Then we assume the claim holds for some $k\geq 0$. For convenience, we denote
    \begin{equation}\label{def:betak}
        \beta_k =\sum_{i=1}^k\frac{B_{i}(1-\sigma)}{2\lambda}\|\ty_{i}-\tx_{i-1}\|^2.
    \end{equation}
    It follows from Lemma~\ref{lem:theta_props}(a) and (b) that for every $u \in \R^n$ that
    \begin{align}
        B_{k}\Theta_{k}(u)+\frac{1}{2}\|u-x_0\|^2
        &\geq B_{k}\Theta_{k}(x_{k}) +\frac{1}{2}\|x_k-x_0\|^2+\frac{1+\mu B_k}{2}\|x_k-u\|^2 \nn \\
        &\geq B_k\Phi(y_{k})+\beta_k+\frac{1+\mu B_k}{2}\|x_k-u\|^2 \label{ineq:used_for_xk_bound},
    \end{align}
    where the second inequality follows from the inductive hypothesis.
    Using \eqref{def:Theta} and \eqref{ineq:used_for_xk_bound}, we have
    \begin{align*}
        B_{k+1}\Theta_{k+1}(u)&+\frac{1}{2}\|u-x_0\|^2 - B_{k+1} \delta_k - b_{k}\theta_{k+1}(u) \stackrel{\eqref{def:Theta}}= B_{k}\Theta_{k}(u)+\frac{1}{2}\|u-x_0\|^2 \\
        &\stackrel{\eqref{ineq:used_for_xk_bound}}\geq B_k\Phi(y_{k})+\beta_k+\frac{1+\mu B_k}{2}\|x_k-u\|^2 
        \ge B_k\theta_{k+1}(y_{k})+\beta_k+ \frac{\tau_k}{2}\|x_k-u\|^2,
    \end{align*}
    where the last inequality follows from Lemmas~\ref{lem:theta_props}(c) and~\ref{lem:b_seq}(b).
    For $u\in\R^n$, define $\tilde{u}=B_{k+1}^{-1}(b_ku+B_ky_k)$. Rearranging the terms, we have
    \begin{align}
        B_{k+1}\Theta_{k+1}(u)&+\frac{1}{2}\|u-x_0\|^2 
        \ge B_{k+1}\delta_k+b_{k}\theta_{k+1}(u) + B_k\theta_{k+1}(y_{k})+\beta_k+ \frac{\tau_k}{2}\|x_k-u\|^2 \nn \\
        \stackrel{\eqref{eq:tx_flora}}\geq& B_{k+1}\theta_{k+1}(\tilde u)+\beta_k+B_{k+1}{{\delta}_{k}}+\frac{\tau_k B_{k+1}^2}{2b_k^2}\|\tx_k-\tilde{u}\|^2 \nn \\
        =&B_{k+1}\theta_{k+1}(\tilde u)+\beta_k+B_{k+1}{{\delta}_{k}}+\frac{B_{k+1}}{2\lambda}\|\tx_k-\tilde{u}\|^2.\label{ineq:Theta_chain_final}
    \end{align}
    where the second inequality is due to the convexity of $\theta_{k+1}$ and~\eqref{eq:tx_flora} and the identity is due to Lemma~\ref{lem:b_seq}(a). 
    It follows from the definition of $\theta_{k+1}$ in~\eqref{def:theta} that
    \[
    \theta_{k+1}(u)\geq\theta_{k+1}(u) -\frac{\mu}{2}\|x-z_{k+1}\|^2\stackrel{\eqref{def:theta}}=\Gamma_k(z_{k+1})-\frac{1}{2\lambda}\|z_{k+1}-\tx_k\|^2+\inner{u_{k+1}}{u-z_{k+1}}.
    \] 
    This inequality and \eqref{ineq:Theta_chain_final} imply that
    \begin{align*}
    &B_{k+1}\Theta_{k+1}(u) +\frac{1}{2}\|u-x_0\|^2 \\
        &\stackrel{\eqref{ineq:Theta_chain_final}}\geq B_{k+1}\left(\Gamma_{k}(z_{k+1})-\frac{1}{2\lambda}\|z_{k+1}-\tx_k\|^2+\inner{u_{k+1}}{\tilde{u}-z_{k+1}} + \delta_k+\frac{1}{2\lambda}\|\tilde{u}-\tx_k\|^2\right)
        +\beta_k.
    \end{align*}
    Minimizing over both sides of the above inequality, we obtain
    \begin{align*}
        &\min_{x\in \R^n} \left\{B_{k+1}\Theta_{k+1}(x) +\frac{1}{2}\|x-x_0\|^2\right\}\\
        &\ge B_{k+1}\min_{x\in\R^n}\left\{\Gamma_{k}(z_{k+1})-\frac{1}{2\lambda}\|z_{k+1}-\tx_k\|^2+\inner{u_{k+1}}{x-z_{k+1}} + \delta_k+\frac{1}{2\lambda}\|x-\tx_k\|^2\right\}
        +\beta_k\\
        &\stackrel{\eqref{ineq:min_linearization}}\geq B_{k+1}\Phi(\ty_{k+1})+\frac{B_{k+1}(1-\sigma)}{2\lambda}\|\ty_{k+1}-\tx_k\|^2 +\beta_k,
    \end{align*}
    where the last inequality follows by Lemma~\ref{lem:min_linearization}.
    Finally, the target inequality \eqref{ineq:Theta_ub} directly follows from the fact that $\Phi(\ty_{k+1}) \ge \Phi(y_{k+1})$ (see Step \textbf{3} of Algorithm \ref{alg:flora}) and the observation that $\beta_{k+1}= \beta_k + B_{k+1}(1-\sigma)\|\ty_{k+1}-\tx_k\|^2/(2\lambda)$ in view of \eqref{def:betak}.

    b) It follows from \eqref{ineq:used_for_xk_bound} with $u=x_*$ that
    \[
         B_k(\Theta_{k}(x_*)-\Phi(y_k)) \stackrel{\eqref{ineq:used_for_xk_bound}}\geq \frac{1+\mu B_k}{2}\|x_*-x_k\|^2-\frac{1}{2}\|x_*-x_0\|^2 + \beta_k
        \geq\frac{1}{2}\|x_*-x_k\|^2-\frac{1}{2}\|x_*-x_0\|^2.
    \]
    Using the second inequality in Lemma~\ref{lem:theta_props}(c) with $x=x_*$ and Lemma~\ref{lem:b_seq}(e), we have
    \begin{align*}
        B_k(\Theta_{k}(x_*)-\Phi(y_k)) \le B_k(\Phi(x_*)-\Phi(y_k) + \Delta_{k-1}) \le 
        B_k\Delta_{k-1} \stackrel{\eqref{ineq:Delta_bound}}\le \delta_0 \sum_{i=0}^{k-1} B_{i+1} \alpha^i.
    \end{align*}
    In view of the definition of $C_\flora$ in Lemma \ref{lem:b_seq}(c), the statement directly follows from combining the above two inequalities. 
\end{proof}

\vspace{1em}

We are now ready to prove Theorem~\ref{thm:flora_complexity}, which directly follows from Lemmas~\ref{lem:b_seq} and~\ref{lem:Theta_ub}.

\vspace{1em}

\noindent
{\bf Proof of Theorem~\ref{thm:flora_complexity}:}
     It follows from Lemma~\ref{lem:Theta_ub}(a),
    \begin{align*}
         B_{k+1}\Phi(y_{k+1}) + \sum_{i=1}^{k+1}\frac{B_{i}(1-\sigma)}{2\lambda}\|\ty_{i}-\tx_{i-1}\|^2
         &\stackrel{\eqref{ineq:Theta_ub}}\leq B_{k+1}\Theta_{k+1}(x_*)+\frac{1}{2}\|x_*-x_0\|^2 \\
         &\leq B_{k+1}\Phi(x_*)+\frac{1}{2}\|x_*-x_0\|^2 + B_{k+1}\Delta_{k}
    \end{align*}
    where the second inequality is due to Lemma~\ref{lem:theta_props}(c) with $x=x_*$.
    Using Lemma~\ref{lem:b_seq}(e) yields
    \[
        B_{k+1}(\Phi(y_{k+1})-\Phi(x_*)) + \sum_{i=1}^{k+1}\frac{B_{i}(1-\sigma)}{2\lambda}\|\ty_{i}-\tx_{i-1}\|^2\leq\frac{1}{2}\|x_*-x_0\|^2+{\delta_0C_\flora}.
    \]
    Therefore, \eqref{ineq:flora_a}-\eqref{ineq:flora_c} immediately follow.
\QEDA

In the course of our analysis in Subsection~\ref{ssec:proofs_falm} (see Proposition \ref{prop:OR}), we found it necessary to uniformly bound the distance from the prox center $\tx_k^\flora$ to the minimum distance optimizer $x_*$. 
The following lemma provides a uniform upper bound on $\|\tx^\flora_{k}-x_*\|$ over the iterations $k\geq 0$.
Part 1 in the proof of Lemma \ref{lem:distance_bound_flora} below largely follows~\cite[Theorem 3.10]{MonteiroSvaiterAcceleration}.

\begin{lemma}\label{lem:distance_bound_flora}
    Suppose we choose $y^\flora_k=\ty^\flora_k$ in Step \textbf{3} of Algorithm \ref{alg:flora} and $\sigma_\flora < 1$, and suppose that $x^\flora_0=0$. Define $\bar{R}^\flora_0=\max\{1, R^\flora_0\}$, where $R^\flora_0:=\|x^\flora_0-x_*\|=\min\{\|x^\flora_0-x\|:x\in X_*\}$, where $X_*$ is the optimal solution set to~\eqref{eq:general_convex_prob}. Then, for every $k\geq 0$, we have
    \begin{equation}\label{ineq:tx_dist_bound}
        \|\tx^\flora_{k}-x_*\|\leq \mathcal{R}_\flora,
    \end{equation}
    where $\mathcal{R}_\flora$ is defined as
    \begin{equation}\label{ineq:cal_R}
        \mathcal{R}_\flora:=\bar{R}_0^\flora(1 + \sqrt{2\delta^\flora_{0}C_\flora})\left(\frac{2}{\sqrt{1-\sigma_\flora}}+1\right),
    \end{equation}
    and $C_\flora$ is as in Theorem \ref{thm:flora_complexity}.
\end{lemma}

\begin{proof} 
We prove the claim in three parts. First, we provide an upper bound on $\|y_{k+1}-x_*\|$. Second, we combine the previous bound with Lemma~\ref{lem:Theta_ub}(b) to bound $\|\tx_k-x_*\|$.

    \textbf{Part 1)}
    First, we show by induction that for every $k\geq 0$, 
    \begin{equation}\label{ineq:hypo}
        \|y_{k+1}-x_*\|\leq \frac{1}{B_{k+1}}\sum_{i=0}^{k}B_{i+1}\|\tx_i-y_{i+1}\| + R_0+\sqrt{2\delta_0C_\flora}.
    \end{equation}
    For $k=0$, observe that $b_0=B_{1}$ and $\tx_0=x_0$ in view of \eqref{eq:bk_update} and \eqref{eq:tx_flora}, the claim \eqref{ineq:hypo} follows directly from the triangle inequality
    \[
        \|y_{1}-x_*\|\leq\|x_0-y_{1}\|+ \|x_0-x_*\|=\frac1{B_1}\sum_{i=0}^0B_{i+1}\|\tx_i-y_{i+1}\| + R_0,
    \]
    which proves the base case.
        We now perform the inductive step.
First, applying the triangle inequality twice and using \eqref{eq:tx_flora}, we have
    \[
        \|y_{k+1}-x_*\|\leq \|y_{k+1}-\tx_k\|+\|x_*-\tx_k\|\stackrel{\eqref{eq:tx_flora}}{\leq} \|y_{k+1}-\tx_k\|+\frac{B_k}{B_{k+1}}\|x_*-y_k\|+\frac{b_k}{B_{k+1}}\|x_*-x_k\|.
    \]
    It thus follows from Lemma~\ref{lem:Theta_ub}(b) that
    \[
        \|y_{k+1}-x_*\| \stackrel{\eqref{ineq:xk_bound_flora}}\leq \|y_{k+1}-\tx_k\|+\frac{B_k}{B_{k+1}}\|x_*-y_k\|+\frac{b_k}{B_{k+1}}(R_0 + \sqrt{2\delta_0C_\flora}).
    \]
    Applying the inductive hypothesis \eqref{ineq:hypo} with $k+1$ replaced by $k$, we obtain
        \begin{align*}
        \|y_{k+1}-x_*\|\stackrel{\eqref{ineq:hypo}}\leq & \|y_{k+1}-\tx_k\|+\frac{1}{B_{k+1}}\sum_{i=0}^{k-1}B_{i+1}\|\tx_i-y_{i+1}\|
        + \frac{B_k}{B_{k+1}} (R_0 + \sqrt{2\delta_0C_\flora}) \\
        +&\frac{b_k}{B_{k+1}} (R_0 + \sqrt{2\delta_0C_\flora})
        = \frac{1}{B_{k+1}}\sum_{i=0}^{k}B_{i+1}\|\tx_i-y_{i+1}\| + R_0+\sqrt{2\delta_0C_\flora}.
    \end{align*}
    Hence, we prove \eqref{ineq:hypo} holds for every $k \ge 0$.

    It follows from \eqref{ineq:flora_b} from Theorem \ref{thm:flora_complexity} and \eqref{ineq:hypo} that
    \begin{align*}
        \|y_{k+1}-x_*\|\leq \left(\frac{\sqrt{\lam}}{\sqrt{1-\sigma}B_{k+1}}\sum_{i=0}^{k}\sqrt{B_{i+1}} + 1\right) (R_0+\sqrt{2\delta_0C_\flora}).
    \end{align*}
    Since $\{B_{k}\}$ is increasing, we have
    \begin{equation}\label{ineq:y-x}
        \|y_{k+1}-x_*\|\leq \left(\frac{\sqrt{\lambda}(k+1)}{\sqrt{1-\sigma}\sqrt{B_{k+1}}} +1\right) (R_0 + \sqrt{2\delta_0C_\flora}) \le \left(\frac{2}{\sqrt{1-\sigma}} +1\right) (R_0 + \sqrt{2\delta_0C_\flora}),
    \end{equation}
    where the second inequality follows from the fact that $B_{k+1} \ge \lambda(k+1)^2/4$ from Lemma~\ref{lem:b_seq}(c).
    Using the definition $\bar R_0=\max\{1, R_0\}$, we note that $R_0+\sqrt{2\delta_0C_\flora}\leq \bar{R}_0(1+\sqrt{2\delta_0C_\flora})$. Therefore, we conclude from \eqref{ineq:y-x} and the definition of $\mathcal{R}_\flora$ in \eqref{ineq:cal_R} that 
    \begin{equation}\label{ineq:dist_R}
        \|y_{k+1}-x_*\|\leq \mathcal{R}_\flora.
    \end{equation}

    \textbf{Part 2)} Next, combining the triangle inequality with the fact that $(\alpha a + \beta b)^2\leq (\alpha +\beta)(\alpha a^2+\beta b^2)$ for $a,b,\alpha,\beta\in\R_{++}$, we have
        \begin{align*}
            \|\tx_k-x_*\|^2&\stackrel{\eqref{eq:tx_flora}}\leq\left(\frac{B_k}{B_{k+1}}\|y_k-x_*\|+\frac{b_k}{B_{k+1}}\|x_k-x_*\|\right)^2\\
            &\leq \left(\frac{B_k}{B_{k+1}} + \frac{b_k}{B_{k+1}}\right)\left(\frac{B_k}{B_{k+1}}\|y_k-x_*\|^2+\frac{b_k}{B_{k+1}}\|x_k-x_*\|^2\right)\\
            &\stackrel{\eqref{ineq:dist_R}}{\leq }\frac{B_k}{B_{k+1}}\mathcal{R}_\flora^2+\frac{b_k}{B_{k+1}}\|x_k-x_*\|^2\stackrel{\eqref{ineq:xk_bound_flora}}\leq \frac{B_k}{B_{k+1}}\mathcal{R}_\flora^2+\frac{b_k}{B_{k+1}}(R_0+\sqrt{2\delta_0C_\flora})^2 \le \mathcal{R}_\flora^2,
        \end{align*}
    where the third inequality follows by Part 1 and the fact that $B_{k+1}=B_k+b_k$ (see \eqref{eq:bk_update}), the fourth inequality by Lemma~\ref{lem:Theta_ub}(b), and the final one by
    $R_0+\sqrt{2\delta_0C_\flora}\leq \bar{R}_0(1+\sqrt{2\delta_0C_\flora})\leq \mathcal{R}_\flora$.
\end{proof}

\section{Deferred Proofs for ACG and Restarted ACG}\label{appdx:primal_deferred}
\subsection{FLOrA Analysis of Algorithm~\ref{alg:ACG}}\label{appdx:acg}

In this subsection, we provide a self-contained analysis of Algorithm~\ref{alg:ACG} by showing that it is an instance of the FLOrA framework.

Clearly the scalar sequences in Algorithm~\ref{alg:flora} and Algorithm~\ref{alg:ACG} are equivalent with $b_k=a_j$, $B_k=A_j$, $\tau_k=\tau_j$, $\mu_\flora=\mu$, and $\lambda_\flora=1/(2L)$. We can then restate Lemma~\ref{lem:b_seq} for Algorithm~\ref{alg:ACG}.
\begin{lemma}\label{lem:101}
    The following statements hold for every $ j\ge 0 $:
    \begin{itemize}
        \item[a)] $ 2La_j^2=A_{j+1}\tau_{j} $;
        \item[b)] $ \tau_j=1+\mu A_j $;
        \item[c)] 
        \[
            A_{j+1}\geq\frac{1}{2L}\max\left\{\frac{(j+1)^2}{4}, \left(1+\frac{1}{2}\sqrt{\frac{\mu}{2L}}\right)^{2j}\right\}.
        \]
    \end{itemize}
\end{lemma}

We start by defining our lower model. On each iteration $j$, define     
\begin{gather}\label{def:Gamma}
        \Gamma_j(x) = \ell_{g}(x;\tx_j)+h(x) + \frac{2L + \mu}{2}\|u-\tx_j\|^2,
    \end{gather}
where $\Gamma_j$ is in fact the objective function in \eqref{def:tyj}.
To match the algorithm statements, we denote FLOrA iterates with $k$ and ACG iterates with $j$. Recalling the objective function $\psi(\cdot)$ defined in~\eqref{eq:prob},  we will show next that Algorithm~\ref{alg:ACG} is an instance of the FLOrA framework (i.e., Algorithm~\ref{alg:flora}) with the correspondence
\begin{eqbox}
\begin{equation}
\begin{gathered}
    \Phi(\cdot)=\psi(\cdot),\quad\Gamma^\flora_k(\cdot) = \Gamma_j(\cdot),\quad \A_k^\flora=\infty,\quad\delta^\flora_k=\alpha_\flora=0, \quad \mu_\flora = \mu,\quad \sigma_\flora=1/2;\\
    \lambda_\flora=\frac{1}{2L},\quad
    y^\flora_{k} = y_j,\quad
   z^\flora_{k}= \tilde y^\flora_{k} = \tilde y_{j},\quad
    x^\flora_k = x_j,\quad u^\flora_{k}=u_j:=2L({\tx_{j-1}-\ty_{j}}),\quad \hat u^\flora_{k}=0.\label{def:acg_flora_corr}
\end{gathered}
\end{equation}
\end{eqbox}


First, We show that the lower model $\Gamma_j$ and $\ty_{j+1}$ satisfy~\eqref{ineq:Gamma_acc} and~\eqref{eq:zkp1_hatuk} with $\A_k^\flora=\infty$.
\begin{lemma}\label{lem:acg_Gamma}
    Consider $\Gamma_j(x)$ defined in~\eqref{def:Gamma}. Then, the following statements hold:
    \begin{enumerate}[label={\rm\alph*)}]
        \item $\Gamma_j(u)\leq \psi(u)+L\|u-\tx_j\|^2$ for every $u \in \R^n$;
        \item $\ty_{j+1}=\underset{u\in\R^n}\argmin\Gamma_j(u)$.
    \end{enumerate}
    Moreover, the two statements satisfy~\eqref{ineq:Gamma_acc} and~\eqref{eq:zkp1_hatuk} with the correspondence~\eqref{def:acg_flora_corr}.
\end{lemma}
\begin{proof}
    a) The claim follows by the $\mu$-strong convexity of $g$ and the definition of $\Gamma_j$ in~\eqref{def:Gamma}.

    b) The claim follows directly from the definitions of $\ty_{j+1}$ in~\eqref{def:tyj} and $\Gamma_j$ in~\eqref{def:Gamma}.

    Finally, it is easy to verify the final claim with the correspondence \eqref{def:acg_flora_corr}.
\end{proof}

 Combining these properties with the $(L+\mu)$-smoothness and $\mu$-strong convexity of $g$, we show that $\Gamma_j$ satisfies the key inequality~\eqref{ineq:flora_cond} in FLOrA.
\begin{lemma}\label{lem:acg_ALAM_ineq}
    Consider $\Gamma_j(x)$ defined in~\eqref{def:Gamma}. Then, for every $j\geq 0$,
    \begin{equation}\label{ineq:psi}
        \frac{1}{L}\left[\psi(\tilde y_{j+1})+L\|\ty_{j+1}-\tx_j\|^2-\Gamma_j(\tilde y_{j+1})\right]\leq \frac{1}{2}\|\ty_{j+1}-\tx_j\|^2.
    \end{equation}
    Moreover, \eqref{ineq:psi}
    satisfies \eqref{ineq:flora_cond} with the correspondence \eqref{def:acg_flora_corr}.
\end{lemma}
\begin{proof}
    Using the definition of $\Gamma_j$ in~\eqref{def:Gamma} and the $(L+\mu)$-smoothness of $g$, we have
    \begin{align*}
        \psi(\ty_{j+1})+L\|\ty_{j+1}-\tx_j\|-\Gamma_j(\ty_{j+1})&\stackrel{\eqref{def:Gamma}}=g(\ty_{j+1})-\ell_{g}(\ty_{j+1};\tx_j)-\frac{\mu}{2}\|\ty_{j+1}-\tx_j\|^2 \\
        &\leq \frac{L+\mu}{2}\|\ty_{j+1}-\tx_{j}\|^2-\frac{\mu}{2}\|\ty_{j+1}-\tx_{j}\|^2
        =\frac{L}{2}\|\ty_{j+1}-\tx_{j}\|^2.
    \end{align*}
   Hence, \eqref{ineq:psi} immediately follows.
   Finally, it is easy to verify the final claim with the correspondence \eqref{def:acg_flora_corr}.
\end{proof}

Finally, we show that the auxiliary sequence $\{x_{j+1}\}$ with update~\eqref{def:xj} is equivalent to the FLOrA sequence $\{x_{k+1}^\flora\}$ with update~\eqref{def:xk_flora}. 

\begin{lemma}\label{lem:acg_xj}
    Choosing $u_{j+1}$ as in~\eqref{def:acg_flora_corr}, we can rewrite $x_{j+1}$ in \eqref{def:xj} as
    \begin{equation}\label{eq:x-new}
        x_{j+1}=\frac{1}{\tau_{j+1}}\left(\tau_j x_j-a_ju_{j+1}+\mu a_j\ty_{j+1}\right).
    \end{equation}
    Moreover,~\eqref{def:xj} is equivalent to~\eqref{def:xk_flora} with the correspondence~\eqref{def:acg_flora_corr}.
\end{lemma}
\begin{proof}
    Using the definition of $x_{j+1}$ in \eqref{def:xj} and Lemma~\ref{lem:101}(a) and (b), we have
    \begin{align*}
        x_{j+1}&\stackrel{\eqref{def:xj}}=\frac{1}{\tau_{j+1}}\left(2La_j\ty_{j+1}+\tau_j x_j-\frac{2La_j^2}{A_{j+1}}x_j+\mu a_j\ty_{j+1} - 2L\frac{A_ja_j}{A_{j+1}}y_j\right)\\
        &\stackrel{\eqref{def:tx}}=\frac{1}{\tau_{j+1}}\left(\tau_j x_j+2La_j\ty_{j+1}-2La_j\tx_j+\mu a_j\ty_{j+1}\right)\\
        &=\frac{1}{\tau_{j+1}}\left(\tau_j x_j-a_ju_{j+1}+\mu a_j\ty_{j+1}\right).
        \end{align*}
        where the second identity is due to \eqref{def:tx} and the last one holds by our choice of $u_{j+1}$ in \eqref{def:acg_flora_corr}.
        Finally, it is easy to verify the last claim in view of \eqref{def:acg_flora_corr} and \eqref{eq:x-new}.
\end{proof}

Having shown that Algorithm~\ref{alg:ACG} is an instance of~Algorithm \ref{alg:flora}, then Theorem~\ref{thm:flora_complexity} holds in the context of this subsection using the translation in~\eqref{def:acg_flora_corr}. Therefore, we can present a simple proof of Lemma~\ref{lem:acg_convergence} based on the correspondence in~\eqref{def:acg_flora_corr}.

\vspace{1em}
\noindent
\textbf{Proof of Lemma~\ref{lem:acg_convergence}:} 
From the correspondence~\eqref{def:acg_flora_corr} and Theorem~\ref{thm:flora_complexity} with $R^\flora_0=R_0$ and $B_k=A_j$, we obtain
\[
    \psi(y_{j})-\psi_*\stackrel{\eqref{def:acg_flora_corr}}{=}\Phi(y^\flora_{k})-\Phi_*\stackrel{\eqref{ineq:flora_a},\eqref{def:acg_flora_corr}}\leq\frac{R_0^2}{2A_{j}},
\]
which proves~\eqref{ineq:acg_primal_gap}.

The second claim follows by first noting that, by the definition of the gradient map,
\begin{equation}
    \mathcal{G}_\psi^{(2L+\mu)^{-1}}(\tx_{j-1})\stackrel{\eqref{def:grad_mapping}}=(2L+\mu)(\tx_{j-1}-y_j).\label{eq:grad_mapping_acg}
\end{equation}
Then, once again applying Theorem~\ref{thm:flora_complexity} with $R^\flora_0=R_0$ and $B_k=A_j$ under the correspondence~\eqref{def:acg_flora_corr}, we obtain
\[
    \|\mathcal{G}_\psi^{(2L+\mu)^{-1}}(\tx_{j-1})\|\stackrel{\eqref{eq:grad_mapping_acg}}{=}(2L+\mu)\|y_j-\tx_{j-1}\|\stackrel{\eqref{def:acg_flora_corr}}=(2L+\mu)\|y^\flora_k-\tx^\flora_{k-1}\|\stackrel{\eqref{ineq:flora_b}\eqref{def:acg_flora_corr}}\leq \frac{(2L+\mu)R_0}{\sqrt{LA_j}},
\]
which proves~\eqref{ineq:acg_grad_mapping}.
\QEDA

Similarly, since Algorithm~\ref{alg:ACG} is an instance of FLOrA, the following lemma is a direct consequence of Lemmas~\ref{lem:theta_props} and \ref{lem:Theta_ub}(a) under the correspondence~\eqref{def:acg_flora_corr}. The proof is omitted, since all results directly follow by substitution from~\eqref{def:acg_flora_corr}.

	\begin{lemma}\label{lem:gamma}
		For all $j\geq 0$, let $\theta_j$ and $\Theta_j$ be as defined in~\eqref{def:theta_acg} and~\eqref{def:Theta_acg}, respectively.
        Then, the following statements hold for every $j \ge 0$:
		
\begin{enumerate}[label={\rm\alph*)}]
    \item $\theta_{j+1}$ and $\Theta_{j+1}$ are $\mu$-strongly convex quadratic functions;
    \item $x_{j} = \underset{x\in\R^n}\argmin\{A_{j}\Theta_{j}(x)+\|x-x_0\|^2/2\}$;
    \item for all $x\in\dom\psi$, $\theta_{j+1}(x)\leq \psi(x)$ and $\Theta_{j+1}(x)\leq \psi(x)$;
    \item \begin{equation}\label{ineq:induction}
		    A_j \psi(y_j)\le \underset{u\in \R^n}\min\left\lbrace A_j\Theta_j(u)+\frac12 \|u-x_0\|^2\right\rbrace.
		\end{equation}
\end{enumerate}
	\end{lemma}

\vspace{1em}

We are now ready to prove Lemma~\ref{lem:tech}, which is the key result enabling the inner complexity bound in Proposition~\ref{prop:inner_to_outer}.

\vspace{1em}

\noindent
\textbf{Proof of Lemma~\ref{lem:tech}:}  
		a) It follows from Lemma \ref{lem:gamma}(d) that
		\begin{equation} 
		    \psi (y_j)
		\stackrel{\eqref{ineq:induction}}\le \underset{u\in \R^n}\min \left \{ \Theta_j (u) + \frac1{2A_j} \|u-x_0\|^2 \right\} \le  \Theta_j( \hat x_j) +  \frac1{2A_j} \|\hat x_j-x_0\|^2. \label{ineq:psi-new}
		\end{equation}
        Using Lemma \ref{lem:gamma}(a) and~\eqref{def:hatxj_sj}, we have for every $u \in \R^n$,
		\[
		\psi (y_j)  - \frac1{2A_j} \|\hat x_j-x_0\|^2 \le \Theta_j( \hat x_j) \stackrel{\eqref{def:hatxj_sj}}\leq \Theta_j(u)- \frac{\mu}{2}\|u-\hat x_j\|^2.
		\]
		Taking $u = y_j$ in the above inequality and using Lemma~\ref{lem:gamma}(c), we obtain
		\[
		\|y_j-\hat x_j\|^2 \le \frac{1}{\mu A_j} \|\hat x_j-x_0\|^2.
		\]
		Using the above inequality, the triangle inequality, and the fact that $(a+b)^2\le 2(a^2+b^2)$, we have
		\[
		\|\hat x_j-x_0\|^2 \le 2(\|\hat x_j-y_j\|^2+\|y_j-x_0\|^2)
		\le \frac{2}{\mu A_j}\|\hat x_j-x_0\|^2 + 2\|y_j-x_0\|^2.
		\]
		Hence, \eqref{ineq:mj_mu} follows from the assumption that $A_j\ge 3/\mu$ and \eqref{ineq:psi-new}.
		
		b) 	It follows from Lemma~\ref{lem:gamma}(d) that for any $u\in\R^n$
		\begin{align*}
		\psi(y_j)+\frac12 \left(\mu + \frac{1}{A_j}\right) \|u- x_j\|^2 &\stackrel{\eqref{ineq:induction}}\le \underset{u\in \R^n}\min\left\lbrace  \Theta_j(u)+ \frac{1}{2A_j} \|u-x_0\|^2\right\rbrace +\frac12 \left( \mu + \frac{1}{A_j}\right) \|u- x_j\|^2\\
		&\le  \Theta_j(u)+ \frac{1}{2A_j} \|u-x_0\|^2,
		\end{align*}
        where the second inequality follows from Lemma \ref{lem:gamma}(a) and (b).
		Taking $ u=y_j $ in the above inequality and Lemma \ref{lem:gamma}(c), we have
		\[
		\frac12 \left( \mu + \frac{1}{A_j}\right) \|y_j- x_j\|^2 \le  \Theta_j(y_j) - \psi(y_j)+ \frac{1}{2A_j} \|y_j-x_0\|^2 \le \frac{1}{2A_j} \|y_j-x_0\|^2,
		\]
		and hence 
		\[
		\|y_j- x_j\|^2 \le \frac{1}{1 + A_j \mu}\|y_j-x_0\|^2 \le \frac14 \|y_j-x_0\|^2
		\]
		where the second inequality is due to $ A_j\ge 3/\mu $. Finally, \eqref{ineq:sj} immediately follows from the above inequality, the triangle inequality and the definition of $ s_j $ in \eqref{def:hatxj_sj}.
	\QEDA

\subsection{Deferred Proofs from Subsection~\ref{ssec:proofs_restart}}\label{appdx:deferred_restart}

We begin by proving that Algorithm~\ref{alg:restart} is an instance of the FLOrA framework. The proof directly follows from substituting terms using the correspondence~\eqref{def:restart_corresp}.

\vspace{1em}

\noindent
{\bf Proof of Lemma \ref{lem:restart_Gamma_eq}:}
        Clearly, the scalar sequences $b_k$, $B_k$, and $\tau_k$ are equivalent. Similarly, $y^\flora_k=w_k$ given the choice $\ty^\flora_{k+1}=y_j$, and $\tx^\flora_k$ in~\eqref{eq:tx_flora} is equivalent to $\tilde{v}_k$ in~\eqref{def:tv} given $y^\flora_k=w_k$ and the choice of $x^\flora_k=v_k$ in~\eqref{def:restart_corresp}. Then we need to show the FLOrA conditions~\eqref{ineq:Gamma_acc},~\eqref{ineq:flora_cond},~\eqref{eq:zkp1_hatuk} hold and the update~\eqref{def:xk_flora} is equivalent to~\eqref{def:vkp1}. 

        First, note that $\Theta_j$ is $(\mu_f+\lambda^{-1})$-strongly convex in view of Lemma~\ref{lem:gamma}(a) and the choice of $\mu_\flora$ in~\eqref{eq:setup}, matching the condition on $\Gamma^\flora_k$ in Step \textbf{2} of Algorithm \ref{alg:flora}. Condition~\eqref{ineq:Gamma_acc} holds by Lemma~\ref{lem:gamma}(c) and the definition of $\psi$ in~\eqref{eq:setup}.
        
        Second, the first relation in \eqref{eq:zkp1_hatuk} holds by Lemma~\ref{lem:gamma}(b)
        in view of $\tx^\flora_k = \tilde v_k= x_0$ (see \eqref{eq:setup} and \eqref{def:restart_corresp}) and the choices $z^\flora_{k+1}=x_j$, $\Gamma^\flora_k = \Theta_j$, and $\A_k^\flora = A_j$ given in \eqref{def:restart_corresp}. Similarly, the second relation in~\eqref{eq:zkp1_hatuk} follows from the definition of $s_j$ in \eqref{def:hatxj_sj}.

        Third, we prove the equivalence of \eqref{ineq:flora_cond} and \eqref{ineq:lora_restart_acg}.
        By the correspondence~\eqref{def:restart_corresp} and the relation $\tx_k^\flora = \tilde v_k= x_0$ noted above, we have
           \begin{gather*}
               \|\lambda_\flora\hat{u}^\flora_{k+1}\|^2+2\lambda_\flora\left[\Phi(\ty^\flora_{k+1})+\frac{1}{2\lambda_\flora}\|\tx^\flora_k-\ty^\flora_{k+1}\|^2 -\Gamma^\flora_k(z^\flora_{k+1})\right]\\\stackrel{\eqref{def:restart_corresp}}=\|\lambda s_j\|^2+2\lambda\left[\phi(y_j)+\frac{1}{2\lambda}\|x_0-y_j\|^2 -\Theta_j(x_j)\right]\stackrel{\eqref{ineq:lora_restart_acg}}\leq \sigma\|y_j-x_0\|^2\stackrel{\eqref{def:restart_corresp}}=\sigma_\flora\|\ty^\flora_{k+1}-\tx^\flora_k\|^2 
           \end{gather*}
        which proves~\eqref{ineq:flora_cond} is equivalent to~\eqref{ineq:lora_restart_acg}. 

        Finally we verify the equivalence of~\eqref{def:xk_flora} and~\eqref{def:vkp1}. Given our choice of $z^\flora_{k+1}=x_j$ and the relation $\tx^\flora_k=\tilde v_k=x_0$, we observe
        \begin{equation*}
            u^\flora_{k+1}\stackrel{\eqref{def:xk_flora}}=\hat{u}^\flora_{k+1}+\frac{\tx^\flora_k-z^\flora_{k+1}}{\lambda}\stackrel{\eqref{def:restart_corresp}}=s_j+\frac{x_0-x_j}{\lambda}
            \stackrel{\eqref{def:hatxj_sj}}=s_j+\frac{A_j}{\lambda}s_j=\frac{\lambda+A_j}{\lambda}s_j,
        \end{equation*}
        which validates our choice of $u^\flora_{k+1}$ in~\eqref{def:restart_corresp}. Then with the choices $x^\flora_k=v_k$, $z^\flora_{k+1}=x_j$, $\mu_\flora=\mu_f$, and $u^\flora_{k+1}=\lambda^{-1}(A_j+\lambda)s_j$ given in \eqref{def:restart_corresp}, we can rewrite \eqref{def:vkp1} as
        \[
        v_{k+1}\stackrel{\eqref{def:vkp1},\eqref{def:restart_corresp}}=\frac{1}{\tau_{k+1}}\left(\tau_k x_k^\flora+b_k\mu_\flora z_{k+1}^\flora-b_ku^\flora_{k+1}\right)
        \stackrel{\eqref{def:xk_flora}}= x^\flora_{k+1}.
        \]
        Therefore, Algorithm~\ref{alg:restart} is an instance of the FLOrA framework.
    \QEDA

Using the results from Subsection~\ref{subsec:ACG} and Appendix~\ref{appdx:acg}, we are now ready to prove Proposition~\ref{prop:inner_to_outer}, which connects the inner ACG subroutine with the outer termination condition.

\vspace{1em}

\noindent
{\bf Proof of Proposition \ref{prop:inner_to_outer}:}
        By Lemma~\ref{lem:101}(c) with $L=L_f - \mu_f$ and $\mu= \lambda^{-1}+\mu_f$ (see \eqref{eq:setup}),  \eqref{eq:bound} implies that $ A_j\ge 5\lambda/\sigma$, with the condition on $\lambda$ ensuring that the $\log$ term is non-negative.
       Using Lemma~\ref{lem:tech}(a) and (b) with $\mu=\mu_f+\lambda^{-1}$, we have
		\begin{align*}
		\|\lam s_j\|^2 + 2\lam[\psi(y_j) - \Theta_j(x_j)]
		& \stackrel{\eqref{ineq:mj_mu},\eqref{ineq:sj}}\le \frac{9\lam^2\|y_j-x_0\|^2}{4A_j^2} + \frac{2\lam(\mu_f+\lam^{-1})}{A_j(\mu_f+\lam^{-1})-2}\|y_j-x_0\|^2 \\
		\le\frac{9\lam^2\|y_j-x_0\|^2}{4A_j^2} + \frac{2}{A_j\lam^{-1}-2}\|y_j-x_0\|^2 &\le \left( \frac{\sigma^2}{10}+\frac{2\sigma}{3}\right) \|y_j-x_0\|^2\leq \sigma\|y_j-x_0\|^2,
		\end{align*}	
        where the third inequality follows from the facts that $A_j \ge 5\lam/\sigma$ and $\sigma \in (0,1)$.
    \QEDA

\section{Deferred Proofs for I-ALM and I-FALM}\label{appdx:deferred_alm}
\subsection{Deferred Proofs for I-ALM}

\noindent
\textbf{Proof of Proposition~\ref{prop:alm_inner_complexity}}
    Set $L$ and $\mu$ as in~\eqref{def:setup_alm} define $\Phi(\cdot)=\L_\rho(\cdot,\lambda_k)$, $\gamma=\varepsilon_k/(4D^2)$, $\eta=(2L+\mu)^{-1}$, and $\bar{x}=x_k$. Using Lemma~\ref{lem:grad_map_merged}(c) and noting $\phi_\gamma(\cdot)=\psi(\cdot)$ in light of \eqref{def:setup_alm}, requiring $\|\mathcal{G}_{\psi}^{(2L+\mu)^{-1}}(\tx_k)\|\leq \varepsilon_k/4D$ guarantees that $\|\mathcal{G}_{\L_\rho(\cdot,\lambda_k)}^{(2L+\mu)^{-1}}(\tx_k)\|\leq \varepsilon_k/2D$ as in Step~\textbf{1}.
    Applying Lemma~\ref{lem:acg_convergence} to $\psi$ as in \eqref{def:setup_alm} together with Lemma~\ref{lem:101}(c), each call to Algorithm~\ref{alg:ACG} in Step \textbf{1} takes
    \[
        \tO\left(1+\frac{\sqrt{L}}{\sqrt{\mu}}\right)=\tO\left(1+\frac{D(\sqrt{L_f}+\sqrt{\rho}\|A\|)}{\sqrt{\varepsilon_k}}\right)
    \]
    ACG iterations to guarantee $\|\mathcal{G}_{\psi}^{(2L+\mu)^{-1}}(\tx_k)\|\leq \varepsilon_k/4D$. The claim follows by taking the uniform lower bound $\varepsilon_k\geq \sigma\rho\varepsilon^2/2$.
\QEDA

\vspace{1em}

\noindent
\textbf{Proof of Lemma~\ref{lem:gamma_dual_props_alm}}
    a) The claim immediately follows from the update~\eqref{def:lambda_alm} and the definition of $\L(x,\lam)$ in~\eqref{eq:strong_duality}.

    b) The inequality directly follows by the definition of the dual $d(\nu)$ in~\eqref{eq:strong_duality}, which implies $-d(\nu)\geq-\L(x,\nu)$ for all $x\in\R^n$.
\QEDA

\vspace{1em}

\noindent
\textbf{Proof of Corollary~\ref{cor:alm_complexity_pd}}
    Noting that $\delta^\lora_k=\varepsilon_0\alpha^k$ (see \eqref{def:alm_corresp}) is summable with 
    \[
    \sum_{i=0}^\infty\delta^\lora_i=\sum_{i=0}^\infty\varepsilon_0\alpha^i=\frac{\varepsilon_0}{1-\alpha}.
    \]
    Hence, applying Lemma~\ref{lem:single_step_lora} to $\Phi(\cdot)=-d(\cdot)$ with $R^\lora_0=R_\Lambda$ and $C_\delta=\varepsilon_0/(1-\alpha)$, we have
    \begin{equation}
        \|\lambda_{k}-\lambda_*\|\stackrel{\eqref{def:alm_corresp}}=\|x^\lora_{k}-x_*\|\stackrel{\eqref{def:alm_corresp},\eqref{ineq:lora_xkdist}}\leq R_\Lambda + \sqrt{\frac{2\rho\varepsilon_0}{1-\alpha}}
        =R_\Lambda + \sqrt{\frac{2}{1-\alpha}}\leq R_\Lambda + D\sqrt{\frac{2}{1-\alpha}},\label{ineq:lora_k_bound_al}
    \end{equation}
    where the second equality follows by the choice $\rho=\varepsilon^{-1}=\varepsilon_0^{-1}$ in Theorem~\ref{thm:al_baseline_complexity} and the last inequality is due to Assumption~\ref{assmp:constrained}(d).
    It follows from the triangle inequality and $\lambda_0=0$ that
    \begin{equation}
        \|\lambda_k\|\leq \|\lambda_k-\lambda_*\|+\|\lambda_*\|= \|\lambda_k-\lambda_*\|+R_\Lambda
        \stackrel{\eqref{ineq:lora_k_bound_al}}\leq 2R_\Lambda + \zeta D,\label{ineq:lora_k_norm_bound_al}
    \end{equation}
    where $\zeta = \sqrt{2/(1-\alpha)}$.
    Suppose that $(x_k,\lambda_k)$ is an $\varepsilon/[2R_\Lambda + (1+\zeta) D]$-solution to~\eqref{eq:ProbIntro_LC}, then by Lemma~\ref{lem:pd_gap}, we have
    \begin{align*}
        |\phi(x_k)-\hat\phi_*|\stackrel{\eqref{ineq:primal_gap_ul}}\leq \max\{R_\Lambda,\|\lambda_k\|+D\}\frac{\varepsilon}{2R\lambda+(1+\zeta)D}
        \stackrel{\eqref{ineq:lora_k_norm_bound_al}}\leq \frac{\varepsilon}{2R_\Lambda+(1+\zeta) D}(2R_\Lambda+(1+\zeta)D) = \varepsilon,
    \end{align*}
    where the second inequality follows from the fact that $R_\Lambda\leq 2R_\Lambda+\zeta D$.
    Then, by Theorem~\ref{thm:al_baseline_complexity} with $\varepsilon$ replaced by $\varepsilon/(2R_\Lambda + (1+\zeta)D)\leq\varepsilon$, the iteration-complexity for $(x_k,\lambda_k)$ to guarantee $|\phi(x_k)-\hat\phi_*|\leq\varepsilon$ and $\|Ax_k-b\|\leq\varepsilon$ is given by~\eqref{cmplx:alm_pd_total}.
\QEDA

\subsection{Deferred Proofs for I-FALM}

\subsubsection{Proof of Lemma~\ref{lem:perturb_solution}}
By subdifferential calculus, we can show
\begin{equation*}
    v\in\partial\tilde{\L}(\cdot,\lambda)(x)=\partial\left(\L(\cdot,\lambda)+\frac{\gamma_p}{2}\|\cdot-x_0\|^2\right)(x)=\partial\L(\cdot,\lambda)(x) +\gamma_p(x-x_0),
\end{equation*} 
hence rearranging yields
\begin{equation*}
    v':=v-\gamma_p(x-x_0)\in\partial\L(\cdot,\lambda)(x).
\end{equation*} 
Then, applying the triangle inequality, Assumption~\ref{assmp:constrained}(d), and the choice $\gamma_p=\varepsilon/(2D)$ we have
\[
    \|v'\|\leq \|v\|+{\gamma_p}\|x-x_0\|\leq\frac{\varepsilon }{2}+\frac{\varepsilon D}{2D}\leq \varepsilon.
\]

\subsubsection{Proof of Lemma \ref{lem:double_perturbation_distance}}\label{proof:double_perturbation_distance} 
Let $\lambda_*\in\Lambda_*$ be the optimal multiplier achieving $R_\Lambda$. Since $-\td$ is $\gamma_d$-strongly convex with minimizer $\tilde\lambda_*$, we have
    \begin{align}
        -\td(\lam_*)+\td(\tilde\lam_*)&\geq \frac{\gamma_d}{2}\|\lambda_*-\tilde\lambda_*\|^2 = \frac{\gamma_d}{2}\|\lambda_*-\lambda_0\|^2 + \gamma_d\inner{\lambda_*-\lambda_0}{\lambda_0-\tilde\lambda_*}+ \frac{\gamma_d}{2}\|\lambda_0-\tilde\lambda_*\|^2 \nn \\
        &\geq\frac{\gamma_d}{2}R_\Lambda^2-\gamma_d R_\Lambda \dtL + \frac{\gamma_d}{2}\dtL^2,\label{ineq:dual_difference_lb}
    \end{align}
    where the second inequality follows from the Cauchy-Schwarz inequality and the facts that $R_\Lambda = \|\lam_0 - \lam_*\|$ and $R_{\tilde \Lambda} = \|\lam_0 - \tilde \lam_*\|$. 
    Define $u(\lambda)=\underset{x\in\R^n}{\argmin}\L(x,\lambda)$, then by the definitions of $d$ and $\tilde d$ in  \eqref{eq:strong_duality} and \eqref{def:pert_dual}, respectively, we have
    \begin{equation}
        \tilde{d}(\lambda)\stackrel{\eqref{def:pert_dual}}\leq \L(u(\lambda),\lambda) + \frac{\gamma_p}{2}\|u(\lambda)-x_0\|^2 - \frac{\gamma_d}{2}\|\lambda-\lambda_0\|^2\stackrel{\eqref{eq:strong_duality}}=d(\lambda) + \frac{\gamma_p}{2}\|u(\lambda)-x_0\|^2 - \frac{\gamma_d}{2}\|\lambda-\lambda_0\|^2.\label{ineq:dual_ub}
    \end{equation}
    Applying~\eqref{ineq:pert_dual_ub} with $\lambda=\lambda_*$ and~\eqref{ineq:dual_ub} with $\lambda=\tilde{\lambda}_*$, we obtain 
    \begin{align*}
        -\td(\lam_*)+\td(\tilde\lam_*)\stackrel{\eqref{ineq:pert_dual_ub},\eqref{ineq:dual_ub},}\leq&-d(\lam_*)+\frac{\gamma_d}{2}\|\lam_*-\lam_0\|^2 +d(\tilde\lam_*) +\frac{\gamma_p}{2}\|u(\tilde \lambda_*)-x_0\|^2 -\frac{\gamma_d}{2}\|\tilde\lam_*-\lam_0\|^2.
    \end{align*}
    Plugging $R_\Lambda = \|\lam_0 - \lam_*\|$ and $R_{\tilde \Lambda} = \|\lam_0 - \tilde \lam_*\|$, and using Assumption~\ref{assmp:constrained}(d), we have
    \begin{align}\label{ineq:dual_difference_ub}
        -\td(\lam_*)+\td(\tilde\lam_*)\leq&-d(\lam_*)+d(\tilde\lam_*)+\frac{\gamma_d}{2}R_{\Lambda}^2-\frac{\gamma_d}{2}\dtL ^2+\frac{\gamma_p}{2}D^2\leq \frac{\gamma_d}{2}R_{\Lambda}^2-\frac{\gamma_d}{2}\dtL ^2+\frac{\gamma_p}{2}D^2.
        \end{align}
        where the second inequality follows from $d(\tilde\lam_*)\leq d(\lam_*)$.
    Combining the lower bound~\eqref{ineq:dual_difference_lb} and upper bound~\eqref{ineq:dual_difference_ub} on $-\td(\lam_*)+\td(\tilde\lam_*)$, we obtain
     \[
    \frac{\gamma_d}{2}R_{\Lambda}^2-\frac{\gamma_d}{2}\dtL ^2+\frac{\gamma_p}{2}D^2\stackrel{\eqref{ineq:dual_difference_ub}}\geq -\td(\lam_*)+\td(\tilde\lam_*)\stackrel{\eqref{ineq:dual_difference_lb}}\geq\frac{\gamma_d}{2}R_\Lambda^2-\gamma_d R_\Lambda \dtL + \frac{\gamma_d}{2}\dtL^2,
    \]
    which, by rearranging, yields
     \begin{equation}
         \dtL ^2-\dtL R_{\Lambda}-\frac{\gamma_p}{2\gamma_d}D^2\leq 0.\label{ineq:quad_dtL_1}
     \end{equation}

    If $\gamma_p=\varepsilon/(2D)$ and $\gamma_d=C_0\varepsilon/(\dtL)$ for some constant $C_0> 0$, then $\gamma_p/\gamma_d=\dtL/(2C_0D)$. Then, we can rewrite~\eqref{ineq:quad_dtL_1} as
    \[
    \dtL ^2 - \dtL\left(R_\Lambda + \frac{D}{4C_0}\right)\leq 0.
    \]
        Clearly, $\dtL$ attains its extremal values when the LHS equals 0. Taking the nonzero solution $R_\Lambda+D/(4C_0)$ yields the claim~\eqref{ineq:pert_distance_propto}.

\subsubsection{Proof of Proposition \ref{prop:alm_inner_complexity_sc}}\label{proof:alm_inner_complexity_sc} 
Set $L$ and $\mu$ as in~\eqref{def:setup_alm_acc} and define $\Phi(\cdot)=\L_\rho(\cdot,\tnu_k)$, $\gamma=\varepsilon_k/(4D^2)$, $\eta=(2L+\mu)^{-1}$, and $\bar{x}=x_k$. Using Lemma~\ref{lem:grad_map_merged}(c) and noting $\phi_\gamma(\cdot)=\psi(\cdot)$ in light of \eqref{def:setup_alm_acc}, requiring $\|\mathcal{G}_{\psi}^{(2L+\mu)^{-1}}(\tx_k)\|\leq \varepsilon_k/4D$ guarantees that $\|\mathcal{G}_{\L_\rho(\cdot,\lambda_k)}^{(2L+\mu)^{-1}}(\tx_k)\|\leq \varepsilon_k/2D$ as in Step~\textbf{2}. Applying Lemma~\ref{lem:acg_convergence} to $\psi$ and using Lemma~\ref{lem:101}(c) with $L$ and $\mu$ as in \eqref{def:setup_alm_acc}, we show that each call to Algorithm~\ref{alg:ACG} takes
    \begin{equation*}
        \tO\left(1+\frac{\sqrt{L}}{\sqrt{\mu}}\right)=\tO\left(1+\frac{\sqrt{L_f}+\sqrt{\rho}\|A\|}{\sqrt{\gamma_p+\varepsilon_k/(4D^2) }}\right)
    \end{equation*}
    ACG iterations to guarantee $\|\mathcal{G}_{\psi}^{(2L+\mu)^{-1}}(\tx_k)\|\leq \varepsilon_k/(4D)$. The claim follows from the trivial lower bound $\gamma_p + \varepsilon_k/(4D^2)\geq \gamma_p=\varepsilon/(2D)$.

\subsubsection{Proof of Lemma~\ref{lem:gamma_bounds_dual_sc}}\label{proof:gamma_bounds_dual_sc}
    a) We first note that
    \[
    \nabla \left(-\tL(x_{k+1},\cdot)+\frac{1}{2\rho}\|\cdot-\tnu_k\|^2\right)(\lambda_{k+1})=\gamma_d\lambda_{k+1}-(Ax_{k+1}-b)+\frac{\lambda_{k+1}-\tnu_k}{\rho}\stackrel{\eqref{eq:lambda_acc}}=\gamma_d\lambda_{k+1}.
    \]
    It thus follows from the $(\rho^{-1}+\gamma_d)$-strong convexity of $-\tL(x_{k+1},\cdot)+\|\cdot-\tnu_k\|^2/(2\rho)$ that for every $\nu\in\R^m$,
    \begin{align*}
        \Gamma_{k}^{\lambda}(\nu)&=-\tL(x_{k+1},\lambda_{k+1})+\frac{1}{2\rho}\|\lambda_{k+1}-\tnu_k\|^2+\inner{\gamma_d\lambda_{k+1}}{\nu-\lambda_{k+1}}+\frac{1+\gamma_d\rho}{2\rho}\|\nu-\lambda_{k+1}\|^2\\
        &\leq -\tL(x_{k+1},\nu)+\frac{1}{2\rho}\|\nu-\tnu_k\|^2.
    \end{align*}
    Hence, this statement immediately follows from the definition of $\tilde d$ in \eqref{def:pert_dual}.
    
    b) Since $\Gamma_k^\lam(\nu)$ is a quadratic function, it is easy to verify that $\hat \lam_{k+1}$ as in \eqref{def:lambda_hat} is the solution to $\min\{\Gamma_k^\lambda(\nu):\nu\in\R^m\}$. Also, it is straightforward to verify \eqref{eqn:min_GammaL_value} by computation.

    c) The claim follows directly from~\eqref{def:nu_k}, $u_{k+1}=\rho^{-1}(\tnu_{k}-\hat\lambda_{k+1})$, and the definition of $\hat\lambda_{k+1}$ in~\eqref{def:lambda_hat}.

    d) The claim follows from the requirement $\alpha\leq (1+\sqrt{\rho\gamma_d})^{-2}$ in the initialization of Algorithm~\ref{alg:dsc_aalm} and Lemma~\ref{lem:b_seq}(d) with $C_\flora=C$ in view of the definition of $C$ in~\eqref{def:R-R} and the correspondence $\alpha_\flora=\alpha$, $\mu_\flora=\gamma_d$, and $\lambda_\flora=\rho$ in~\eqref{def:acc_alm_corresp}. 
    %

\subsubsection{Proof of Corollary~\ref{cor:aalm_complexity_pd}}\label{proof:aalm_complexity_pd}
By the parameters chosen in Theorem~\ref{thm:acc_alm_complexity_1}, Lemma~\ref{lem:gamma_bounds_dual_sc} and Proposition~\ref{prop:OR} imply that Algorithm~\ref{alg:dsc_aalm} is an instance of the FLOrA framework under the correspondence~\eqref{def:acc_alm_corresp}. Then, applying Lemma~\ref{lem:distance_bound_flora} with $\Phi(\cdot)=-\td(\cdot)$ noting that $\mathcal{R}_\flora=\mathcal{R}$ (where $\mathcal{R}$ is as in~\eqref{def:R-R}), we obtain for any $k\geq 1$,
    \begin{equation}
    \begin{split}
    \|\tnu_{k-1}-\tilde\lambda_*\|\stackrel{\eqref{def:acc_alm_corresp}}=\|\tx^\flora_{k-1}-x_*\|\stackrel{\eqref{def:acc_alm_corresp},\eqref{ineq:tx_dist_bound}}\leq \mathcal{R}.
    \end{split}\label{ineq:tnu_dist}
    \end{equation}
    Then, it follows from the triangle inequality that
    \begin{equation}
    \begin{split}
    \|\lam_k\|&\leq\|\lam_k-\tnu_{k-1}\|+\|\tnu_{k-1}-\tilde{\lambda}_*\|+\|\tilde{\lambda}_*\|\stackrel{\eqref{eq:lambda_acc}}=\rho\|Ax_k-b\|+\|\tnu_{k-1}-\tilde{\lambda}_*\|+\|\tilde{\lambda}_*\|\\
    &\stackrel{\eqref{ineq:tnu_dist}}\leq \rho\|Ax_k-b\|+\mathcal{R}+\|\tilde{\lambda}_*\|
    \leq \rho\|Ax_k-b\|+2\mathcal{R},
    \end{split}\label{ineq:lamk_bound_pd}
    \end{equation}
    where the last inequality follows from the definition of $\mathcal{R}$ in~\eqref{def:R-R} and the choice $\lambda_0=0$. 
    
    Suppose that $(x_k,\lam_k)$ is an $\varepsilon$-primal-dual solution to~\eqref{eq:ProbIntro_LC}. Then, by Lemma~\ref{lem:pd_gap} the absolute primal gap is bounded from above by
    \[
        |\phi(x_k)-\hat\phi_*|\stackrel{\eqref{ineq:primal_gap_ul}}\leq \max\{(D+\|\lam_k\|),R_\Lambda\}\varepsilon\stackrel{\eqref{ineq:lamk_bound_pd}}{\leq} (D+\rho\|Ax_k-b\|+2\mathcal{R})\varepsilon\stackrel{\eqref{def:approximate_kkt}}\leq (D+2\mathcal{R})\varepsilon + \rho\varepsilon^2\leq 2(D+\mathcal{R})\varepsilon,
    \]
    where the second inequality follows from $R_\Lambda\leq (D+\rho\|Ax_k-b\|+2\mathcal{R})$ and the final inequality follows by the condition $\rho\varepsilon =4\sigma\rho\varepsilon\leq 1$ and Assumption~\ref{assmp:constrained}(d).
    
    Therefore,~\eqref{cmplx:total_pd_alm_1} follows by substituting $\varepsilon=\varepsilon_g/(2(D+\mathcal{R}))$ into Theorem~\ref{thm:acc_alm_complexity_1}  and using the fact that $\mathcal{R}=\mathcal{O}(\hat R_\Lambda + D)$ under the parameter settings of Theorem~\ref{thm:acc_alm_complexity_1} and noting that $\|Ax_k-b\|\leq \varepsilon_g\leq \varepsilon$.

\end{document}